\documentclass[fleqn]{elsarticle}
\usepackage{amsmath,amssymb,amsthm,graphicx,epstopdf,float,multirow,booktabs,color,appendix,bm}
\usepackage[top=1.25in, bottom=1.0in, left=1.0in, right=1.0in]{geometry}

\numberwithin{equation}{section}
\newtheorem{thm}{Theorem}[section]
\newtheorem{coro}{Corollary}[section]

\newtheorem{rmk}{Remark}[section]
\newtheorem{lem}{Lemma}[section]

\newtheorem{ex}{Example}[section]
\graphicspath{{FIGS/}}

\begin{document}

\begin{frontmatter}

\title{Efficient energy-preserving exponential integrators for multi-component Hamiltonian systems}

\author{Xuelong Gu$^{a}$, Chaolong Jiang$^{b}$, Yushun Wang$^{a}$, Wenjun Cai$^{a,*}$}
\address[1]{Jiangsu Key Laboratory for NSLSCS, School of Mathematical Sciences, Nanjing Normal University}
\address[2]{School of Statistics and Mathematics, Yunnan University of Finance and Economics}

\begin{abstract}

In this paper, we develop a framework to construct energy-preserving methods for multi-component Hamiltonian systems, combining the exponential integrator and the partitioned averaged vector field method. This leads to numerical schemes with advantages of original energy conservation, long-time stability, and excellent behavior for highly oscillatory or stiff problems. Compared to the existing energy-preserving exponential integrators (EP-EI) in practical implementation, our proposed methods are much efficient which can at least be computed by subsystem instead of handling a nonlinear coupling system at a time. Moreover, for most cases, such as the Klein-Gordon-Schr\"{o}dinger equations and the Klein-Gordon-Zakharov equations considered in this paper, the computational cost can be further reduced. Specifically, one part of the derived schemes is totally explicit, and the other is linearly implicit. In addition, we present rigorous proof of conserving the original energy of Hamiltonian systems, in which an alternative technique is utilized so that no additional assumptions are required, in contrast to the proof strategies used for the existing EP-EI. Numerical experiments are provided to demonstrate the significant advantages in accuracy, computational efficiency, and the ability to capture highly oscillatory solutions.
	
\end{abstract}


\end{frontmatter}

\begin{figure}[b]
	\small \baselineskip=10pt
	\rule[2mm]{1.8cm}{0.2mm} \par
	$^{*}$Corresponding author.\\
	E-mail address: caiwenjun@njnu.edu.cn (W. Cai).
\end{figure}

\section{Introduction}
In this paper, we consider multi-component Hamiltonian systems of forms
\begin{equation}\label{eq1-1}
	\dot{y}_i = S_{i}\big(L_{i}y_i + \nabla_{y_i} U(y) \big),\quad  i=1,2\cdots,m,
\end{equation} 
where $y_i \in \mathbb{R}^{d}$, $y=(y_1,y_2,\cdots,y_m)^\top$, $U\left(y\right)$ is the potential energy. Here $S_{i}$ and $L_{i}$, $i=1,2$ are skew-symmetric and symmetric matrices, respectively. There are lots of physical models admitting the forms of \eqref{eq1-1}, such as the Kepler and H\'{e}non-Heiles systems in celestial mechanics \cite{hairer2006}, and semi-discretization of coupled PDEs like the Zakharov \cite{zakharov1} and Klein-Gordon-Zakharov \cite{bellan-06,dendy-90} equations in plasma physics, the Klein-Gordon-Schr\"{o}dinger \cite{bellan-06} and Klein-Gordon-Dirac \cite{greiner-94,Holten-91} equations in quantum mechanics and so on. The above system can be recast into a compact Hamiltonian system as
\begin{equation}\label{eq1-2}
	\dot{y}=S\nabla H(y),
\end{equation}
where $S$ is a block diagonal matrix given by $S={\rm diag}(S_1,S_2\cdots,S_m)$, and the conserved Hamiltonian energy reads
\begin{equation}
	H(y) = \dfrac{1}{2}\sum_{i=1}^{m}y_i^\top L_iy_i + U(y).
\end{equation}   

Numerical methods, especially the energy-preserving methods, for these conservative systems have attracted a significant amount of attention in recent years. According to the equivalent Hamiltonian form of \eqref{eq1-2}, traditional energy-preserving methods can be utilized straightforwardly, such as the discrete gradient (dG) methods \cite{1996Time,1988Hamiltonian}, the averaged vector field (AVF) methods \cite{quispel2008new,JCM-34-479}, and the Hamiltonian boundary value methods (HBVMs) \cite{2019-Line}. However, the resulting schemes are inevitably fully implicit and complex due to the nonlinearity and the multi-component factor. As a consequence, very few works were devoted to applying these traditional methods to the systems \eqref{eq1-1} \cite{zhang2016,fu2020structure}. Recently, linearly implicit schemes have become preferable, attributing to their computational efficiency \cite{matsuo-01-EP-linearly-implicit,dahlby-11-general-SP,PolarDG-21}. One systematical methodology to construct linearly implicit schemes for general systems is the energy quadratization technique, which includes the invariant energy quadratization (IEQ) approach \cite{yang2017efficient,yang2017numerical,gong2020} and the scalar auxiliary variable (SAV) approach \cite{shen2018scalar,shen2018convergence}. Although the energy quadratization technique is originally proposed for dissipative gradient flow models, it has also been applied to Hamiltonian systems \cite{cai2019b,cai2020a,jiang2020a,LiDongFang2020Linearly}, but no multi-component systems have been concerned so far. One main reason that restricts the further applications may be attributed to the preservation of so-called modified energy, and continuous efforts are still made to improve this shortage \cite{LagrangeMultiplierShen}. 

The above-mentioned energy-preserving methods usually utilized one kind of time integrations for every component in the Hamiltonian system \eqref{eq1-1}. Otherwise, some more efficient energy-preserving schemes can be obtained.  Taking the Klein-Gordon-Schr\"{o}dinger (KGS) equations for example, the authors in \cite{wang2018unconditional} proposed two linearly implicit schemes by utilizing different time integrations for components corresponding to the Klein-Gordon and Schr\"odinger parts, respectively. Similar ideas can be found from the (partly) linearly implicit schemes for the Klein-Gordon-Zakharov (KGZ) equations \cite{KGZ1} and the Zakharov system \cite{FiniteDifferenceZS}. Nevertheless, the construction of these energy-preserving schemes is quite technical and can not be extended to general multi-component Hamiltonian systems \eqref{eq1-1}. Recently, Cai et al. proposed a partitioned AVF (PAVF) method \cite{cai2018partitioned} which can be viewed as a combination of the AVF method and the coordinate increment dG. Due to the flexibility of the partition strategy, the resulting schemes of the PAVF method are usually simpler (e.g., partly linearly implicit) and thus more efficient than the standard AVF method while keeping the energy still being preserved. More importantly, the PAVF method is targeted explicitly for general Hamiltonian systems \eqref{eq1-1} with multi-component.  

When the Hamiltonian systems \eqref{eq1-1} have highly oscillatory solutions or are derived from semi-discrete time-dependent PDEs that often belong to very stiff problems, exponential integrators are much preferable because they permit larger time step sizes and achieve higher accuracy than nonexponential ones. As far as we know, most exponential integrators focus on the improvement of numerical stability and accuracy (see, e.g., \cite{Cox-02,du-stability,exponential-intergrators}), and very few works consider the structure-preserving properties, especially the conservation of energy. In \cite{li2016exponential}, Li and Wu proposed a general framework of constructing second-order EP-EI by combining exponential integrators and the AVF method (denoted by EAVF). This method is then extended to arbitrarily high order according to the modifying integrator theory \cite{EAVFHighOrder}. Based on the dG method, Shen and Leok also proposed EP-EI as one kind of geometric exponential integrators studied in \cite{shen2019a}. For particular quadratic invariants, combining the Runge-Kutta method, Bhatt and Moore \cite{eerksp} developed higher-order exponential integrators for linearly damped systems. However, all these EP-EI are fully implicit for general nonlinear problems, and iterations are required for every time step, making them time-consuming. To the best of our knowledge, the only linearly implicit EP-EI for Hamiltonian systems is constructed with the help of the energy quadratization approach and thus only preserves a modified energy \cite{KGESemiExpInte}.

In this paper, we propose an exponential partitioned AVF method (EPAVF) for Hamiltonian systems \eqref{eq1-1} with multi-component, which is developed by combining PAVF and EAVF and thus possesses both advantages in the computational efficiency and the ability to deal with highly oscillatory solutions or very stiff problems. More specifically, the resulting schemes of EPAVF can be solved by subsystem, and for concrete problems, it may exhibit better performance for fast calculation. Taking the KGS equations and the KGZ equations with a parameter $0<\varepsilon\ll 1$ for illustration, we present a detailed process of constructing the EPAVF schemes under periodic and homogeneous Dirichlet boundary conditions, respectively. The derived schemes are highly efficient, in which one part can be solved explicitly while the other is linearly implicit. Moreover, all schemes can well capture the highly oscillatory solutions of these two equations, and the meshing strategy requirement (or $\varepsilon$-scalability) \cite{bao2013a,bao2016,bao2017} is $\tau=\mathcal{O}(\varepsilon^2)$ with $\tau$ the time step size. Besides the computational efficiency, we also provide a new approach to prove the energy conservation of EPAVF. In contrast to the proof techniques used in \cite{li2016exponential} and \cite{shen2019a} where extra assumptions on $S_i$ and $L_i$ are required, our approach only utilize the symmetry properties of $S_i$ and $L_i$ which is essential for problems \eqref{eq1-1}. Although the proof is carried out for the multi-component Hamiltonian systems \eqref{eq1-1}, it can also be employed to re-analyze EAVF in \cite{li2016exponential} and the exponential integrators in \cite{shen2019a}.

The rest of this paper is organized as follows. In section~\ref{EPAVF formulation}, we present the EPAVF method for general multi-component Hamiltonian systems. In conjunction with its adjoint method, a second-order EPAVF method is also developed through the composition technique. In section \ref{Properties}, a rigorous proof of energy conservation of the EPAVF method is provided. Such methods are then applied to the KGS and KGZ equations under periodic and homogeneous boundary conditions, respectively in section \ref{Numerical Experiments}, where concrete schemes are constructed, and ample numerical experiments are carried out to demonstrate their superior behaviors. The last section is concerned with the conclusion.

\section{Construction of the EPAVF methods}\label{EPAVF formulation}

In this section, we present the general form of the EPAVF method for multi-component Hamiltonian systems \eqref{eq1-1} on a finite time interval $t\in(0, T]$. Let $N_{t}$ be a given positive integer and $\tau = \frac{T}{N_{t}}$ be the time step size. Then the time domain is partitioned uniformly as $\bigcup_{n = 1}^{N_{t}} I_{n}$ where $I_{n} = (t_{n-1}, t_{n}]$, $t_n=n\tau$. Applying the well-known variation-of-constants formula to each equation of \eqref{eq1-1} gives
\begin{equation}\label{variation of constant}
	y_{i}(t_{n+1}) = \exp{(V_{i})}y_{i}(t_{n}) + \tau \int_{0}^{1} \exp\big({\left(1 - \xi\right)V_{i}}\big)S_{i}\nabla_{y_{i}} U\big(y(t_{n} + \xi\tau)\big)d\xi, \ i = 1,2,\cdots,m,
\end{equation}   
where $V_{i} = \tau S_{i}L_{i}$. To complete the construction of a numerical method, the integrals in \eqref{variation of constant} have to be approximated. Various strategies have been developed such as the exponential time differencing (ETD) methods \cite{Cox-02} that have been widely applied to phase field models \cite{etdParabolic,etdNonlocalAc} and may be ingeniously designed to possess the energy stability. However, in order to achieve the energy conservation for Hamiltonian systems, the approximations should be done more delicately. An effective way is to approximate  $\nabla U\left(y\right)$ by a suitable dG (denoted by $\overline{\nabla} U\left(\widehat{y}, y\right)$), which satisfies the following two conditions 
\begin{equation*}
	\overline{\nabla}U(\widehat{y},y) = \nabla U(y) \quad\text{and} \quad \overline{\nabla}U(\widehat{y},y)^{\top}\left(\widehat{y} - y\right) = U\left(\widehat{y}\right) - U\left(y\right),
\end{equation*}
for all $\widehat{y}, y$. For the construction of EPAVF, we choose the dG as a combination of the coordinate increment dG and the AVF method, which reads
\begin{equation}\label{dg epavf}
	\overline{\nabla}U\left(\widehat{y}, y\right) = 
	\begin{pmatrix}
		\int_{0}^{1} \nabla_{y_{1}}U \left(\xi \widehat{y}_{1} + \left(1 - \xi\right)y_{1}, y_{2}, \cdots, y_{m-1}, y_{m}\right) d\xi \\ 
		\int_{0}^{1} \nabla_{y_{2}}U \left(\widehat{y}_{1}, \xi \widehat{y}_{2} + \left(1 - \xi\right) y_{2}, \cdots, y_{m-1}, y_{m}\right) d\xi \\ 
		\vdots \\ 
		\int_{0}^{1} \nabla_{y_{m-1}} U \left(\widehat{y}_{1}, \xi  \widehat{y}_{2}, \cdots, \xi \widehat{y}_{m-1} + \left( 1 - \xi \right)y_{m-1} , y_{m}\right) d\xi  \\ 
		\int_{0}^{1} \nabla_{y_{m}}U \left(\widehat{y}_{1},  \widehat{y}_{2}, \cdots, \widehat{y}_{m-1}, \xi \widehat{y}_{m} + \left(1- \xi\right) y_{m} \right) d\xi 
	\end{pmatrix}.
\end{equation}
Then, the EPAVF method is derived as
\begin{equation} \label{epavf}
	y_{i}^{n+1} = \exp\left(V_{i}\right)y_{i}^{n} + \tau\varphi\left(V_{i}\right)S_{i}\overline{\nabla}_{y_i}U\left(y^{n+1}, y^n\right),\quad i=1,2,\cdots,m,
\end{equation}
where $\overline{\nabla}_{y_i}U\left(y^{n+1}, y^n\right)$ represents  the $i$th component of $\overline{\nabla}U\left(y^{n+1}, y^n\right)$ and $\varphi(V_i)$ is defined by
\begin{equation}\label{phiVd}
	\varphi\left(V_i\right) = \int_{0}^{1} \exp\left({\left(1 - \xi\right)V_{i}}\right)d\xi.
\end{equation} 
Since the components are separately integrated with the definition of dG \eqref{dg epavf}, the resulting schemes can be naturally solved by subsystem, which makes EPAVF much more efficient than EAVF. Moreover, as we will demonstrate below, in many cases, EPAVF can easily be linearly implicit or even partly explicit so that the computational cost can be reduced significantly.

\begin{rmk}\label{rmk1}
	It should be noticed that the definition of the EPAVF method is not unique, any reordering of $y_i$ will result to a different method. That is, denoting by $\tilde{i}$ the index after a reordering of $i$, the EPAVF method now becomes
	\begin{equation*}
			y_{\tilde{{i}}}^{n+1} = \exp\left(V_{\tilde{i}}\right)y_{\tilde{i}}^{n} + \tau\varphi\left(V_{\tilde{i}}\right)S_{\tilde{i}}\overline{\nabla}_{y_{\tilde{i}}}U\left(y^{n+1}, y^n\right),\quad i =1,2,\cdots,m.
	\end{equation*}
	As a consequence, such a method provides a lot of flexibility in practical computations. Nevertheless, one can quickly check that the EPAVF method is only of the first-order accuracy by Taylor expansion, independent of the ordering. How to select the most accurate ordering relies on concrete problems and will be demonstrated in the following examples.
\end{rmk}

In practical computation, second-order schemes are more preferable. Therefore, by the composition technique, we further improve the accuracy of the EPAVF method. Denoting the one-step evolution  of the EPAVF method \eqref{epavf} as $\Phi_{\tau}$. Then its adjoint method $\Phi_{\tau}^*$ can be obtained by replacing $n$ with $n+1$ and $\tau$ with $-\tau$, which yields
\begin{equation} \label{epavf adjoint}
	y_{i}^{n+1} = \exp\left(V_{i}\right)y_{i}^{n} + \tau\varphi\left(V_{i}\right)S_{i}\overline{\nabla}^{*}_{y_{i}}U\left(y^{n+1}, y^n\right), 
\end{equation}
where the dG $\overline{\nabla}^{*}U\left(\widehat{y}, y\right)$ is defined by
\begin{equation}\label{dg epavf adjoint}
	\overline{\nabla}^{*}U\left(\widehat{y}, y\right) = 
	\begin{pmatrix}
		\int_{0}^{1} \nabla_{y_{1}}U \left(\xi \widehat{y}_{1} + \left(1 - \xi\right)y_{1}, \widehat{y}_{2}, \cdots, \widehat{y}_{m-1}, \widehat{y}_{m}\right) d\xi \\ 
		\int_{0}^{1} \nabla_{y_{2}}U \left(y_{1}, \xi \widehat{y}_{2} + \left(1 - \xi\right) y_{2}, \cdots, \widehat{y}_{m-1}, \widehat{y}_{m}\right) d\xi \\ 
		\vdots \\ 
		\int_{0}^{1} \nabla_{y_{m-1}}U \left(y_{1},   y_{2}, \cdots, \xi \widehat{y}_{m-1} + \left( 1 - \xi \right)y_{m-1} , \widehat{y}_{m}\right) d\xi  \\ 
		\int_{0}^{1} \nabla_{y_{m}}U \left(y_{1},  y_{2}, \cdots, y_{m-1}, \xi \widehat{y}_{m} + \left(1- \xi\right) y_{m} \right) d\xi 
	\end{pmatrix}.
\end{equation}
Subsequently, a second-order symmetric method $\Psi_\tau$ is constructed by the composition of $\Phi_\tau$, $\Phi_\tau^*$ as
\begin{equation}\label{composition}
	\Psi_{\tau} = \Phi^{*}_{\tau/2} \circ \Phi_{\tau/2}.
\end{equation}
For the later use, we denote the second-order scheme \eqref{composition} by EPAVF-C. Notice that higher-order EPAVF can also be obtained by an increment of the composition stages, and readers are referred to \cite{hairer2006} for details. In the current paper, we only focus on the first- and the second-order EPAVF methods.

\section{Energy conservation of the EPAVF methods}\label{Properties}

In this section, we will prove the energy conservation of the EPAVF methods. For simplification, the following discussions are carried out with respect to the case when $i=2$ in the multi-component Hamiltonian systems \eqref{eq1-1}. The obtained results can be extended to general cases straightforwardly.

For convenience of the notation, we introduce two new variables $p,q$ and rewrite the targeted system as
\begin{equation}\label{model 2 bolock}
	\begin{cases}
		\dot{p} = S_{1}\left(L_{1}p + \nabla_{p}U\left(p, q\right)\right), \\ 
		\dot{q} = S_{2}\left(L_{2}q + \nabla_{q}U\left(p, q\right)\right).
	\end{cases}
\end{equation}
Here, the conservative Hamiltonian energy becomes 
\begin{equation}\label{energy 2 block}
	H\left(p,q\right) = \dfrac{1}{2} p^{\top}L_{1}p + \dfrac{1}{2} q^{\top}L_{2}q + U\left(p,q\right).
\end{equation}
The corresponding EPAVF method \eqref{epavf} and its adjoint \eqref{epavf adjoint} have the forms
\begin{equation}\label{epavf 2 block}
	\left\lbrace\begin{aligned}
		p^{n+1} &= \exp{\left(V_{1}\right)} p^{n} + \tau\varphi\left(V_{1}\right)S_{1}\overline{\nabla}_pU(p^{n+1}, q^{n+1},p^n,q^n), \\ 
		q^{n+1} &= \exp{\left(V_{2}\right)} q^{n} + \tau\varphi\left(V_{2}\right)S_{2}\overline{\nabla}_qU(p^{n+1}, q^{n+1},p^n,q^n),
	\end{aligned}\right.
\end{equation} 
and
\begin{equation}\label{epavf 2 block adjoint}
	\left\lbrace 
	\begin{aligned}
		p^{n+1} &= \exp{(V_{1})}p^{n} + \tau \varphi(V_{1}) S_{1} \overline{\nabla}^{*}_pU(p^{n+1}, q^{n+1}, p^{n}, q^{n}), \\ 
		q^{n+1} &= \exp{\left(V_{2}\right)} q^{n} + \tau\varphi\left(V_{2}\right)S_{2}\overline{\nabla}^{*}_qU(p^{n+1}, q^{n+1},p^n,q^n),
	\end{aligned}
	\right.
\end{equation}
where 
\begin{equation*}
	\begin{aligned}
	&	\overline{\nabla}U_p(p^{n+1}, q^{n+1},p^n,q^n)=\int_{0}^{1}\nabla_{p}U\left(\xi p^{n+1} + \left(1 - \xi\right) p^{n}, q^{n}\right) d\xi,\\
		&\overline{\nabla}U_q(p^{n+1}, q^{n+1},p^n,q^n)=\int_{0}\nabla_{q}U\left(p^{n+1}, \xi q^{n+1} + \left(1 - \xi\right)q^{n}\right) d\xi, \\ 
		&\overline{\nabla}^{*}U_p(p^{n+1}, q^{n+1},p^n,q^n)=\int_{0}^{1}\nabla_{p}U\left(\xi p^{n+1} + \left(1 - \xi\right) p^{n}, q^{n+1}\right) d\xi,\\
		&\overline{\nabla}^{*}U_q(p^{n+1}, q^{n+1},p^n,q^n)=\int_{0}\nabla_{q}U\left(p^{n}, \xi q^{n+1} + \left(1 - \xi\right)q^{n}\right) d\xi.
	\end{aligned}
\end{equation*}
In the following derivations, they are abbreviated as $\overline{\nabla}_pU$, $\overline{\nabla}_qU$, $\overline{\nabla}^{*}_pU$ and $\overline{\nabla}^{*}_qU$ for short. Moreover, the second order EPAVF--C method \eqref{composition} can be obtained via the composition of  \eqref{epavf 2 block} and \eqref{epavf 2 block adjoint}.

\begin{lem}\label{conservation lemma}
	Assume that $L$ is symmetric, $S$ is skew--symmetric and let $V=\tau SL$. The following properties hold:
	\begin{itemize}
		\item[{\rm (1)}] $L \exp(V) = \exp{(-V)}^{\top} L$;
		\item[{\rm (2)}] $\exp(V)S=S\exp(-V)^\top$;
		\item[{\rm (3)}] $L \varphi(V) = \varphi{(-V)}^{\top} L$;
		\item[{\rm (4)}] $\varphi(V)S=S\varphi(-V)^\top$;
		\item[{\rm (5)}] $V\varphi(V)=\exp(V)-I$;
		\item[{\rm (6)}] $\exp{(-V)}\varphi(V) = \varphi(-V)$.
	\end{itemize}
\end{lem}
\begin{proof}
	The commutable properties (1), (2) follow from the direct Taylor expansion and a reordering of matrix multiplications as
	\[
	\begin{aligned}
		\exp(-V)^\top L&=\sum_{k=0}^\infty\frac{(-1)^k\tau^k}{k!}\big[(SL)^k\big]^\top L=\sum_{k=0}^\infty\frac{\tau^k}{k!}(LS)^kL=L\sum_{k=0}^\infty\frac{\tau^k}{k!}(SL)^k=L\exp(V),\\
		S\exp(-V)^\top &=\sum_{k=0}^\infty\frac{(-1)^k\tau^k}{k!}S\big[(SL)^k\big]^\top =\sum_{k=0}^\infty\frac{\tau^k}{k!}S(LS)^k=\sum_{k=0}^\infty\frac{\tau^k}{k!}(SL)^kS=\exp(V)S.		
	\end{aligned}
\]
Subsequently, according to the definition of $\varphi(V)$ \eqref{phiVd}, one can easily obtain that the commutable properties also hold for $\varphi(V)$, i.e., properties (3), (4). The last two identities follow from a direct calculation, and we omit the trivial proof here.

\end{proof}
\begin{thm}
	The EPAVF method \eqref{epavf 2 block} and its adjoint \eqref{epavf 2 block adjoint} preserve a discrete Hamiltonian energy, i.e.,
	\begin{equation}\label{energy}
		H\left(p^{n+1}, q^{n+1}\right) = H\left(p^{n}, q^{n}\right).
	\end{equation} 
\end{thm}
\begin{proof}
	Since the proofs of energy conservation are quite similar for EPAVF and its adjoint, we only present that for EPAVF \eqref{epavf 2 block}. By direct calculation and recalling $V_i=\tau S_iL_i$, $i=1,2$, we have
\[
	\begin{aligned}
		&\frac{1}{2}(p^{n+1})^\top L_1p^{n+1}-\frac{1}{2}(p^{n})^\top L_1p^{n}\\
		&=\frac{1}{2}\Big(\exp(V_1)p^n+\tau\varphi(V_1)S_1\overline{\nabla}_pU\Big)^\top L_1p^{n+1}-\frac{1}{2}(p^{n})^\top L_1p^{n}\\
		&=\frac{1}{2}(p^n)^\top\exp(V_1)^\top L_1\Big(\exp(V_1)p^n+\tau\varphi(V_1)S\overline{\nabla}_pU\Big)-\frac{1}{2}\tau\overline{\nabla}_pU^\top S\varphi(V_1)^\top L_1p^{n+1}-\frac{1}{2}(p^{n})^\top L_1p^{n}\\
		&=\frac{1}{2}\tau(p^{n})^\top \exp(V_1)^\top L_1\varphi(V_1)S_1\overline{\nabla}_pU+\frac{1}{2}\tau\overline{\nabla}_pU^\top \big(L_1\varphi(V_1) S_1\big)^\top p^{n+1}\quad \mbox{(Property (1))}\\
		&=\frac{1}{2}\tau(p^{n})^\top \exp(V_1)^\top \varphi(-V_1)^\top L_1S_1\overline{\nabla}_pU-\frac{1}{2}\overline{\nabla}_pU^\top V_1\varphi(-V_1)p^{n+1}\quad \mbox{(Property (3))}\\
		&=-\frac{1}{2}(p^{n})^\top \big(V_1\varphi(V_1)\big)^\top\overline{\nabla}_pU-\frac{1}{2}\overline{\nabla}_pU^\top V_1\varphi(-V_1)p^{n+1}\quad \mbox{(Property (6))}\\
		&=-\frac{1}{2}\overline{\nabla}_pU^\top\big(\exp(V_1)-I\big)p^{n}+\frac{1}{2}\overline{\nabla}_pU^\top \big(\exp(-V_1)-I\big) p^{n+1}\quad \mbox{(Property (5))}\\
		&=-\frac{1}{2}\overline{\nabla}_pU^\top\big(\exp(V_1)-2I+I\big)p^{n}+\frac{1}{2}\overline{\nabla}_pU^\top \big(\exp(-V_1)-2I+I\big) p^{n+1}\\
		&=-\overline{\nabla}_pU^\top (p^{n+1}-p^n)-\frac{1}{2}\overline{\nabla}_pU^\top\big(\exp(V_1)+I\big)p^{n}+\frac{1}{2}\overline{\nabla}_pU^\top \big(\exp(-V_1)+I\big) p^{n+1}\\
		&=-U(p^{n+1}, q^{n}) + U(p^{n}, q^n)-\frac{1}{2}\overline{\nabla}_pU^\top\big(\exp(V_1)+I\big)p^{n}+\frac{1}{2}\overline{\nabla}_pU^\top \big(\exp(-V_1)+I\big) p^{n+1}.
	\end{aligned}
\]
Further inserting the expression of $p^{n+1}$, we  obtain
\[\begin{aligned}		
	&-\frac{1}{2}\overline{\nabla}_pU^\top\big(\exp(V_1)+I\big)p^{n}+\frac{1}{2}\overline{\nabla}_pU^\top \big(\exp(-V_1)+I\big) p^{n+1}\\
		&=-\frac{1}{2}\overline{\nabla}_pU^\top\big(\exp(V_1)+I\big)p^{n}+\frac{1}{2}\overline{\nabla}_pU^\top \big(\exp(-V_1)+I\big) \big(\exp(V_1)p^n+\tau\varphi(V_1)S_1\overline{\nabla}_pU\big)\\
		&=\frac{1}{2}\tau\overline{\nabla}_pU^\top \big(\exp(-V_1)\varphi(V_1)S_1+\varphi(V_1)S_1\big)\overline{\nabla}_pU\\
		&=\frac{1}{2}\tau\overline{\nabla}_pU^\top \big(\varphi(-V_1)S_1+\varphi(V_1)S_1\big)\overline{\nabla}_pU\quad\mbox{(Property (6))}\\
		&=\frac{1}{2}\tau\overline{\nabla}_pU^\top \big(S_1\varphi(V_1)^\top+\varphi(V_1)S_1\big)\overline{\nabla}_pU\quad\mbox{(Property (4))}\\
		&=0,
	\end{aligned}
	\]
	where the last equality is obtained by the skew-symmetry of $S\varphi(V)^\top+\varphi(V)S$. Therefore, we have 
		\begin{equation}\label{eneprf1}
		\dfrac{1}{2}(q^{n+1})^{\top}L_1q^{n+1} - \dfrac{1}{2}(q^{n})^{\top}L_1q^{n} = -U(p^{n+1}, q^{n}) + U(p^{n}, q^n).
	\end{equation}
	Similarly, we can derive
		\begin{equation}\label{eneprf2}
		\dfrac{1}{2}(q^{n+1})^{\top}L_2q^{n+1} - \dfrac{1}{2}(q^{n})^{\top}L_2q^{n} = -U(p^{n+1}, q^{n+1}) + U(p^{n+1}, q^n).
	\end{equation}
	Summing \eqref{eneprf1} and \eqref{eneprf2} together completes the proof. 
\end{proof}

\begin{rmk}
	When $i=1$, the Hamiltonian system \eqref{eq1-1} reduces to
	\begin{equation}\label{single}
	\dot{y}=S(Ly+\nabla U(y)),
	\end{equation}
	and EPAVF becomes EAVF proposed in \cite{li2016exponential}. Subsequently, the above proof technique can naturally be employed to prove the energy conservation of EAVF. Although the authors in \cite{li2016exponential} have given an alternative proof, our method is much simpler and does not require the discussion by situations whether $L$ is singular or not, which was made in the proof procedures in \cite{li2016exponential}.

\end{rmk}

\begin{rmk}
		Generally, the resulting numerical schemes of EPAVF at least can be solved by subsystem. In many specific cases, the efficiency of EPAVF can be further improved if the Hamiltonian functions or functionals have some specific forms. In the circumstances such as the coupled Schr\"odinger-KdV equations \cite{CAI2018200} and the Schr\"odinger-Boussinesq equations \cite{LIAO201893}, the nonlinear potential takes form of $U(p, q) = F(p) + pq^2$, where $F(p)$ is a nonlinear function. When applying EPAVF, the variable $q^{n+1}$ can be updated by a linear solver. Moreover, for the KGS and the KGZ equations, whose nonlinear potentials are at most quadratic with respect to each components, the resulting scheme of EPAVF can be linearly implicit and even partly explicit. Such an argument will be shown in the next section.
\end{rmk}

\begin{rmk}
	In \cite{shen2019a}, the authors also studied the exponential integrators for the single Hamiltonian system \eqref{single} and proved the energy conservation. However, the proof was made under the assumption that $S$ and $L$ commute. For the nonlinear Schr\"odinger equation and KdV equation considered in \cite{shen2019a}, this assumption does satisfy. But for a wide class of Hamiltonian systems (e.g., the KGS and KGZ equations), this assumption is not essential. In contrast, our proof approach has nothing additional assumptions on $S$ and $L$, as long as the model problem can be written into the form of \eqref{eq1-1}.
\end{rmk}

\begin{coro}
	The EPAVF--C method \eqref{composition} also possesses the discrete energy conservation law \eqref{energy}.
\end{coro}

\begin{proof}
	Since EPAVF and its adjoint both possess the same discrete energy, their composition is obviously energy conservation.
\end{proof}

\section{Numerical examples}\label{Numerical Experiments}
Taking the KGS and KGZ equations for model equations, in this section, we demonstrate the detailed derivation of the EPAVF methods, and the computational efficiency will be revealed from the resulting schemes.

\subsection{Klein-Gordon-Schr\"odinger equations}  
Consider the following nonlinear KGS equations
\begin{equation}\label{KGS origin}
	\begin{cases}
		i\psi_t\left(\mathbf{x}, t\right) + \beta\Delta\psi\left(\mathbf{x}, t\right) + u\left(\mathbf{x}, t\right)\psi\left(\mathbf{x}, t\right) = 0, \\ 
		\varepsilon^{2}u_{tt}\left(\mathbf{x}, t\right) - \Delta u\left(\mathbf{x}, t\right) + \dfrac{1}{\varepsilon^{2}}u\left(\mathbf{x}, t\right) - |\psi\left(\mathbf{x}, t\right)|^{2} = 0, \\
		\psi(\textbf{x}, 0) = \psi_{0}(\textbf{x}), \ u(\textbf{x}, 0) = u_{0}(\textbf{\textbf{x}}), \ \partial_{t}u(\textbf{x}, 0) = \dfrac{1}{\varepsilon^{2}}u_{1}(\textbf{x}),
	\end{cases}
\end{equation}
with periodic boundary conditions for both $\psi$ and $u$, where $\psi$ represents a complex-valued scalar nucleons field, $u$ is a real-valued scalar meson field, and $0<\varepsilon\leq 1$ is a dimensionless parameter inversely proportional to the speed of light. The KGS model describes a system of conserved scalar nucleons interacting with the neutral scalar mesons coupled through the Yukawa interactions \cite{makhankov1978dynamics}. By introducing $\psi = q + pi$ (both $q, p$ are real-valued variables) and an intermediate variable $v = u_t$, the KGS equations (\ref{KGS origin}) can be reformed into a first-order system
\begin{equation}\label{KGS 1st order}
	\begin{cases}
		q_t + \beta \Delta p+ p u = 0,  \\
		p_t - \beta \Delta q - q u = 0,  \\
		u_t = v,\\
		\varepsilon^{2}v_t - \Delta u  + \dfrac{1}{\varepsilon^{2}} u - q^{2} - p^{2} = 0.  \\
	\end{cases}
\end{equation}
Let $z = \left( q, p, u, v\right)^{\top}$. We can rewrite the above equations \eqref{KGS 1st order} into a compact infinite-dimensional Hamiltonian system as
\begin{equation*}
	z_t = \mathcal{D} \dfrac{\delta \mathcal{H}}{\delta z}, \quad \mbox{with}\quad
	\mathcal{D}=
	\begin{pmatrix}
		0 & \frac{1}{2} & 0 & 0 \\ 
		-\frac{1}{2} & 0 & 0 & 0 \\ 
		0 & 0 & 0 & \frac{1}{\varepsilon^{2}} \\ 
		0 & 0 & -\frac{1}{\varepsilon^{2}} & 0
	\end{pmatrix},
\end{equation*}
and the conservative Hamiltonian functional reads
\begin{equation*}
	\mathcal{H}(t) = \int_{\Omega} \frac{1}{2}\left(\frac{1}{\varepsilon^{2}} u^{2} + \varepsilon^{2} v^{2} + |\nabla u|^{2}\right) + \beta\left(|\nabla q|^{2} + |\nabla p|^{2}\right) - \left(q^{2} + p^{2}\right)u \ d\textbf{x}.
\end{equation*}

\subsubsection{Spatial discretization}\label{Fourier discrete}
Since we consider the KGS equations under periodic boundary conditions, the spatial discretization is done by the Fourier pseudospectral method. For simplification, we take the one-dimensional case to illustrate the semi-discretization briefly, and similar procedures can be generalized to higher dimensions. We leave a brief derivation of EPAVF for 2D KGS equations in \ref{appendix}.

Let the computation domain $\Omega = \left[a,b\right]$ and $N$ be a given positive even integer. Then, the spatial step is defined as $h = (b-a)/N$ and the mesh grid is denoted by $\Omega_h=\{x_j|x_j=a+jh,j=0,\cdots,N\}$. Let $V_h=\{\textbf{v}| \textbf{v}=(v_1,v_2,\cdots,v_{N})\}$ be the space of grid functions on $\Omega_h$.  Throughout this paper, we use bold letters $\textbf{u}, \textbf{v},\cdots$ to represent vectors. Applying the Fourier pseudospectral method to the KGS system \eqref{KGS 1st order}, we obtain
\begin{equation}\label{KGS-semi}
	\begin{cases}
		\textbf{q}_t + \beta \mathbb{D}_2 \textbf{p} +  \textbf{p} \odot\textbf{u} = 0,  \\
		\textbf{p}_t - \beta \mathbb{D}_2  \textbf{q }-  \textbf{q}  \odot\textbf{u} = 0,  \\
		\textbf{u}_t =  \textbf{v},\\
		\varepsilon^{2}\textbf{v}_t - \mathbb{D}_2 \textbf{u}  + \dfrac{1}{\varepsilon^{2}}  \textbf{u} -  \textbf{q}^{2} -  \textbf{p}^{2} = 0,
	\end{cases}
\end{equation}
where $\odot$ represents the point multiplication, $\mathbb{D}_2$ is the second-order spectral differential matrix  which can be diagonalized by
\begin{equation*}
	\mathbb{D}_2 = F_N^{-1}\Lambda_{per}F_N, \quad \Lambda_{per}=-\big[\mu \mbox{diag}(0,1,\cdots,N/2, -N/2+1,\cdots,-1)\big]^2,
\end{equation*}
where $\mu=2\pi/(b-a)$, $F_N$ and $F_N^{-1}$ represent the discrete Fourier transform and its inverse, respectively  \cite{shen2011spectral}. As a consequence, in the practical computation, we can utilize the fast Fourier transform (FFT)  to reduce the computational cost significantly.

After arrangement, we can rewrite the above system \eqref{KGS-semi} into the standard form of \eqref{eq1-1} as follows:
\begin{equation}\label{KGS semi}
	\left\lbrace\begin{aligned}
		&\left(\begin{array}{c}
			\textbf{q}_t\\
			\textbf{p}_t
		\end{array}\right)
		=\left(\begin{array}{cc}
			0 & \frac{1}{2}I_N\\
			-\frac{1}{2}I_N & 0
		\end{array}\right)\left[\left(\begin{array}{cc}
			-2\beta\mathbb{D}_2 & 0\\
			0 & -2\beta\mathbb{D}_2
		\end{array}\right)\left(\begin{array}{c}
			{\textbf{q}}\\
			{\textbf{p}}
		\end{array}\right)+\left(\begin{array}{c}
		-2\textbf{q}\odot\textbf{u}\\
		-2\textbf{p}\odot\textbf{u}
		\end{array}\right)\right],\\[1ex]
		&\left(\begin{array}{c}
			\textbf{u}_t\\
			\textbf{v}_t
		\end{array}\right)
		=\left(\begin{array}{cc}
			0 & \frac{1}{\varepsilon^2}I_N\\
			-\frac{1}{\varepsilon^2}I_N & 0
		\end{array}\right)\left[\left(\begin{array}{cc}
			-\mathbb{D}_2+\frac{1}{\varepsilon^2}I_N & 0\\
			0 & \varepsilon^2I_N
		\end{array}\right)\left(\begin{array}{c}
			{\textbf{u}}\\
			{\textbf{v}}
		\end{array}\right)+\left(\begin{array}{c}
			-(\textbf{p}^2+\textbf{q}^2)\\
			0
		\end{array}\right)\right],
	\end{aligned}\right.
\end{equation}
where $I_N$ is the identity matrix of dimension $N$. Notice that under periodic boundary conditions, $\mathbb{D}_2$ is symmetric thus, the semi-discrete system \eqref{KGS semi} satisfies the energy conservation in a discrete sense, i.e.,
\begin{equation}\label{kgs-semi-ene}
	\frac{dH}{dt}=0,\ H= \dfrac{1}{2} \left( \frac{1}{\varepsilon^{2}} \|\mathbf{u}\|_{h, per}^{2} +  \varepsilon^{2}\|\mathbf{v}\|_{h, per}^{2} + |\mathbf{u}|_{1, per}^{2} \right) + \beta \left(|\mathbf{q}|_{1, per}^{2} + |\mathbf{p}|_{1, per}^{2}\right) - \left(\mathbf{q}^{2} + \mathbf{p}^{2},\mathbf{u}\right)_{h,per},
\end{equation}
where the discrete inner product and the corresponding norms are defined as
\begin{equation*}
	\left(\mathbf{u}, \mathbf{v}\right)_{h,per} = h\sum\limits_{l = 0}^{N-1} \mathbf{u}_{l}\overline{\mathbf{v}}_{l}, \quad \Vert \mathbf{u} \Vert_{h,per} = \left(\mathbf{u}, \mathbf{u}\right)_{h,per} ^{1/2}, \quad |\mathbf{u}|_{1,per} = \left(-\mathbb{D}_2\mathbf{u}, \mathbf{u}\right)_{h,per}^{1/2},
\end{equation*}
for any grid functions $\textbf{u}, \textbf{v}$. Here $\overline{\mathbf{v}}$ represents the complex conjugate of $\mathbf{v}$ and the subscript `\textit{per}' is used for the periodic boundary condition.

\subsubsection{Derivation of the EPAVF schemes}

To apply the EPAVF methods, we need to calculate the matrix exponentials $\exp(V_i)$ and $\varphi(V_i)$, $i=1,2$. For the semi-discretization  \eqref{KGS semi}, we obtain
\begin{equation}\label{expV}
	\exp(V_1) = 
	\left(\begin{array}{cc}
		\cos(\tau \beta \mathbb{D}_2) & -\sin(\tau \beta \mathbb{D}_2)\\ 
		\sin(\tau \beta \mathbb{D}_2) & \cos(\tau \beta \mathbb{D}_2)
	\end{array}\right), \quad
	\exp(V_2)= 
	\left(\begin{array}{cc}
		\cos(\tau \widetilde{\mathbb{D}}_2)& \frac{\sin( \tau \widetilde{\mathbb{D}}_2) }{ \widetilde{\mathbb{D}}_2}	\\ 
		-\widetilde{\mathbb{D}}_2\sin( \tau \widetilde{\mathbb{D}}_2) & \cos(\tau \widetilde{\mathbb{D}}_2)
	\end{array} \right),
\end{equation}
and
\begin{equation}\label{phiV}
	\varphi(V_1) = 
	\left(\begin{array}{cc}
		\frac{\sin(\tau \beta \mathbb{D}_2)}{\tau\beta \mathbb{D}_2} & \frac{\cos(\tau \beta \mathbb{D}_2)-I_{N}}{\tau\beta \mathbb{D}_2} \\ [1ex]
		\frac{I_{N} - \cos(\tau \beta \mathbb{D}_2)}{\tau\beta \mathbb{D}_2} & \frac{\sin(\tau \beta \mathbb{D}_2)}{\tau\beta \mathbb{D}_2}
	\end{array}\right), \ 
	\varphi(V_2) = 
	\left(\begin{array}{cc}
		\frac{\sin( \tau \widetilde{\mathbb{D}}_2 ) } {\tau\widetilde{\mathbb{D}}_2} & \frac{   I_{N} - \cos(\tau \widetilde{\mathbb{D}}_2) }{\tau\widetilde{\mathbb{D}}_2^{2}}  \\[1ex]
		\frac{\cos( \tau \widetilde{\mathbb{D}}_2 ) - I_{N}}{\tau} & \frac{\sin( \tau \widetilde{\mathbb{D}}_2 )  }{\tau\widetilde{\mathbb{D}}_2}
	\end{array}\right),
\end{equation}
where $\widetilde{\mathbb{D}}_2 = \frac{1}{\varepsilon^{2}} \left(I_{N} - \varepsilon^{2}\mathbb{D}_2\right)^{1/2}$. For the simplification of notations, we further denote
\begin{equation*}
	\exp(V_i)= 
	\left(\begin{array}{cc}
		\exp_{11}^i & \exp_{12}^i\\
		\exp^i_{21} & \exp_{22}^i
	\end{array}\right), \ 
	\varphi(V_i) = 
	\left(\begin{array}{cc}
		\varphi_{11}^i & \varphi_{12}^i \\ 
		\varphi_{21}^i & \varphi_{22}^i
	\end{array}\right) , \ 
	i = 1,2. 
\end{equation*}
\begin{rmk}
	As well known, the efficiency of exponential integrators highly depends on the computation of the matrix exponentials. Fortunately, under periodic boundary conditions,  fast Fourier transformations can be utilized to significantly reduce the computational cost. Denote the eigenvalues of $\widetilde{\mathbb{D}}_2 $ by $\widetilde{{\Lambda}}_{per}$ with $\widetilde{{\Lambda}}_{per}=\frac{1}{\varepsilon^{2}} \left(I_{N}- \varepsilon^{2}\Lambda_{per}\right)^{1/2}$. Then $\exp(V_i)$ and $\varphi(V_i)$ can be calculated by 
	\begin{equation}\label{computed of exponential}
		\exp(V_i)=\left(\begin{array}{cc}
			F_N^{-1}	&  0\\
			0	& F_N^{-1}
		\end{array}
		\right)\left(\begin{array}{cc}
			\widetilde{\exp}_{11}^i & \widetilde{\exp}_{12}^i\\
			\widetilde{\exp}^i_{21} & \widetilde{\exp}_{22}^i
		\end{array}\right)\left(\begin{array}{cc}
			F_N	&  0\\
			0	& F_N
		\end{array}
		\right),
	\end{equation}
	and 
	\begin{equation}
		\varphi(V_i)=\left(\begin{array}{cc}
			F_N^{-1}	&  0\\
			0	& F_N^{-1}
		\end{array}
		\right)\left(\begin{array}{cc}
			\widetilde{\varphi}_{11}^i & \widetilde{\varphi}_{12}^i\\
			\widetilde{\varphi}^i_{21} & \widetilde{\varphi}_{22}^i
		\end{array}\right)\left(\begin{array}{cc}
			F_N	&  0\\
			0	& F_N
		\end{array}
		\right),
	\end{equation}
	where $\widetilde{\exp}^i_{jk}$ and $\widetilde{\varphi}^i_{jk}$, $i,j,k=1,2$ are obtained by replacing $\mathbb{D}_2$ and $\widetilde{\mathbb{D}}_2$ with ${\rm diag}(\Lambda_{per})$ and ${\rm diag}(\widetilde{\Lambda}_{per})$ in \eqref{expV}-\eqref{phiV}, respectively.
	
\end{rmk}

\begin{rmk}
	Notice that the first component of $\Lambda_{per}$ is zero, so additional treatment should be employed to calculate $\varphi(V_1)$ or $\widetilde{\varphi}_{ij}^1, i = 1,2$. For examples,
	\[
	\widetilde{\varphi}_{11}^1={\rm diag}\Big[1,\frac{\sin(\tau\beta\lambda_2)}{\tau\beta\lambda_2},\cdots,\frac{\sin(\tau\beta\lambda_N)}{\tau\beta\lambda_N}\Big],\quad 	\widetilde{\varphi}_{12}^1={\rm diag}\Big[0,\frac{\cos(\tau\beta\lambda_2)-1}{\tau\beta\lambda_2},\cdots,\frac{\cos(\tau\beta\lambda_N)-1}{\tau\beta\lambda_N}\Big],
	\]
	where $\lambda_k$, $k=2,\cdots,N$ represents the $k$-th component of $\Lambda_{per}$. Similar results can be obtained for $\widetilde{\varphi}_{21}^1$ and  $\widetilde{\varphi}_{22}^1$.
\end{rmk}

With the above notations,  the EPAVF scheme \eqref{epavf 2 block} for the KGS equations can be written as
\begin{equation}\label{scheme-kgs}
	\begin{cases}
		\textbf{q}^{n+1} = \exp_{11}^1\textbf{q}^{n} + \exp_{12}^1\textbf{p}^{n} - \tau\varphi_{11}^1 \left(\textbf{u}^{n}\odot \textbf{p}^{n+1/2}\right)+  \tau \varphi_{12}^1 \left(\textbf{u}^{n} \odot \textbf{q}^{n+1/2}\right) ,  \\ 
		\textbf{p}^{n+1} =  \exp_{21}^1\textbf{q}^{n} + \exp_{22}^1\textbf{p}^{n} - \tau \varphi^1_{21} \left(\textbf{u}^{n}\odot \textbf{p}^{n+1/2}\right)+  \tau \varphi^1_{22} \left(\textbf{u}^{n} \odot \textbf{q}^{n+1/2}\right) , \\
		\textbf{u}^{n+1} =\exp_{11}^2\textbf{u}^{n} + \exp_{12}^2\textbf{v}^{n} + \frac{\tau}{\varepsilon^{2}}\varphi_{12}^2 \left( \left(\textbf{q}^{n+1}\right)^2 + \left(\textbf{p}^{n+1}\right)^2 \right), \\ 
		\textbf{v}^{n+1} = \exp_{21}^2\textbf{u}^{n} + \exp_{22}^2\textbf{v}^{n} + \frac{\tau}{\varepsilon^{2}}\varphi_{22}^2 \left( \left(\textbf{q}^{n+1}\right)^2 + \left(\textbf{p}^{n+1}\right)^2 \right).
	\end{cases}
\end{equation}
In view of scheme \eqref{scheme-kgs}, the computation of $\textbf q^{n+1}, \textbf p^{n+1}$ and $\textbf u^{n+1}, \textbf v^{n+1}$ is separated. Moreover, the system of $\textbf q^{n+1}, \textbf p^{n+1}$ is linearly implicit with variable coefficients. Although linear solvers can then be utilized, we find that a fixed-point iteration will be much more efficient in the practical implementation. Once $\textbf q^{n+1}$ and $\textbf p^{n+1}$ are obtained, the system of $\textbf u^{n+1}, \textbf v^{n+1}$ is solved explicitly. For the adjoint of scheme \eqref{scheme-kgs}, we have
\begin{equation}\label{adjoint-kgs}
	\begin{cases}
		\textbf{q}^{n+1} = \exp_{11}^1\textbf{q}^{n} + \exp_{12}^1\textbf{p}^{n} - \tau\varphi_{11}^1 \left(\textbf{u}^{n+1}\odot \textbf{p}^{n+1/2}\right)+  \tau \varphi_{12}^1 \left(\textbf{u}^{n+1} \odot \textbf{q}^{n+1/2}\right) ,  \\ 
		\textbf{p}^{n+1} =  \exp_{21}^1\textbf{q}^{n} + \exp_{22}^1\textbf{p}^{n} - \tau \varphi^1_{21} \left(\textbf{u}^{n+1}\odot \textbf{p}^{n+1/2}\right)+  \tau \varphi^1_{22} \left(\textbf{u}^{n+1} \odot \textbf{q}^{n+1/2}\right) , \\
		\textbf{u}^{n+1} =\exp_{11}^2\textbf{u}^{n} + \exp_{12}^2\textbf{v}^{n} + \frac{\tau}{\varepsilon^{2}}\varphi_{12}^2 \left( \left(\textbf{q}^{n}\right)^2 + \left(\textbf{p}^{n}\right)^2 \right), \\ 
		\textbf{v}^{n+1} = \exp_{21}^2\textbf{u}^{n} + \exp_{22}^2\textbf{v}^{n} + \frac{\tau}{\varepsilon^{2}}\varphi_{22}^2 \left( \left(\textbf{q}^{n}\right)^2 + \left(\textbf{p}^{n}\right)^2 \right),
	\end{cases}
\end{equation}
and the procedure of computation is contrary to that of \eqref{scheme-kgs}. An explicit solver is firstly applied to the system of $\textbf u^{n+1}, \textbf v^{n+1}$, and then a fixed-point iteration is used to solve the linear systems of $\textbf q^{n+1}, \textbf p^{n+1}$. The second-order EPAVF-C scheme is constructed by the composition of schemes \eqref{scheme-kgs} and \eqref{adjoint-kgs}, which will still be highly efficient than existing energy-preserving schemes for the KGS equations from the following numerical experiments.

\begin{rmk}
		Although Cai constructed PAVF schemes for the KGS which is linearly implicit in \cite{cai2018partitioned}, one notice that the computational cost can be further reduced here. The variables $\mathbf{u}^{n+1}$, $\mathbf{v}^{n+1}$ are updated explicitly in \eqref{scheme-kgs}, \eqref{adjoint-kgs}.
\end{rmk}

\begin{thm}
	The EPAVF schemes, i.e., \eqref{scheme-kgs}, \eqref{adjoint-kgs} and their composition, all preserve a fully discrete energy conservation law
	\[
	H(\textbf{q}^{n+1},\textbf{p}^{n+1},\textbf{u}^{n+1},\textbf{v}^{n+1})=H(\textbf{q}^n,\textbf{p}^n,\textbf{u}^n,\textbf{v}^n),
	\]
	where the energy function $H$ is defined in \eqref{kgs-semi-ene}.
\end{thm}

\subsubsection{Numerical experiments}

\begin{ex}
	Consider the one-dimensional KGS equations on the spatial interval $\Omega= [-32, 32]$  with $\beta=1$ and the initial conditions $\psi_{0}, u_{0}, u_{1}$ given in \eqref{KGS origin} as follows
	\begin{equation*}
		\psi_{0}(x) = \frac{1+i}{2}\mbox{\rm sech}{\left(x^{2}\right)}, \quad u_{0}(x) = \frac{1}{2}\exp{\left(-x^{2}\right)}, \quad u_{1}(x) = \frac{1}{\sqrt{2}}\exp{\left(-x^{2}\right)}.
	\end{equation*} 
\end{ex}

We first verify the time accuracy of EPAVF-C \eqref{scheme-kgs}-\eqref{adjoint-kgs} for the KGS equations at $t=1$. The  errors with respect to $\psi$ and $u$ are defined by
\begin{equation*}
	\begin{aligned}
		&e_{\psi,\varepsilon}^{\tau,h}:=  \max\Big\{\| {\rm Re}\left(\psi(\cdot, 1) - \psi^n\right) \|_\infty ,   \| {\rm Im}\left(\psi(\cdot, 1) - \psi^n \right) \|_\infty \Big\} , \\ 
		&e_{u,\varepsilon}^{\tau,h}:= \| u(\cdot,1) - u^n \|_\infty,
	\end{aligned}
\end{equation*}
where ${\rm Re}(\cdot)$ and ${\rm Im}(\cdot)$ represent the real and imaginary parts of a complex variable, respectively. Since the exact solution under the above initial conditions is not available, reference solutions of $\psi(\cdot, 1)$ and $u(\cdot, 1)$ are obtained by EPAVF-C under a very fine mesh $h = 1/32$ and time step $\tau = 2.5 \times 10^{-6}$.  Tables~\ref{tab1}-\ref{tab5} list the results of EPAVF-C and other three  existing schemes, i.e., EAVF \cite{li2016exponential}, PAVF-C \cite{cai2018partitioned} and AVF \cite{mclachlan1999geometric}  with different $\varepsilon$ for comparison. A very small mesh size $h = 1/8$ is used, which guarantees that the error caused by spatial discretization can be ignored. For EPAVF-C, we present the convergence results of both $\psi$ and $u$, while for others we only show the results of $u$ as a similar result can be found for $\psi$. 

 From Tables~\ref{tab1}-\ref{tab5}, we can draw the following observations: 
 \begin{itemize}
 	\item[(i)] All the schemes exhibit a second-order convergence w.r.t different $\varepsilon$ provided the time step $\tau$ is sufficiently small, and the approximation in $u$ is more accurate than $\psi$ (cf. Tables~\ref{tab1}-\ref{tab2}). Moreover, the two exponential ones (EPAVF-C and EAVF) have smaller errors than the non-exponential ones (PAVF-C and AVF). 
 	
 	\item[(ii)] The mesh strategy (or $\varepsilon$-scalability) of EPAVF-C and EAVF is $\tau=\mathcal{O}(\varepsilon^2)$, in order to compute a ``correct'' solution (cf. upper triangle above the diagonal with values in italics of Tables~\ref{tab1}-\ref{tab3}). While the $\varepsilon$-scalability of PAVF-C and AVF is  $\tau=\mathcal{O}(\varepsilon^3)$  (cf. upper triangle above the diagonal with values in italics of Tables~\ref{tab4}-\ref{tab5}). Therefore, a much smaller time step is required to get the correct approximations and the convergence results for the non-exponential integrators.
 \end{itemize}

We also present the snapshots of numerical solutions of $\psi$ and $u$ at $t=10$ in Figure~\ref{SOLUTIONS T=10}. The wave profiles of EPAVF-C and EAVF are highly matched with the reference ones, while those of PAVF-C and AVF are clearly wrong so that the oscillatory waves cannot be captured, which demonstrates the superior advantage of exponential integrators for dealing with highly oscillatory problems.

\begin{table}[H]
        \centering
        \caption{Temporal error analysis of $\psi$ solved by EPAVF-C with different $\varepsilon$}
        \begin{tabular*}{0.9\textwidth}[h]{@{\extracolsep{\fill}}c c c c c c c c c}	
            \toprule[2pt]
            & & $\tau_0 = 0.2$ & $\tau_0/2^2$ & $\tau_0/2^4$ & $\tau_0/2^6$ & $\tau_0/2^8$ & $\tau_0/2^{10}$ \\  
            \hline
             \multirow{2}{*}{$\varepsilon_0=1$} & $e_{\psi,\varepsilon}^{\tau,h}$ & 1.047e-03 & 6.6548e-05 & 4.1574e-06 & 2.5990e-07 & 1.6314e-08 & 1.0896e-09 \\ 
            & Rate & - & 2.0064 & 2.0003 & 1.9998 & 1.9969 & 1.9521 \\ 
            \hline 
            \multirow{2}{*}{$\varepsilon_0/2$} & $e_{\psi,\varepsilon}^{\tau,h}$ & \textit{4.7450e-03} & 2.9812e-04 & 1.8634e-05 & 1.1646e-06 & 7.2797e-08 & 4.5565e-09 \\ 
            & Rate & - & 1.9962 & 1.9999 & 2.0000 & 1.9999 & 1.9989 \\ 
            \hline 
            \multirow{2}{*}{$\varepsilon_0/2^2$} & $e_{\psi,\varepsilon}^{\tau,h}$ & 1.9743e-02 & \textit{9.6187e-04} & 5.9483e-05 & 3.7153e-06 & 2.3222e-07 & 1.4535e-08 \\ 
            & Rate & - & 2.1797 & 2.0076 & 2.0005 & 2.0000 & 1.9989 \\ 
            \hline 
            \multirow{2}{*}{$\varepsilon_0/2^3$} & $e_{\psi,\varepsilon}^{\tau,h}$ & 2.3970e-01 & 6.4782e-03 & \textit{3.3094e-04} & 2.0468e-05 & 1.2785e-06 & 7.9910e-08 \\ 
            & Rate & - & 2.6408 & 2.1455 & 2.0076 & 2.0005 & 1.9999 \\ 
            \hline 
            \multirow{2}{*}{$\varepsilon_0/2^4$} & $e_{\psi,\varepsilon}^{\tau,h}$ & 5.8882e-02 & 6.4276e-02 & 8.9693e-04 & \textit{4.6231e-05} & 2.8605e-06 & 1.7863e-07 \\ 
            & Rate & - & -0.0632 & 3.0816 & 2.1390 & 2.0073 & 2.0006 \\ 
            \hline 
            \multirow{2}{*}{$\varepsilon_0/2^5$} & $e_{\psi,\varepsilon}^{\tau,h}$ & 1.0082e-02 & 8.2454e-03 & 8.9167e-03 & 1.6962e-04 & \textit{8.7355e-06} & 5.4038e-07 \\ 
            & Rate & - & 0.1450 & -0.0565 & 2.8581 & 2.1397 & 2.0074 \\ 
            \bottomrule[2pt]
       \end{tabular*}
        \label{tab1}
    \end{table}

    \begin{table}[H]
        \centering
        \caption{Temporal error analysis of $u$ solved by EPAVF-C with different $\varepsilon$}
        \begin{tabular*}{0.9\textwidth}[h]{@{\extracolsep{\fill}}c c c c c c c c c}
            \toprule[2pt]
            & & $\tau_0 = 0.2$ & $\tau_0/2^2$ & $\tau_0/2^4$ & $\tau_0/2^6$ & $\tau_0/2^8$ & $\tau_0/2^{10}$ \\  
            \hline
             \multirow{2}{*}{$\varepsilon_0=1$} & $e_{u,\varepsilon}^{\tau,h}$ & 7.7086e-04 & 2.8692e-05 & 1.7527e-06 & 1.0956e-07 & 6.8640e-09 & 4.4522e-10 \\ 
            & Rate & - & 2.3739 & 2.0165 & 1.9999 & 1.9983 & 1.9732 \\ 
            \hline 
            \multirow{2}{*}{$\varepsilon_0/2$} & $e_{u,\varepsilon}^{\tau,h}$ & \textit{1.8925e-03} & 7.9872e-05 & 4.9233e-06 & 3.0755e-07 & 1.9227e-08 & 1.2073e-09 \\ 
            & Rate & - & 2.2588 & 2.0100 & 2.0004 & 1.9998 & 1.9966 \\ 
            \hline 
            \multirow{2}{*}{$\varepsilon_0/2^2$} & $e_{u,\varepsilon}^{\tau,h}$ & 3.1996e-03 & \textit{1.3368e-04} & 8.3419e-06 & 5.2241e-07 & 3.2260e-08 & 2.0468e-09 \\ 
            & Rate & - & 2.2905 & 2.0011 & 1.9986 & 1.9998 & 1.9980 \\ 
            \hline 
            \multirow{2}{*}{$\varepsilon_0/2^3$} & $e_{u,\varepsilon}^{\tau,h}$ & 2.0177e-03 & 6.2847e-05 & \textit{3.0683e-06} & 1.9109e-07 & 1.1943e-08 & 7.4906e-10 \\ 
            & Rate & - & 2.5024 & 2.1782 & 2.0025 & 2.0000 & 1.9975 \\   
            \hline 
            \multirow{2}{*}{$\varepsilon_0/2^4$} & $e_{u,\varepsilon}^{\tau,h}$ & 3.6244e-04 & 4.2723e-04 & 3.2745e-06 & \textit{1.7197e-06} & 1.0653e-08 & 6.6693e-10 \\ 
            & Rate & - & -0.1186 & 3.5138 & 2.1255 & 2.0064 & 1.9988 \\ 
            \hline 
            \multirow{2}{*}{$\varepsilon_0/2^5$} & $e_{u,\varepsilon}^{\tau,h}$ & 5.1385e-05 & 3.2956e-05 & 3.5790e-05 & 1.8177e-07 & \textit{9.4711e-09} & 5.8809e-10 \\ 
            & Rate & - & 0.3204 & -0.0595 & 3.8107 & 2.1312 & 2.0047 \\ 
            \bottomrule[2pt]
                   \end{tabular*}
        \label{tab2}
    \end{table}

     \begin{table}[H]
        \centering
        \caption{Temporal error analysis of $u$ solved by EAVF with different $\varepsilon$}
        \begin{tabular*}{0.9\textwidth}[h]{@{\extracolsep{\fill}}c c c c c c c c c}
            \toprule[2pt]
            & & $\tau_0 = 0.2$ & $\tau_0/2^2$ & $\tau_0/2^4$ & $\tau_0/2^6$ & $\tau_0/2^8$ & $\tau_0/2^{10}$ \\  
            \hline
             \multirow{2}{*}{$\varepsilon_0=1$} & $e_{u,\varepsilon}^{\tau,h}$ & 1.7665e-03 & 1.3889e-04 & 8.8326e-06 & 5.5250e-07 & 3.4549e-08 & 2.1756e-09 \\ 
            & Rate & - & 1.8339 & 1.9880 & 1.9994 & 1.9996 & 1.9946 \\ 
            \hline 
            \multirow{2}{*}{$\varepsilon_0/2$} & $e_{u,\varepsilon}^{\tau,h}$ & \textit{4.0100e-03} & 3.0137e-04 & 1.9077e-05 & 1.1931e-06 & 7.4581e-08 & 4.6669e-09 \\ 
            & Rate & - & 1.8670 & 1.9908 & 1.9995 & 1.9999 & 1.9991 \\ 
            \hline 
            \multirow{2}{*}{$\varepsilon_0/2^2$} & $e_{u,\varepsilon}^{\tau,h}$ & 3.3702e-03 & \textit{5.1512e-04} & 3.5039e-05 & 2.2016e-06 & 1.3765e-07 & 8.6088e-09 \\ 
            & Rate & - & 1.3549 & 1.9389 & 1.9962 & 1.9997 & 1.9995 \\ 
            \hline 
            \multirow{2}{*}{$\varepsilon_0/2^3$} & $e_{u,\varepsilon}^{\tau,h}$ & 1.9394e-03 & 5.6261e-05 & \textit{1.3535e-05} & 9.1784e-07 & 5.7659e-08 & 3.6047e-09 \\ 
            & Rate & - & 2.5538 & 1.0276 & 1.9411 & 1.9963 & 1.9992 \\   
            \hline 
            \multirow{2}{*}{$\varepsilon_0/2^4$} & $e_{u,\varepsilon}^{\tau,h}$ & 3.0642e-04 & 4.1921e-04 & 3.0830e-06 & \textit{7.6699e-07} & 5.1913e-08 & 3.2621e-09 \\ 
            & Rate & - & -0.2261 & 3.5436 & 1.0035 & 1.9425 & 1.9961 \\ 
            \hline 
            \multirow{2}{*}{$\varepsilon_0/2^5$} & $e_{u,\varepsilon}^{\tau,h}$ & 3.8684e-05 & 2.3840e-05 & 3.5147e-05 & 1.7375e-07 & \textit{4.4648e-08} & 3.0290e-09 \\ 
            & Rate & - & 0.3492 & -0.2800 & 3.8301 & 0.9802 & 1.9408 \\ 
            \bottomrule[2pt]
       \end{tabular*}
        \label{tab3}
    \end{table}

    \begin{table}[H]
        \centering
        \caption{Temporal error analysis of $u$ solved by PAVF-C with different $\varepsilon$}
        \begin{tabular*}{0.9\textwidth}[h]{@{\extracolsep{\fill}}c c c c c c c c c}
            \toprule[2pt]
            & & $\tau_0 = 0.2$ & $\tau_0/2^3$ & $\tau_0/2^6$ & $\tau_0/2^9$ & $\tau_0/2^{12}$ \\  
            \hline
             \multirow{2}{*}{$\varepsilon_0=1$} & $e_{u,\varepsilon}^{\tau,h}$ & 1.3461e-03 & 2.3404e-05 & 2.6573e-07 & 5.6143e-09 & 4.2426e-08  \\ 
            & Rate & - & 1.9486 & 2.0000 & 2.0085 & -0.9726    \\ 
            \hline 
            \multirow{2}{*}{$\varepsilon_0/2$} & $e_{u,\varepsilon}^{\tau,h}$ & \textit{2.4245e-02} & 3.9034e-04 & 6.1029e-06 & 9.5350e-08 & 5.9386e-09   \\ 
            & Rate & - & 1.9856 & 2.0000 & 2.0000 & 1.3335    \\ 
            \hline 
            \multirow{2}{*}{$\varepsilon_0/2^2$} & $e_{u,\varepsilon}^{\tau,h}$ & 1.1853e-00 & \textit{1.9790e-02} & 3.0726e-04 & 4.8044e-06 & 7.6579e-08 \\ 
            & Rate & - & 1.9681 & 2.0031 & 2.0001 & 1.9900  \\ 
            \hline 
            \multirow{2}{*}{$\varepsilon_0/2^3$} & $e_{u,\varepsilon}^{\tau,h}$ & 3.1225e-01 & 9.0667e-01 & \textit{2.1088e-02} & 3.3911e-04 & 5.3010e-06  \\ 
            & Rate & - & -0.5126 & 1.8087 & 1.9862 & 1.9998  \\ 
            \hline 
            \multirow{2}{*}{$\varepsilon_0/2^4$} & $e_{u,\varepsilon}^{\tau,h}$ & 4.4645e-01 & 2.3475e-01 & 7.7685e-01 & \textit{2.7607e-02} & 4.3704e-04   \\ 
            & Rate & - & 0.3091 & -0.5755 & 1.6048 & 1.9937  \\ 
            \hline 
            \multirow{2}{*}{$\varepsilon_0/2^5$} & $e_{u,\varepsilon}^{\tau,h}$ & 2.9648e-01 & 2.2446e-01 & 9.7850e-01 & 8.8872e-01 & 2.1330e-02  \\ 
            & Rate & - & 0.1338 & -0.7080 & 0.0463 & 1.7936  \\ 
            \bottomrule[2pt] 
        \end{tabular*}
        \label{tab4}
    \end{table}

        \begin{table}[H]
        \centering
        \caption{Temporal error analysis of $u$ solved by AVF with different $\varepsilon$}
       \begin{tabular*}{0.9\textwidth}[h]{@{\extracolsep{\fill}}c c c c c c c c c}
            \toprule[2pt]
            & & $\tau_0 = 0.2$ & $\tau_0/2^3$ & $\tau_0/2^6$ & $\tau_0/2^9$ & $\tau_0/2^{12}$ \\  
            \hline
             \multirow{2}{*}{$\varepsilon_0=1$} & $e_{u,\varepsilon}^{\tau,h}$ & 5.8230e-03 & 9.0708e-05 & 1.4174e-06 & 2.2093e-08 & 8.9002e-09  \\ 
            & Rate & - & 2.0015 & 2.0000 & 2.0017 & 0.4372    \\ 
            \hline 
            \multirow{2}{*}{$\varepsilon_0/2$} & $e_{u,\varepsilon}^{\tau,h}$ & \textit{8.4507e-02} & 1.5600e-03 & 2.4437e-05 & 3.8184e-07 & 6.5865e-09   \\ 
            & Rate & - & 1.9198& 1.9988 & 2.0000 & 1.9524    \\ 
            \hline 
            \multirow{2}{*}{$\varepsilon_0/2^2$} & $e_{u,\varepsilon}^{\tau,h}$ & 2.2123e-01 & \textit{7.9914e-02} & 1.2296e-03 & 1.9205e-05 & 3.0022e-07 \\ 
            & Rate & - & 0.4897 & 2.0074 & 2.0002 & 1.9998  \\ 
            \hline 
            \multirow{2}{*}{$\varepsilon_0/2^3$} & $e_{u,\varepsilon}^{\tau,h}$ & 2.7941e-01 & 1.1040e-00 & \textit{7.6467e-02} & 1.3546e-03 & 2.1203e-05  \\ 
            & Rate & - & -0.6608 & 1.2839 & 1.9396 & 1.9992  \\ 
            \hline 
            \multirow{2}{*}{$\varepsilon_0/2^4$} & $e_{u,\varepsilon}^{\tau,h}$ & 3.4113e-01 & 1.1037e-00 & 4.2018e-02 & \textit{1.5037e-01} & 1.7472e-03   \\ 
            & Rate & - & -0.5647 & 1.5717 & -0.4421 & 1.9714  \\ 
            \hline 
            \multirow{2}{*}{$\varepsilon_0/2^5$} & $e_{u,\varepsilon}^{\tau,h}$ & 9.1835e-01 & 1.8162e-01 & 3.1667e-01 & 3.6397e-01 & 8.8520e-02  \\ 
            & Rate & - & 0.7794 & -0.2673 & -0.0670 & 0.6779  \\ 
            \bottomrule[2pt]
        \end{tabular*}
        \label{tab5}
    \end{table}

\begin{figure}[H]
	\centering
	\begin{minipage}[t]{0.24\textwidth}
		\includegraphics[width=1\linewidth]{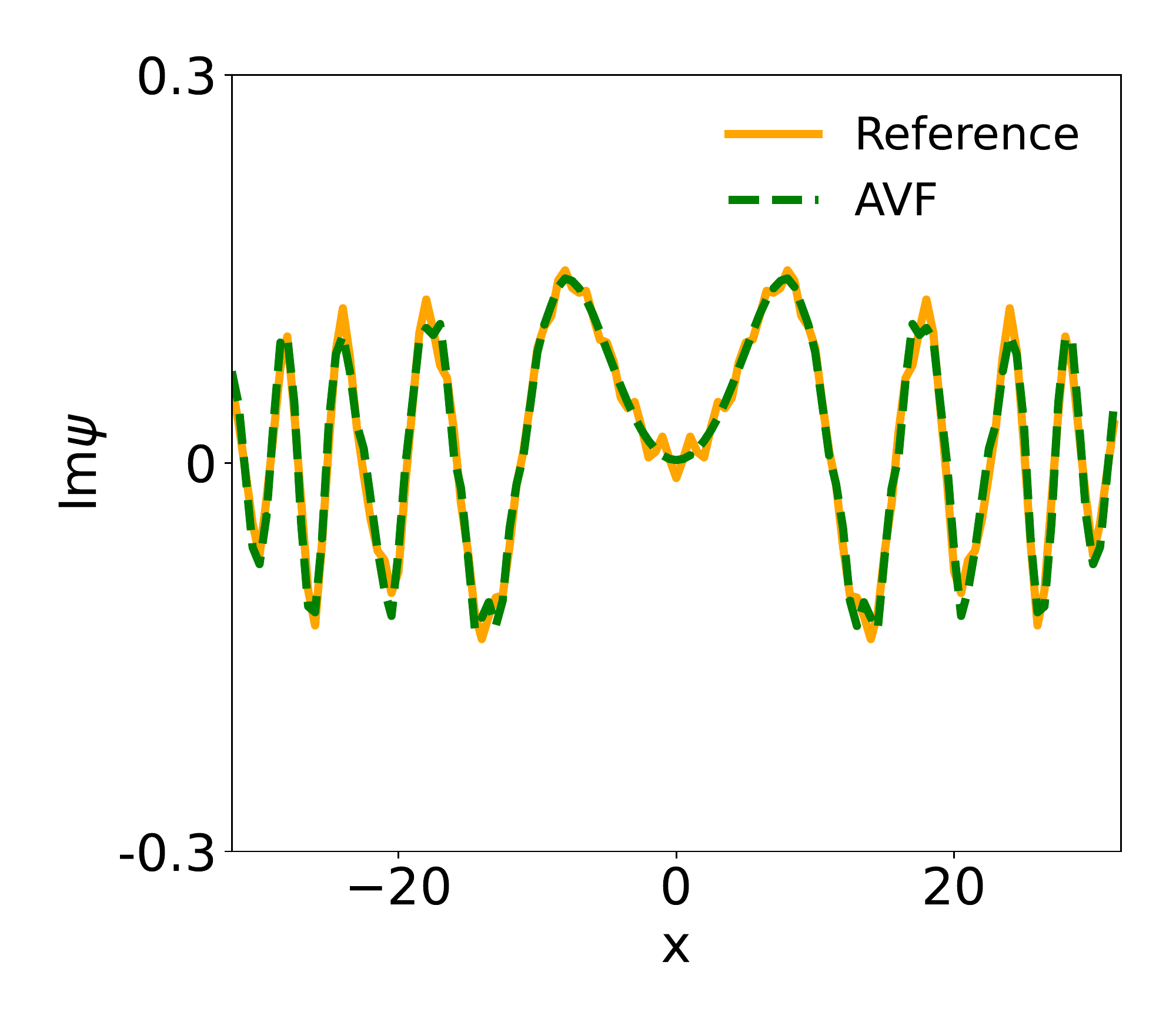}
	\end{minipage}
	\begin{minipage}[t]{0.24\textwidth}
		\includegraphics[width=1\linewidth]{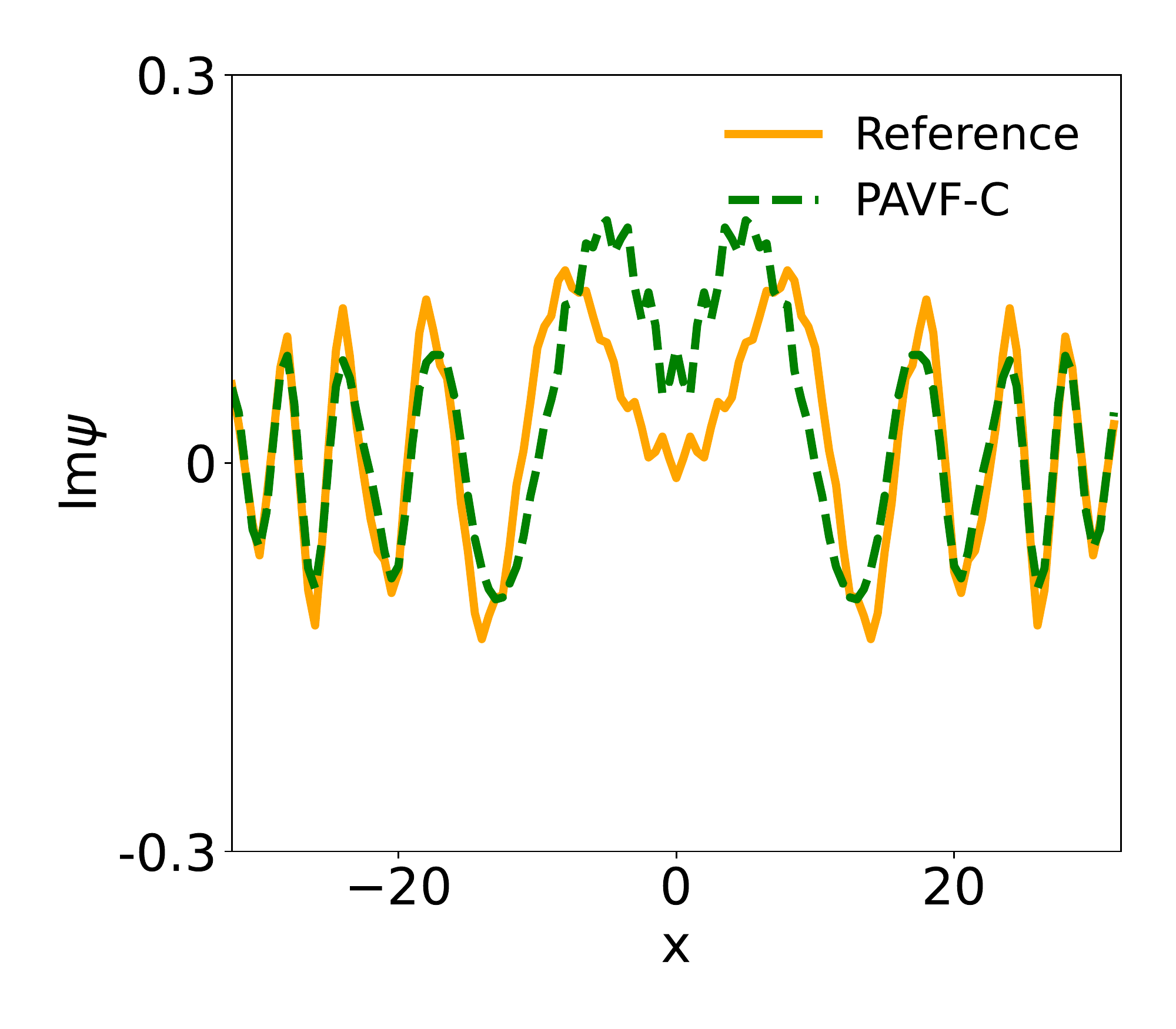}
	\end{minipage}
	\begin{minipage}[t]{0.24\textwidth}
		\includegraphics[width=1\linewidth]{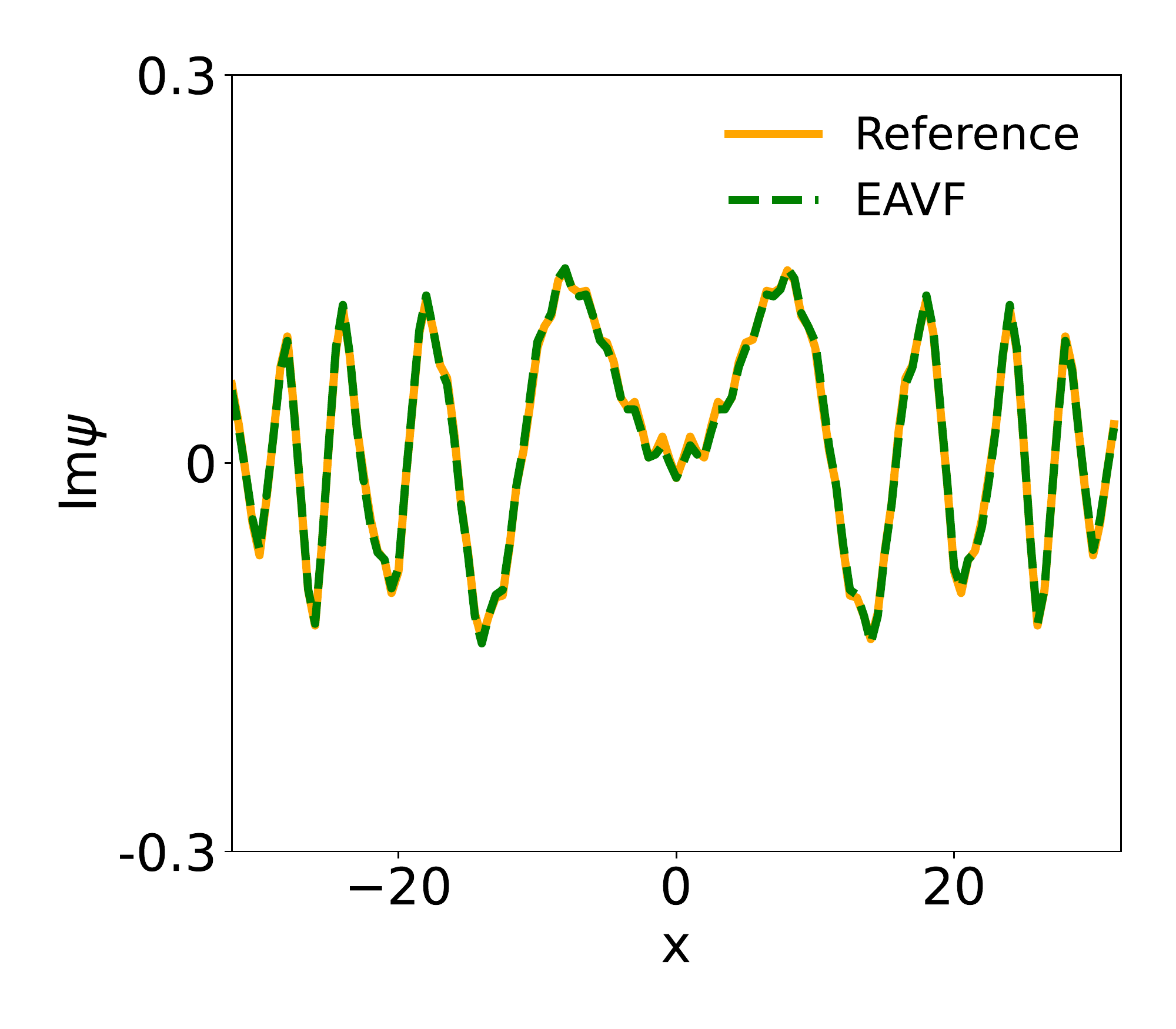}
	\end{minipage}
	\begin{minipage}[t]{0.24\textwidth}
		\includegraphics[width=1\linewidth]{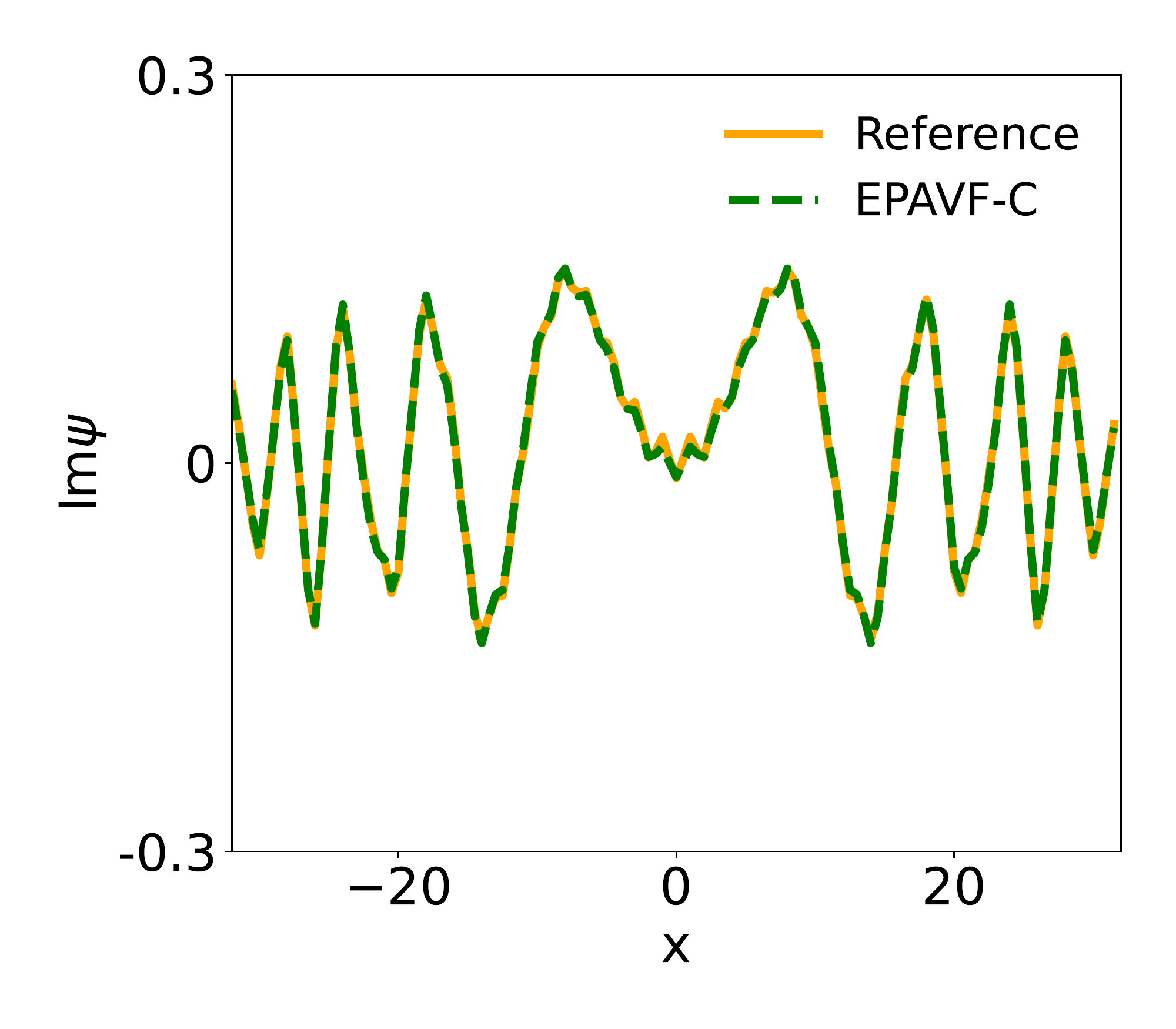}
	\end{minipage}
	\begin{minipage}[t]{0.24\textwidth}
		\includegraphics[width=1\linewidth]{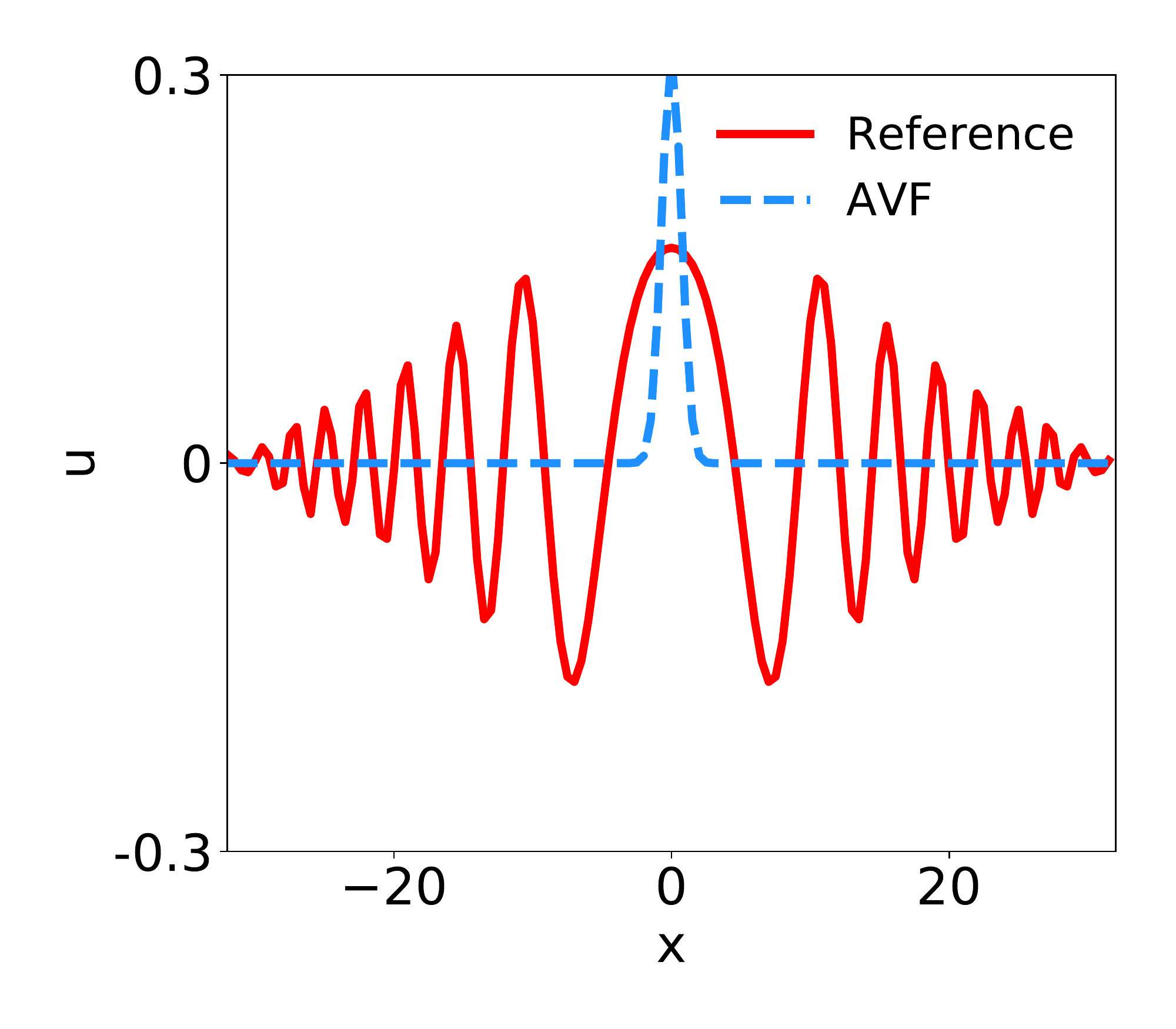}
	\end{minipage}
	\begin{minipage}[t]{0.24\textwidth}
		\includegraphics[width=1\linewidth]{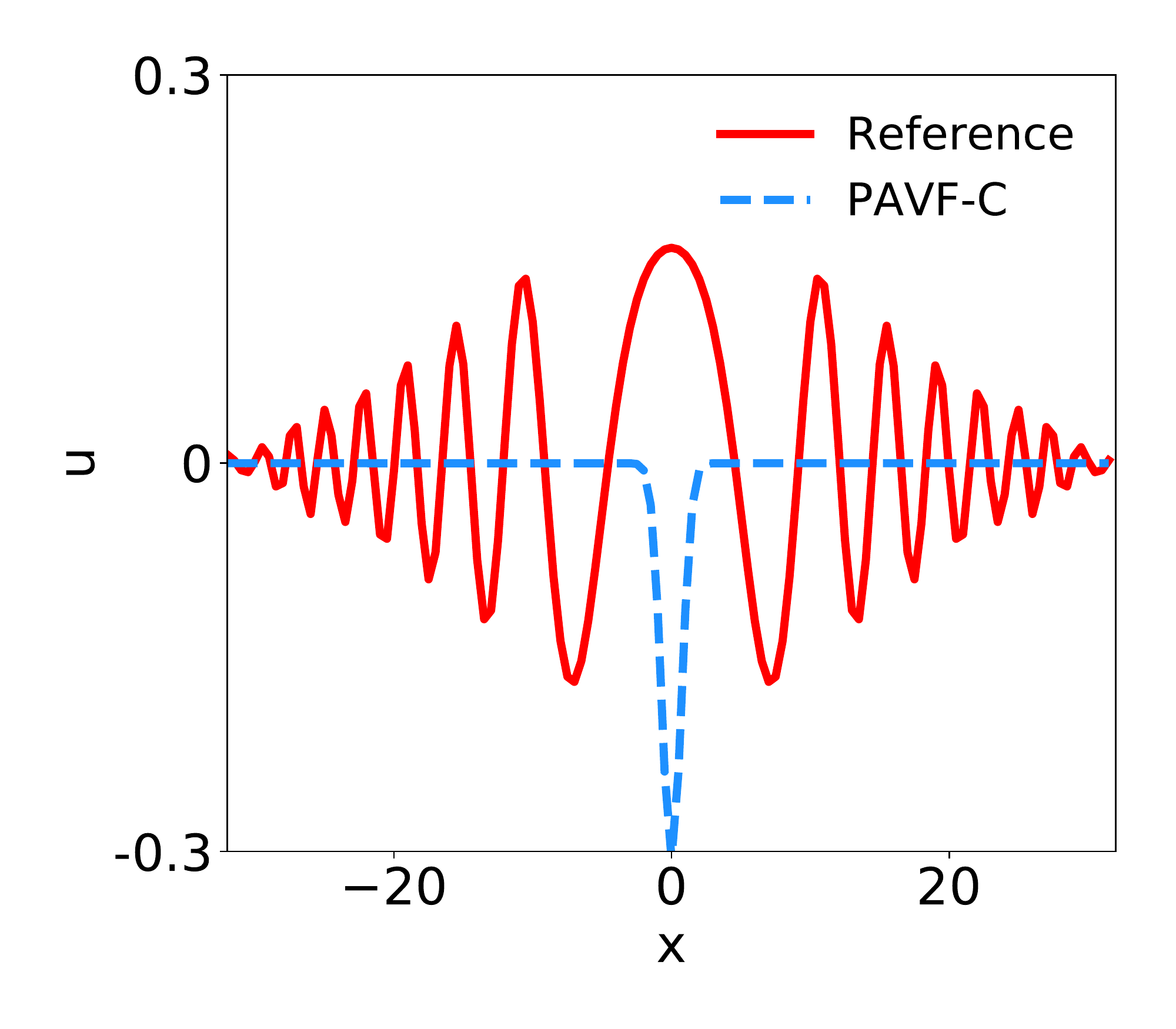}
	\end{minipage}
	\begin{minipage}[t]{0.24\textwidth}
		\includegraphics[width=1\linewidth]{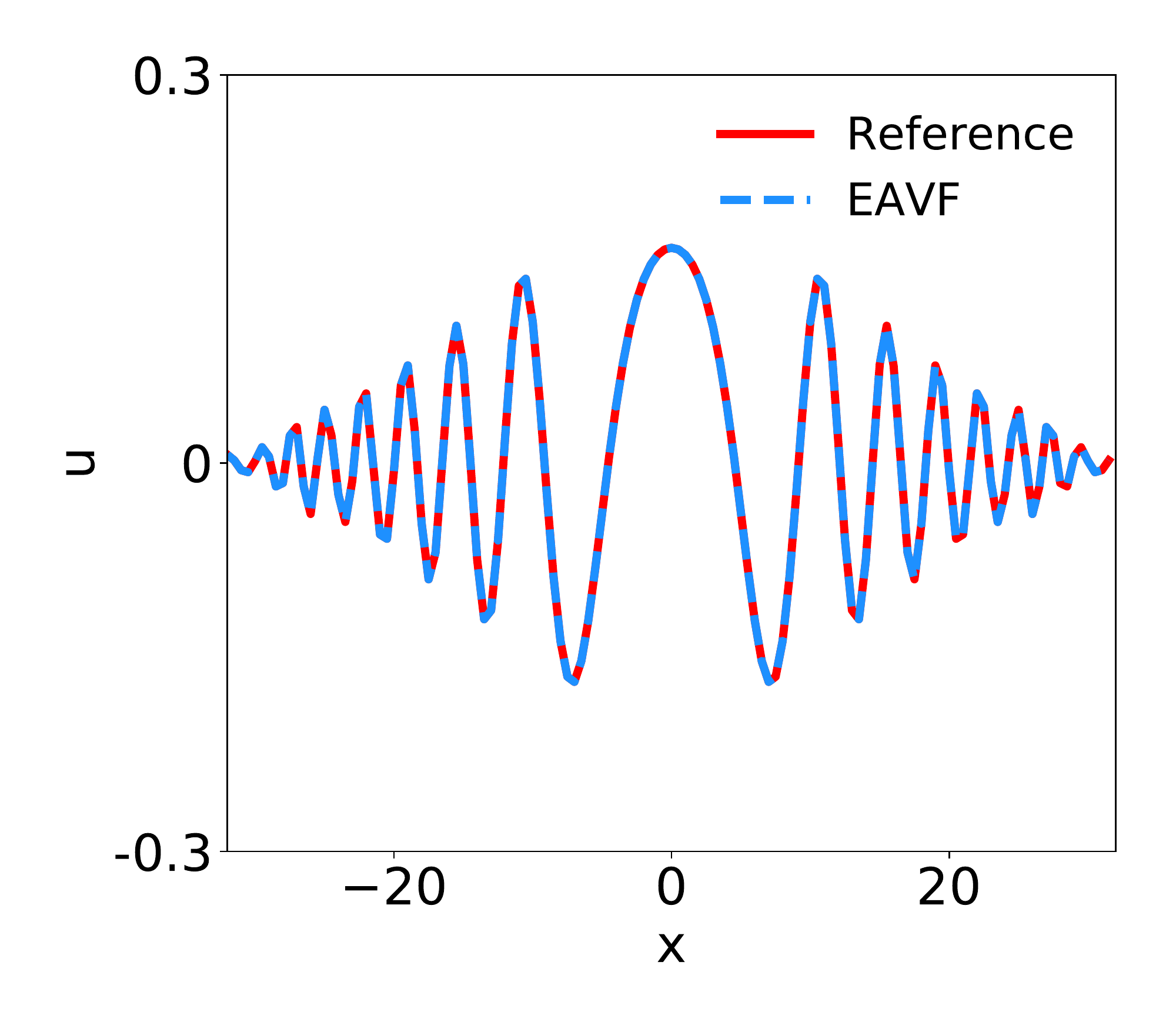}
	\end{minipage}
	\begin{minipage}[t]{0.24\textwidth}
	\includegraphics[width=1\linewidth]{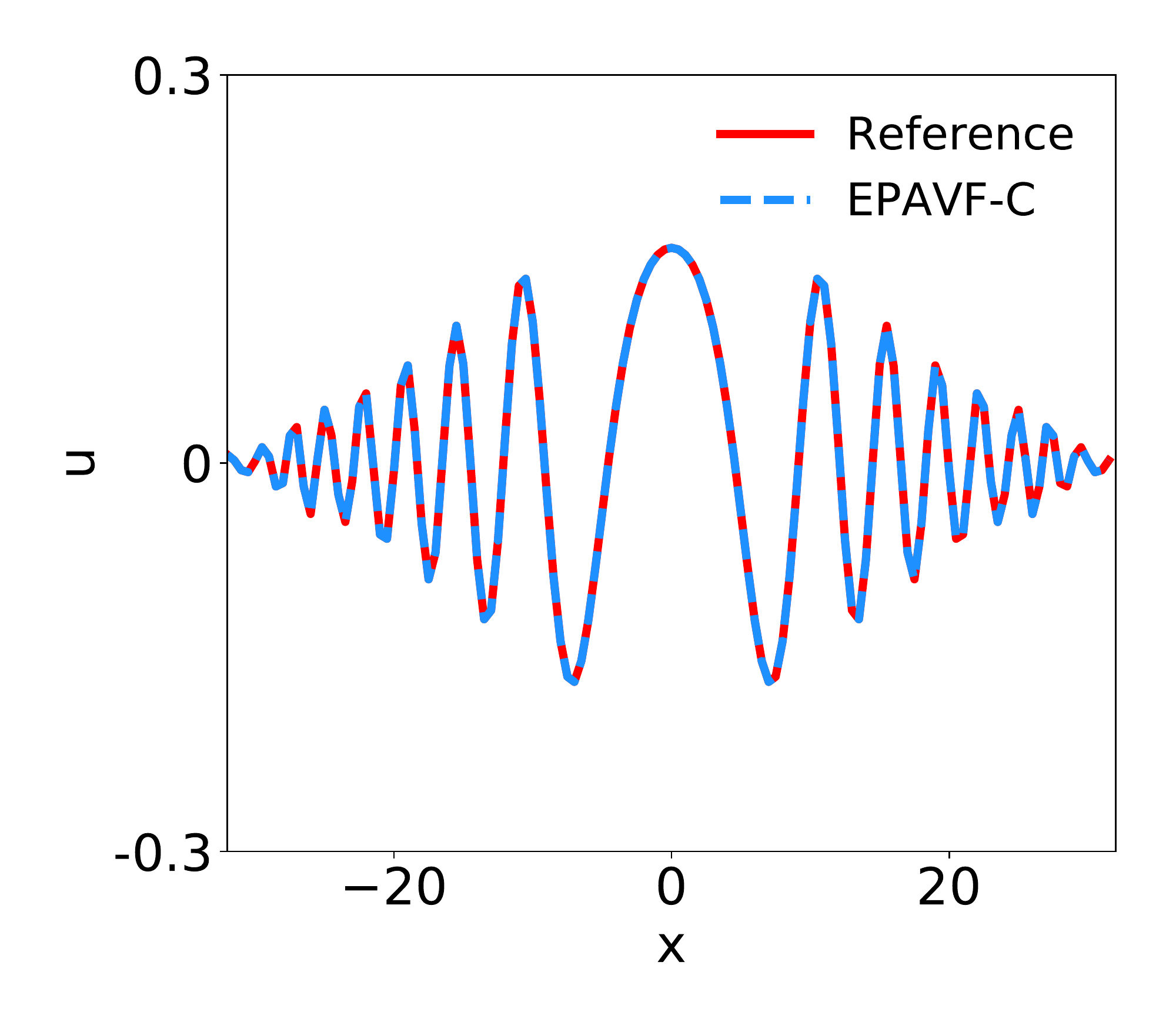}
	\end{minipage}
	\caption{Numerical solutions of the KGS equations by different schemes with $h=1/2$ and $\tau=0.1$ at $t= 10$.}
	\label{SOLUTIONS T=10}
\end{figure}

Besides the accuracy, we also compare the efficiency of the four schemes. From Figure~\ref{CPU Error of solutions}, we find that the computational efficiency of the exponential integrators is also higher than the nonexponential ones. Specifically, EPAVF-C is the most efficient method among them, while AVF is the least efficient, regarding $\psi$ and $u$. For EAVF and PAVF-C, however, the performance of their efficiency depends on the variables.  That is,  we can easily tell the difference between EAVF and PAVF-C for $\psi$, but this difference becomes subtle instead for $u$.
\begin{figure}[H]
	\centering
	\begin{minipage}[t]{0.45\textwidth}
		\centering
		\includegraphics[width=1\linewidth]{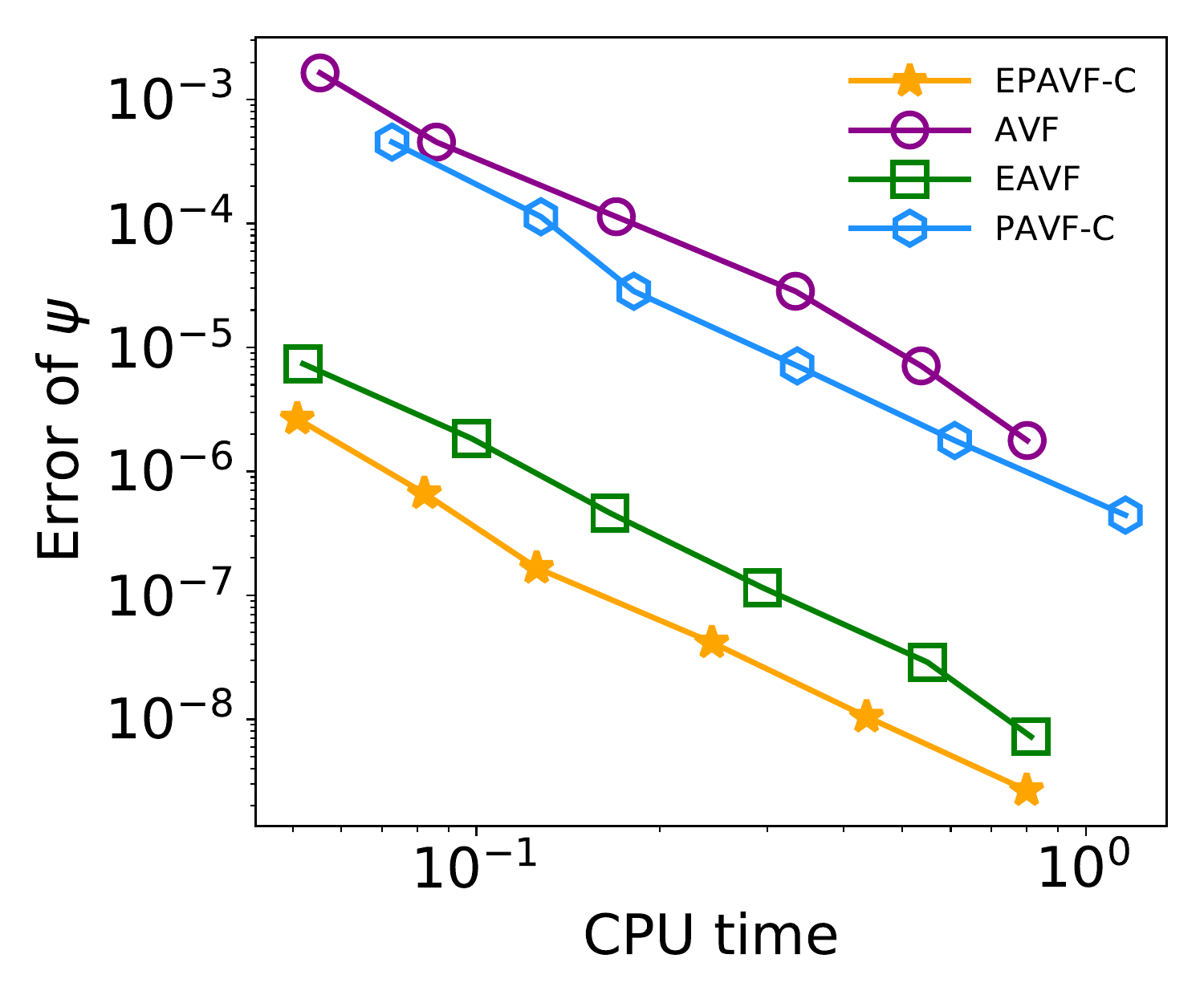}
	\end{minipage}
	\begin{minipage}[t]{0.45\textwidth}
		\centering
		\includegraphics[width=1\linewidth]{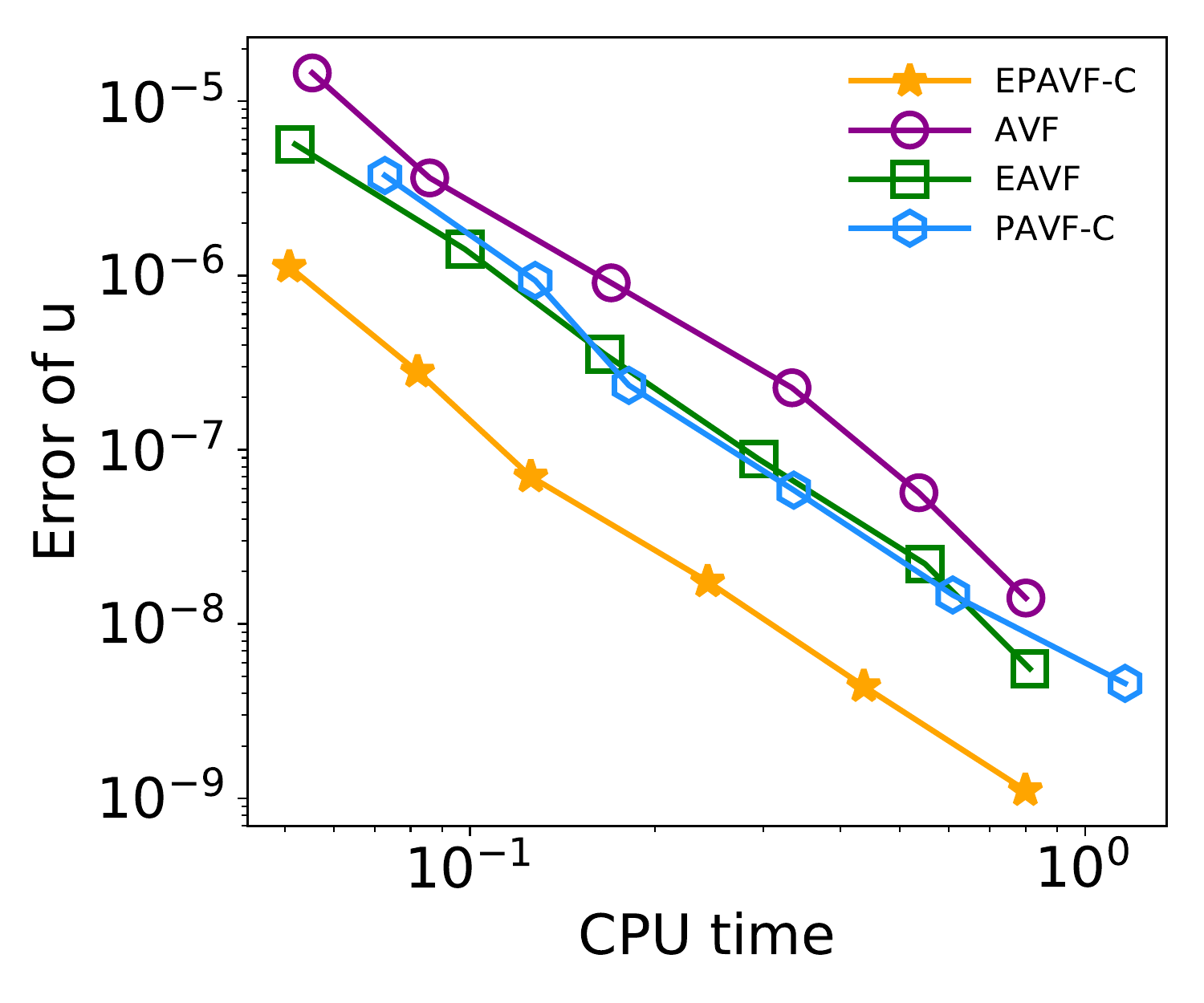}
	\end{minipage}
	\caption{Computational efficiency of different schemes for the KGS equations with respect to $\psi$ (left) and $u$ (right) under the time step $\tau = 0.025 \times 2^{-k}, k = 0,\cdots,5$ and $\varepsilon=1$.}
	\label{CPU Error of solutions}
\end{figure}

Next, we discuss the ordering problem mentioned in Remark.~\ref{rmk1}. Notice that EPAVF \eqref{scheme-kgs} and its adjoint \eqref{adjoint-kgs} can be viewed as two kinds of orderings. In Figure~\ref{rate kgs epavf, adjoint}, we present the temporal errors of different $\varepsilon$ for both $\varphi$ and $u$. When $\varepsilon$ is considerably large, the performance of these two schemes is merely the same. However, the difference emerges as $\varepsilon$ becomes smaller. EPAVF \eqref{scheme-kgs}  shows smaller errors and a faster convergence rate than the adjoint scheme. In view of schemes \eqref{scheme-kgs} and \eqref{adjoint-kgs}, we may conclude that updating the non-oscillatory parts first and then computing the highly oscillatory parts may lead to smaller errors.
\begin{figure}[H]
	\centering
	\begin{minipage}[t]{0.45\textwidth}
		\includegraphics[width=1\linewidth]{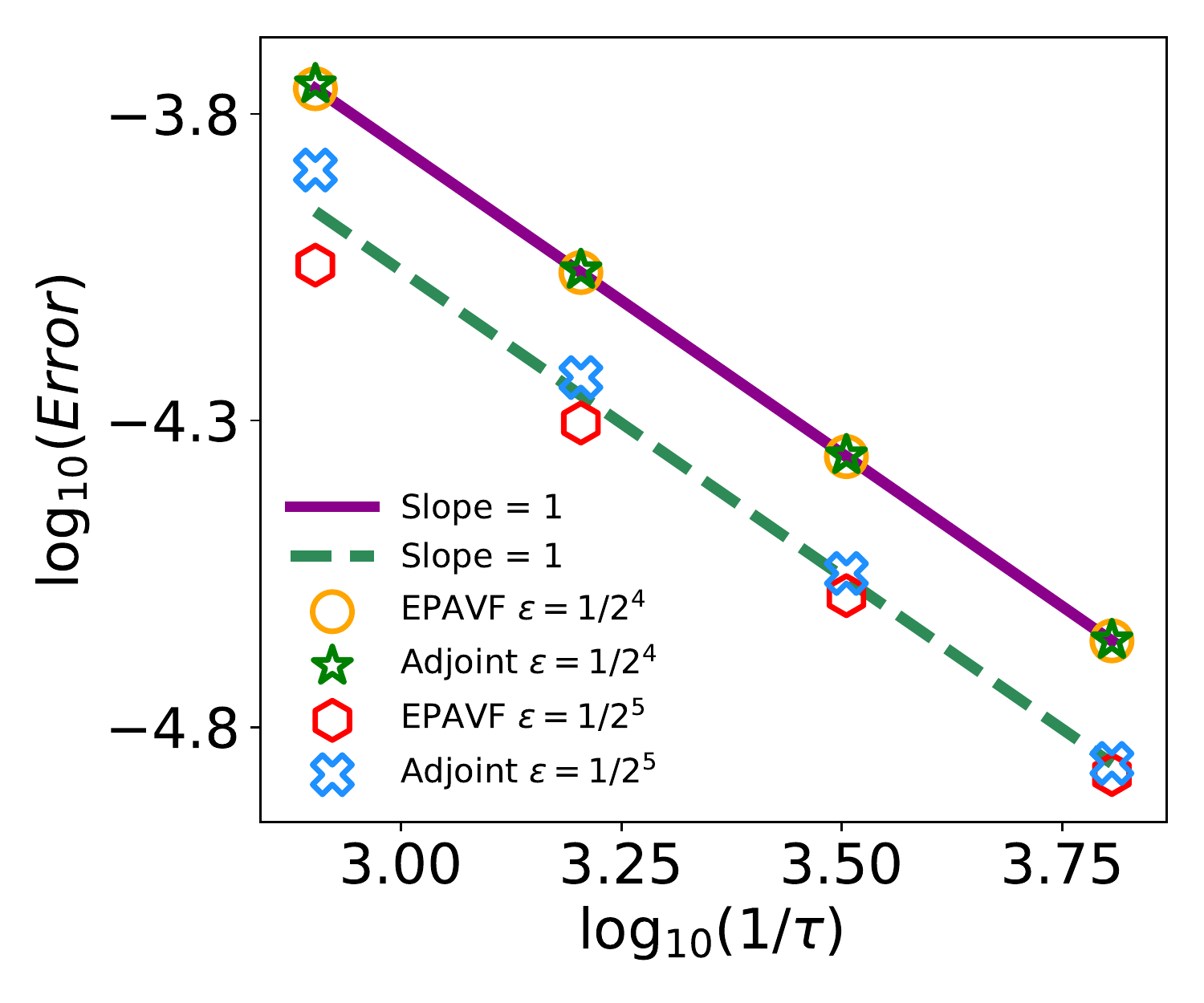}
	\end{minipage}
	\begin{minipage}[t]{0.45\textwidth}
		\includegraphics[width=1\linewidth]{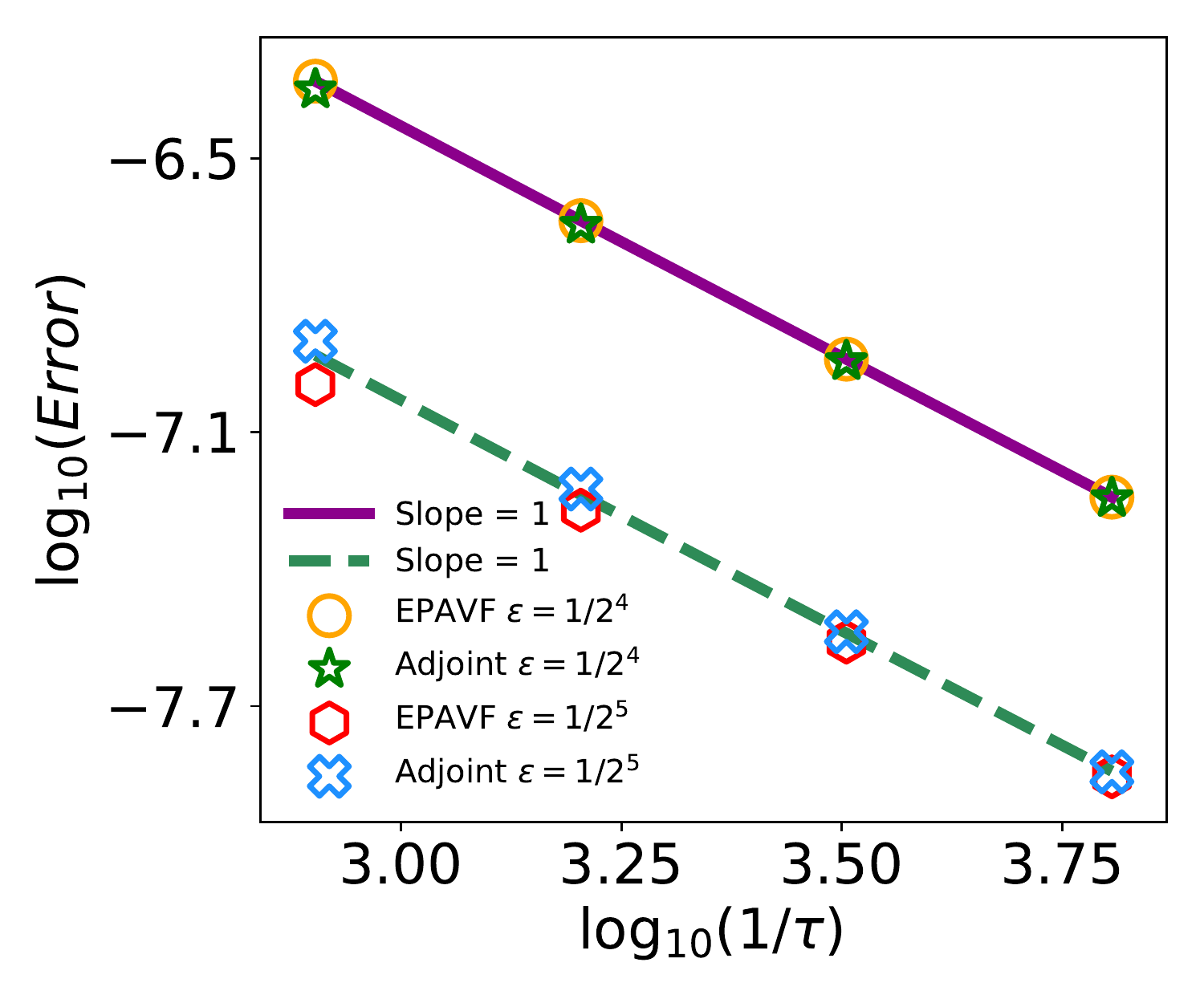}
	\end{minipage}
	\caption{Temporal error of EPAVF and its adjoint for the  KGS equations with respect to $\psi$ (left) and $u$ (right) under the time step $\tau = 0.00125 \times 2^{-k}, k =0 , \cdots, 3$.}
	\label{rate kgs epavf, adjoint}
\end{figure}

In the last experiment of this example, we present the result of the evolution of Hamiltonian energy by EPAVF-C. The EPAVF scheme and its adjoint exhibit similar behavior, and we omit here. We use the relative energy error \begin{equation}
	RH^n = \big|(H^n - H^0)/H^0\big|,
\end{equation}
to measure the conservative properties in the rest experiments.
From Figure~\ref{conservation kgs1d}, we can observe that EPAVF-C preserves the discrete energy to machine accuracy, independent of the decreasing of $\varepsilon$. Since the implementation of EPAVF-C involves the fixed point iteration, one can see a linear growth in the energy errors, although the scheme is exactly energy conservation. How to avoid the accumulation of round-off errors caused by iteration needs further studies.

\begin{figure}[H]
	\centering
	\includegraphics[width = 0.9\textwidth]{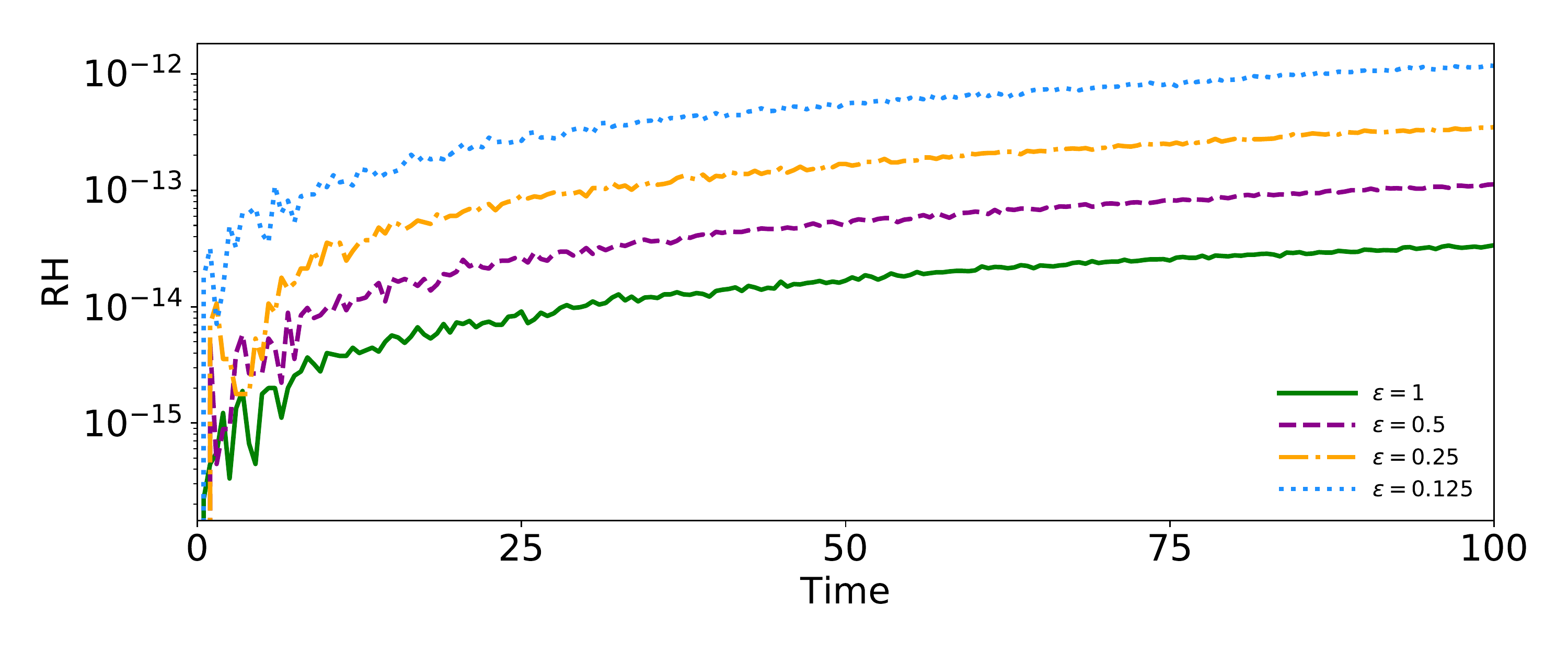}
	\caption{Energy error of 1D KGS equations solved by EPAVF-C with different $\varepsilon$.}
	\label{conservation kgs1d}
\end{figure}

\begin{ex}
	In this experiment, we consider the 2D KGS equations with the initial conditions  
	\begin{equation*}
		\begin{split}
			&\psi_{0}\left(x, y\right) = \frac{2}{\exp{\left( x^{2} + 2y^{2} \right)}+\exp{\left(- x^{2} - 2y^{2} \right)}}\exp{\left(5i \cdot\mbox{\rm sech}{\left( \sqrt{4x^2 + y^2} \right)}\right)}, \\ 
			& u_{0}\left(x, y\right) = \exp{\left(-x^{2} - y^{2}\right)}, \ u_{1}\left(x, y\right) = \dfrac{\exp{\left(-x^{2} - y^{2}\right)}}{2},
		\end{split}
	\end{equation*}
and periodic boundary conditions. This problem will be solved on the spatial domain $\Omega = \left[-64, 64\right] \times\left[-64, 64\right]$ until $t= 10$. The parameter $\beta$ is fixed to $1$, and numerical solutions of the KGS equations with $\varepsilon =1, 0.1, 0.01$ will be simulated by EPAVF-C under spatial step $h_x = h_y = 1/4$ and time step $\tau = 0.1$.
\end{ex}

Figures~\ref{Kgs 2D phih}-\ref{Kgs 2D uh} show the snapshots of numerical solutions $|\psi|^2$ and $u$ by EPAVF-C with different $\varepsilon$ at $t = 1, 2, 4, 6$, respectively. As the value of $\varepsilon$ decreasing, highly oscillatory waves emerges, especially from the snapshots of $u$. Nevertheless, the EPAVF-C scheme can well capture the wave dynamics and iterations. Moreover, the corresponding energy is well-preserved even for very small $\varepsilon$ in Figure~\ref{energy 2dkgs}.

\begin{figure}[H]
	\centering
	\begin{minipage}{0.23\textwidth}
		\includegraphics[width=1.0\linewidth]{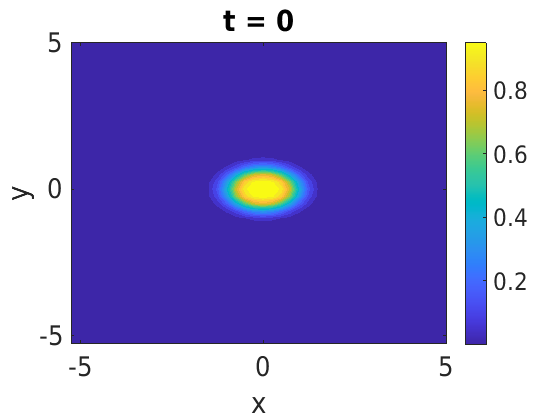}
	\end{minipage}
	\begin{minipage}{0.23\textwidth}
		\includegraphics[width=1.0\linewidth]{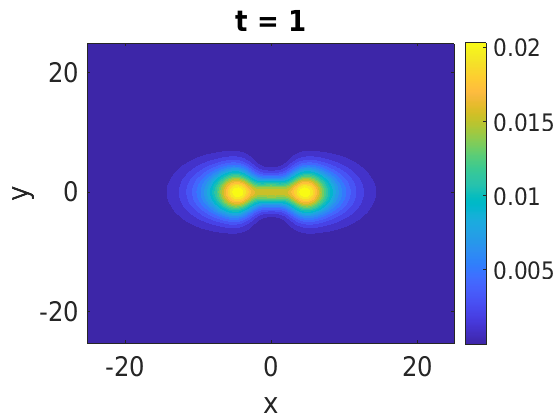}
	\end{minipage}
	\begin{minipage}{0.23\textwidth}
		\includegraphics[width=1.0\linewidth]{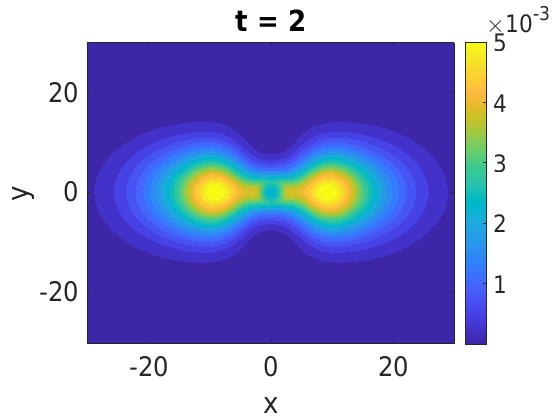}
	\end{minipage}
	\begin{minipage}{0.23\textwidth}
		\includegraphics[width=1.0\linewidth]{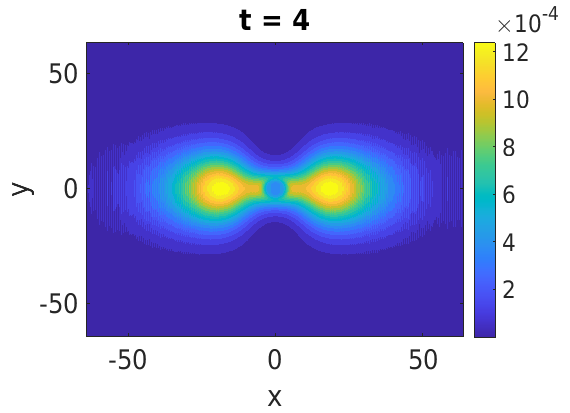}
	\end{minipage}
		\begin{minipage}{0.23\textwidth}
		\includegraphics[width=1.0\linewidth]{kgs2d_phih_t0.png}
	\end{minipage}
	\begin{minipage}{0.23\textwidth}
		\includegraphics[width=1.0\linewidth]{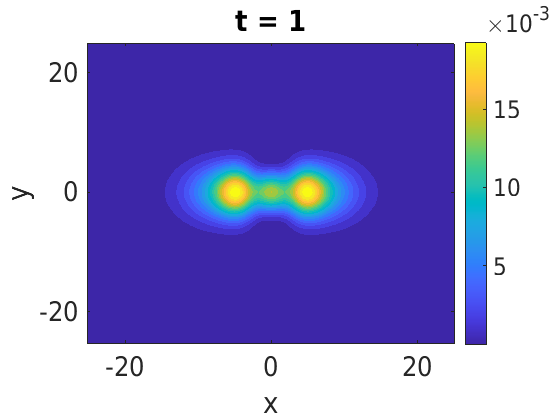}
	\end{minipage}
	\begin{minipage}{0.23\textwidth}
		\includegraphics[width=1.0\linewidth]{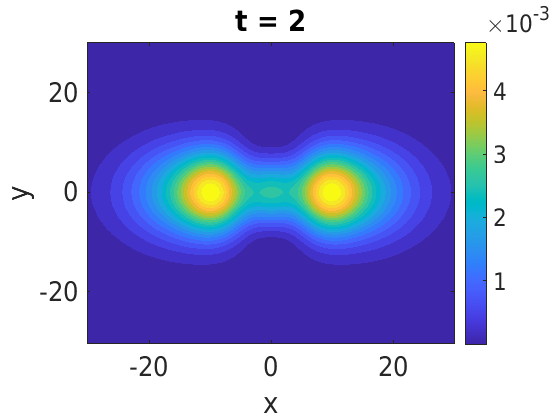}
	\end{minipage}
	\begin{minipage}{0.23\textwidth}
		\includegraphics[width=1.0\linewidth]{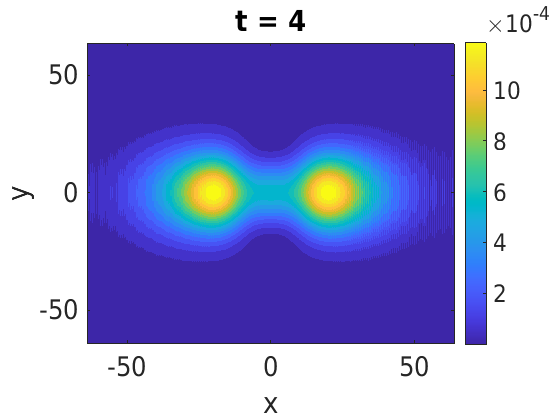}
	\end{minipage}
		\begin{minipage}{0.23\textwidth}
		\includegraphics[width=1.0\linewidth]{kgs2d_phih_t0.png}
	\end{minipage}
	\begin{minipage}{0.23\textwidth}
		\includegraphics[width=1.0\linewidth]{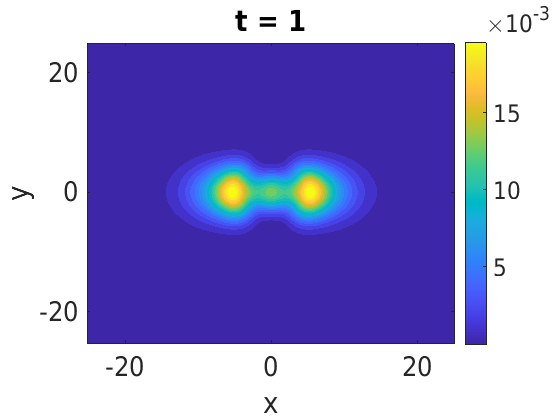}
	\end{minipage}
	\begin{minipage}{0.23\textwidth}
		\includegraphics[width=1.0\linewidth]{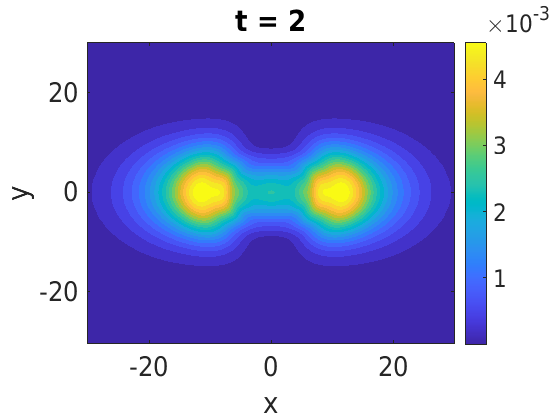}
	\end{minipage}
	\begin{minipage}{0.23\textwidth}
		\includegraphics[width=1.0\linewidth]{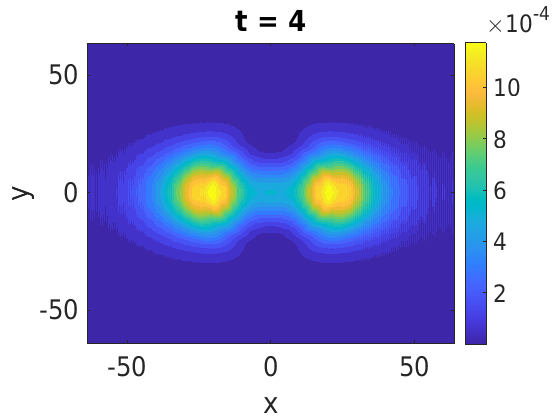}
	\end{minipage}

	\caption{Snapshots of $\psi$ for the 2D KGS equations by EPAVF--C with $\varepsilon = 1$ (first row); $\varepsilon = 0.1$ (middle row); $\varepsilon = 0.01$ (last row).}
	\label{Kgs 2D phih}
\end{figure}

\begin{figure}[H]
	\centering
	\begin{minipage}{0.23\textwidth}
		\includegraphics[width=1.0\linewidth]{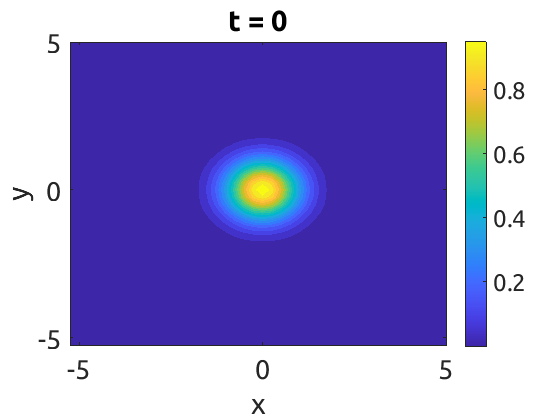}
	\end{minipage}
	\begin{minipage}{0.23\textwidth}
		\includegraphics[width=1.0\linewidth]{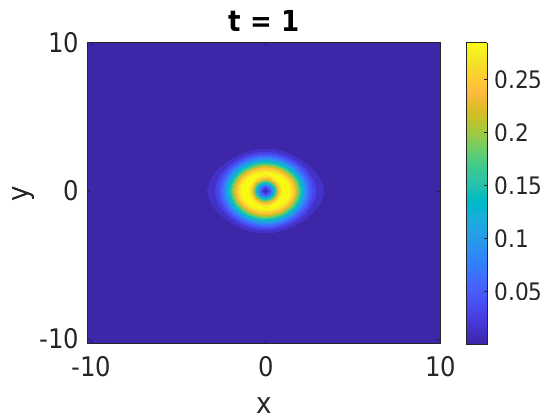}
	\end{minipage}
	\begin{minipage}{0.23\textwidth}
		\includegraphics[width=1.0\linewidth]{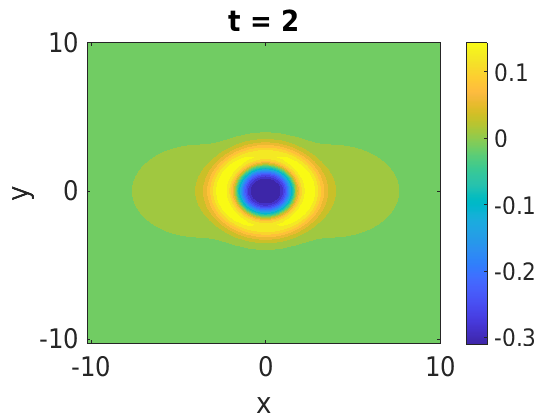}
	\end{minipage}
	\begin{minipage}{0.23\textwidth}
		\includegraphics[width=1.0\linewidth]{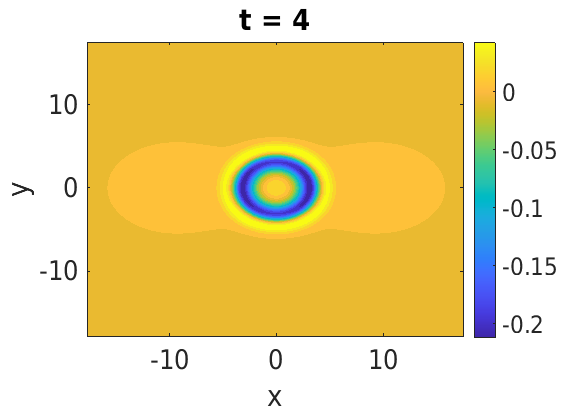}
	\end{minipage}
	\begin{minipage}{0.23\textwidth}
		\includegraphics[width=1.0\linewidth]{kgs2d_uh_t0.png}
	\end{minipage}
	\begin{minipage}{0.23\textwidth}
		\includegraphics[width=1.0\linewidth]{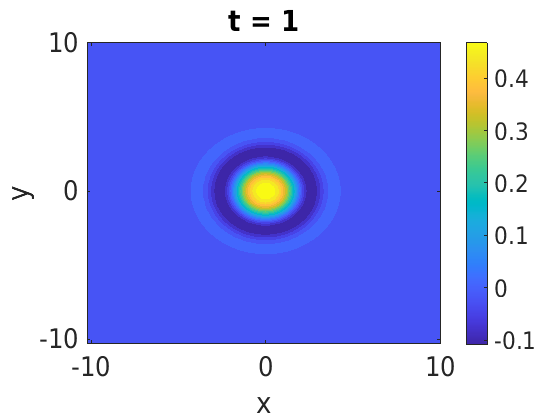}
	\end{minipage}
	\begin{minipage}{0.23\textwidth}
		\includegraphics[width=1.0\linewidth]{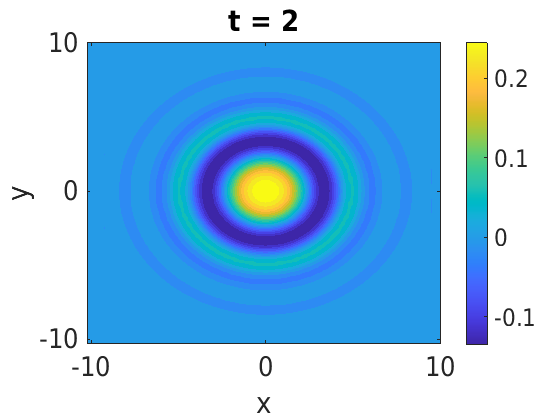}
	\end{minipage}
	\begin{minipage}{0.23\textwidth}
		\includegraphics[width=1.0\linewidth]{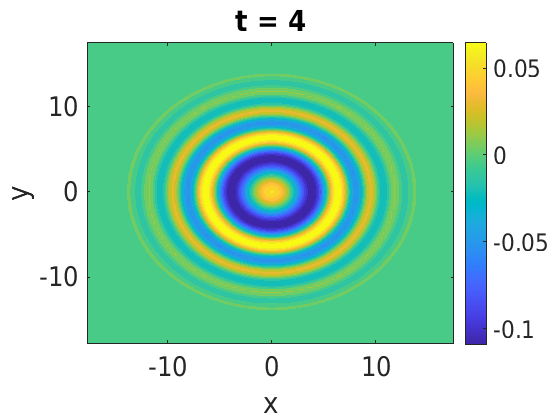}
	\end{minipage}
	\begin{minipage}{0.23\textwidth}
		\includegraphics[width=1.0\linewidth]{kgs2d_uh_t0.png}
	\end{minipage}
	\begin{minipage}{0.23\textwidth}
		\includegraphics[width=1.0\linewidth]{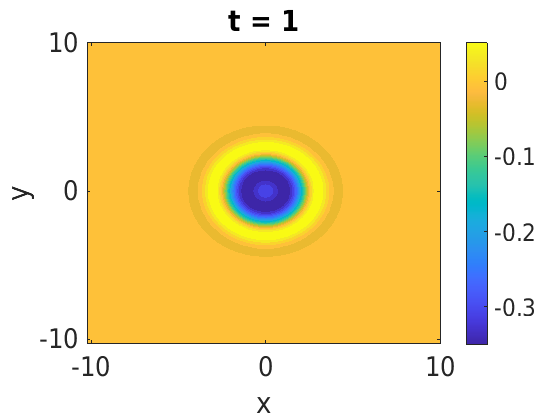}
	\end{minipage}
	\begin{minipage}{0.23\textwidth}
		\includegraphics[width=1.0\linewidth]{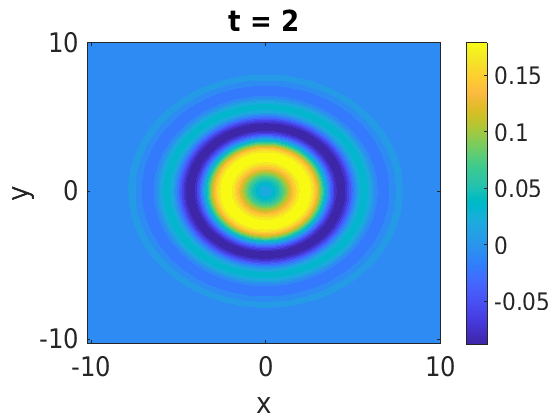}
	\end{minipage}
	\begin{minipage}{0.23\textwidth}
		\includegraphics[width=1.0\linewidth]{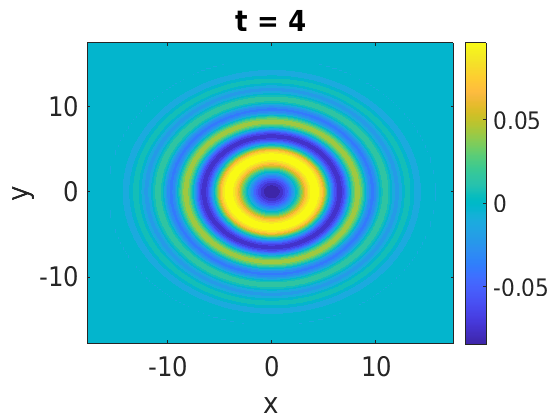}
	\end{minipage}
	\caption{Snapshots of $u$ for the 2D KGS equations by EPAVF-C with $\varepsilon = 1$ (first row); $\varepsilon = 0.1$ (middle row); $\varepsilon = 0.01$ (last row).}
	\label{Kgs 2D uh}
\end{figure}
\begin{figure}[H]
	\centering
		\includegraphics[width=0.9\textwidth]{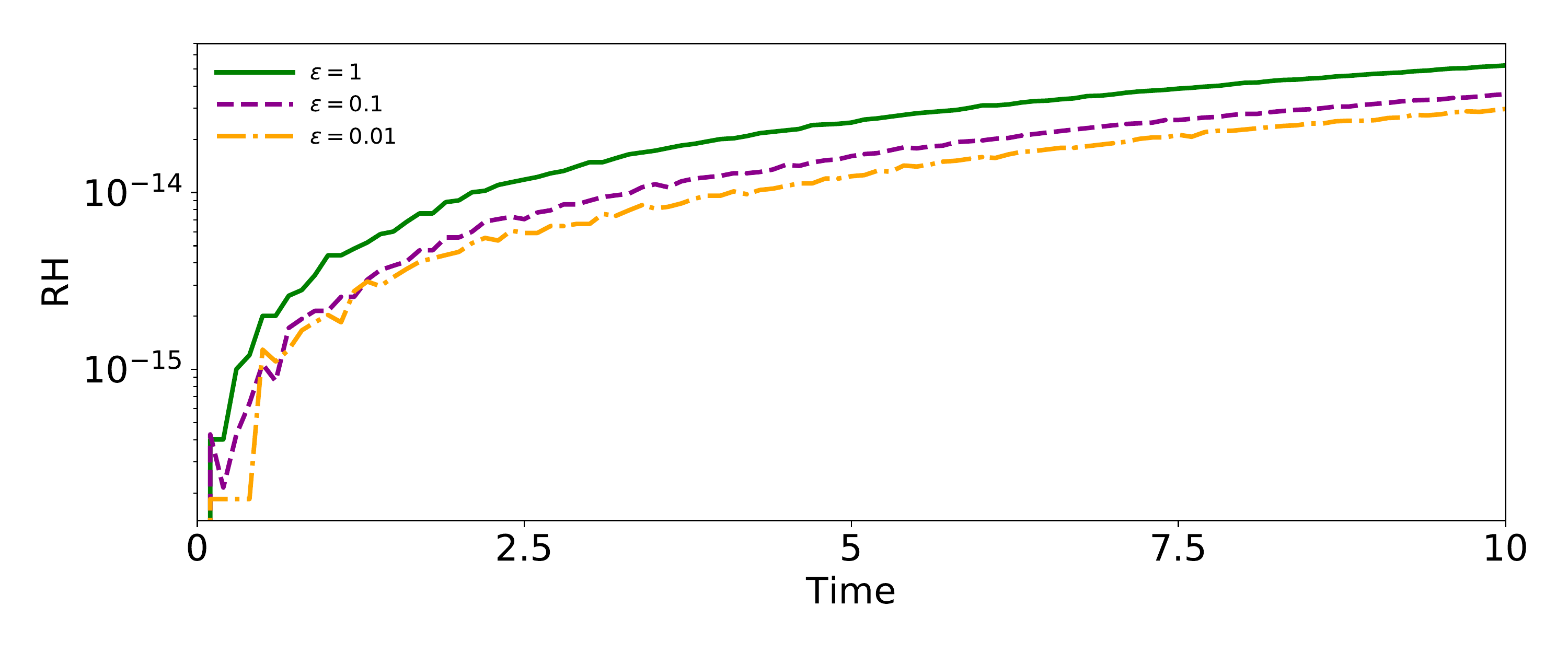}
		\caption{Energy error of the 2D KGS equations by EPAVF-C with different $\varepsilon$. }
		\label{energy 2dkgs}
\end{figure}

\subsection{Klein-Gordon-Zakharov equations}
Next, we consider the following coupled KGZ equations  
\begin{equation}\label{KGZ origin}
	\begin{cases}
		\varepsilon^{2}E_{tt}(\mathbf{x}, t) - \Delta E(\mathbf{x}, t) + \dfrac{1}{\varepsilon^{2}}E(\mathbf{x}, t) + E(\mathbf{x}, t)M(\mathbf{x}, t) = 0, \\ 
		M_{tt}(\mathbf{x}, t) - \Delta M(\mathbf{x}, t) - \Delta\left(E(\mathbf{x}, t)\right)^{2} = 0, 
	\end{cases}
\end{equation}
equipped with the following initial-boundary conditions
\begin{equation*}
	\begin{cases}
		E(\mathbf{x}, 0) = E_{0}(\mathbf{x}), \ E_t(\mathbf{x}, 0) = \dfrac{1}{\varepsilon^{2}}E_{1}(\mathbf{x}), \\
		M(\mathbf{x}, 0) = M_{0}(\mathbf{x}), \ M_t(\mathbf{x}, 0) = M_{1}(\mathbf{x}), \\ 
		\left. M(\mathbf{x}, t)\right|_{\partial \Omega} = \left. E(\mathbf{x}, t)\right|_{\partial \Omega} = 0,
	\end{cases}
\end{equation*}
where $E\left(\textbf{x}, t\right), M\left(\textbf{x}, t\right)$ are real-valued functions representing the scale component of the electric field raised by electrons and the derivation of ion density from its equilibrium, respectively, $\varepsilon$ is a dimensionless parameter inversely proportional to the plasma frequency \cite{bao2016}. The KGZ equations (\ref{KGZ origin}) is a classical model for describing the mutual interaction between the Langmuir waves and ion acoustic waves in a plasma. By introducing intermediate variables $F=E_t$, $N_t=M+E^2$, the system (\ref{KGZ origin}) can be reformulated as the following first-order system 
\begin{equation}\label{kgz-1st}
\left\lbrace	\begin{aligned}
		E_t&= F, \\ 
		\varepsilon^{2}F_t &= \Delta E - \dfrac{1}{\varepsilon^{2}}E - EM, \\ 
		M_t &= \Delta N, \\ 
		N_t &= M + E^2,
	\end{aligned}\right.
\end{equation}
which also admits an infinite-dimensional Hamiltonian structure
\begin{equation*}
	{Z}_t = \mathcal{D}\frac{\delta \mathcal{H}}{\delta Z}, \quad
	Z = \left( E, F, M, N\right)^{\top}, \quad 
	\mathcal{D} = 
	\begin{pmatrix}
		0 & \frac{1}{2\varepsilon^{2}} & 0 & 0 \\
		-\frac{1}{2\varepsilon^{2}} & 0 & 0 & 0 \\ 
		0 & 0 & 0 & -1 \\ 
		0 & 0 & 1 & 0 
	\end{pmatrix},
\end{equation*}
and the corresponding Hamiltonian functional $\mathcal{H}$ is given by
\begin{equation*}
		\mathcal{H}(t)= \int_{\Omega} \Big(\dfrac{1}{\varepsilon^{2}}E^{2} + \varepsilon^{2} F^{2}+\left| \nabla E\right|^{2} \Big)+\Big(\dfrac{1}{2} M^{2} + \dfrac{1}{2}\left|\nabla N\right|^{2}\Big) + ME^{2} d\mathbf{x}.
\end{equation*}

\subsubsection{Spatial discretization}

Notice that zero Dirichlet boundary conditions are considered for the KGZ equations, thus it is natural to apply the sine pseudospectral method \cite{Zhao2016On} to achieve a highly accurate semi-discretization. The resulting semi-discrete scheme for the one-dimensional KGZ equations \eqref{kgz-1st} is then given by
\begin{equation}\label{kgz-semi}
\left\lbrace	\begin{aligned}
		\textbf{E}_t&= \textbf{F}, \\ 
		\varepsilon^{2}\textbf{F}_t &= \mathbb{D}_2 \textbf{E} - \dfrac{1}{\varepsilon^{2}}\textbf{E}-\textbf{E}\textbf{M}, \\ 
		\textbf{M}_t &= \mathbb{D}_2\textbf{N}, \\ 
		\textbf{N}_t&= \textbf{M}+\textbf{E}^2.
	\end{aligned}\right.
\end{equation}
Without any ambiguity, we still use $\mathbb{D}_{2}$ to represent the second-order spectral differential matrix related to the sine pseudospectral method, which can also be diagonalized as
\begin{equation*}
	\mathbb{D}_{2} = {S}_N^{-1}\Lambda_{dir}{S}_N, \quad \Lambda_{dir}= - \Big[\frac{\mu}{2}\mbox{diag}\big(1,2,\cdots,N-1\big)\Big]^2,
\end{equation*}
where $\mu$ is defined the same as that for the KGS equations, ${S}_N$ and $S_N^{-1}$  represent the discrete sine transform and its inverse, respectively. Hence, the fast sine transform (FST) can be utilized to accelerate the computation.

Further arranging the semi-discrete system \eqref{kgz-semi}, we obtain
\begin{equation}\label{KGZ semi}
	\left\lbrace\begin{aligned}
		\left(\begin{array}{c}
			{\textbf{E}}_t \\ {\textbf{F}}_t
		\end{array}\right)
		&=
			\left(\begin{array}{cc}
			0 & \frac{1}{2\varepsilon^{2}}I_{N} \\ 
			-\frac{1}{2\varepsilon^{2}}I_{N} & 0
		\end{array}\right)
		\left[
			\left(\begin{array}{cc}
			\frac{2}{\varepsilon^{2}}I_{N} - 2\mathbb{D}_{2} & 0 \\ 
			0 & 2\varepsilon^{2}I_{N}
		\end{array}\right)
			\left(\begin{array}{c}
			\textbf{E} \\ \textbf{F} 
		\end{array}\right)
		+ 
			\left(\begin{array}{c}
			2\textbf{E}\odot\textbf{M} \\ 0
		\end{array}\right)
		\right],
		\\[1ex]
		\left(\begin{array}{c}
			{\textbf{M}}_t \\ {\textbf{N}}_t
	\end{array}\right)
		&= 
		\left(\begin{array}{cc}
			0 & -I_{N} \\ 
			I_{N} & 0
		\end{array}\right)
		\left[
			\left(\begin{array}{cc}
			I_{N} & 0 \\ 
			0 & -\mathbb{D}_{2}
			\end{array}\right)
			\left(\begin{array}{c}
			\textbf{M} \\ \textbf{N}
			\end{array}\right)
		+
			\left(\begin{array}{c}
			\textbf{E}^2 \\ 0
		\end{array}\right)
		\right],
	\end{aligned}\right.
\end{equation}
which possesses the following discrete energy conservation law due to the symmetry of $\mathbb{D}_2$, i.e., 
\begin{equation}\label{kgz-semi-ene}
	\frac{d}{dt}H=0,\quad {H}=  \Big(\dfrac{1}{\varepsilon^{2}}\|\textbf{E}\|_{h,dir}^2 + \varepsilon^{2} \|\textbf{F}\|_{h,dir}^2+|\textbf{E}|_{1,dir}^{2} \Big)+\Big(\dfrac{1}{2} \|\textbf{M}\|_{h,dir}^2 + \dfrac{1}{2}|\textbf{N}|_{1,dir}^{2}\Big) + (\textbf{M},\textbf{E}^{2})_{h,dir},
\end{equation}
where the discrete inner product and related norms are defined as follows:
\begin{equation*}
	\left(\mathbf{u}, \mathbf{v}\right)_{h,dir} = h\sum\limits_{l = 1}^{N_{x}-1} \mathbf{u}_{l}\mathbf{v}_{l}, \quad \Vert \mathbf{u} \Vert_{h,dir} = \left(\mathbf{u}, \mathbf{u}\right)_{h, dir}^{1/2}, \quad |\mathbf{u}|_{1,dir} = \left(-\mathbb{D}_{2}\mathbf{u}, \mathbf{u}\right)^{1/2}_{h,dir}.
\end{equation*}
Here, `\textit{dir}' is used for the homogeneous Dirichlet boundary condition.

\subsubsection{Derivation of the EPAVF schemes}

Similarly, we still need to calculate the matrix exponentials in advance. It is clear that the corresponding matrices of the system \eqref{KGZ semi} have different structures from that in \eqref{KGS semi}, and the resulting matrix exponentials are given by
\begin{equation}\label{kgz-exp}
	\exp{\left(V_{1}\right)} =
	\left( 
	\begin{smallmatrix}
		\cos{\left(\tau \widetilde{\mathbb{D}}_{2}^{1/2}\right)} & \frac{\sin{\left(\tau \widetilde{\mathbb{D}}_{2}^{1/2}\right)}}{\widetilde{\mathbb{D}}_{2}^{1/2}} \\ 
		-\widetilde{\mathbb{D}}_{2}^{1/2}\sin{\left(\tau \widetilde{\mathbb{D}}_{2}^{1/2}\right)} & \cos{\left(\tau \widetilde{\mathbb{D}}_{2}^{1/2}\right)} 
	\end{smallmatrix}
	\right), \quad 
	\exp{\left(V_{2}\right)} = 
	\left(
	\begin{smallmatrix}
		\cos{\left(\tau(-\mathbb{D}_{2})^{1/2}\right)} & -(-\mathbb{D}_{2})^{1/2}\sin{\left(\tau (-\mathbb{D}_{2})^{1/2}\right)} \\ 
		\frac{\sin{\left(\tau(-\mathbb{D}_{2})^{1/2}\right)}}{(-\mathbb{D}_{2})^{1/2}} & \cos{\left(\tau (-\mathbb{D}_{2})^{1/2}\right)}  
	\end{smallmatrix}
	\right),
\end{equation}
\begin{equation}\label{kgz-phi}
	\varphi\left(V_{1}\right) = 
	\left(
	\begin{smallmatrix}
		\frac{\sin{\left(\tau \widetilde{\mathbb{D}}_{2}^{1/2}\right)}}{\tau\widetilde{\mathbb{D}}_{2}^{1/2}} & \frac{I - \cos{\left(\tau \widetilde{\mathbb{D}}_{2}^{1/2}\right)}}{\tau\widetilde{\mathbb{D}}_{2}} \\
		\frac{\cos{\left(\tau \widetilde{\mathbb{D}}_{2}^{1/2}\right)} - I}{\tau} & 	\frac{\sin{\left(\tau \widetilde{\mathbb{D}}_{2}^{1/2}\right)}}{\tau\widetilde{\mathbb{D}}_{2}^{1/2}}
	\end{smallmatrix}
	\right), \quad 
	\varphi\left(V_{2}\right) = 
	\left(
	\begin{smallmatrix}
		\frac{\sin{\left(\tau (-\mathbb{D}_2)^{1/2}\right)}}{\tau(-\mathbb{D}_2)^{1/2}} & \frac{\cos{\left(\tau (-\mathbb{D}_2)^{1/2}\right)} - I}{\tau} \\ 
	\frac{I - \cos{(\tau (-\mathbb{D}_2)^{1/2})}}{-\tau \mathbb{D}_2}&  \frac{\sin{\left(\tau (-\mathbb{D}_2)^{1/2}\right)}}{\tau(-\mathbb{D}_2)^{1/2}}
	\end{smallmatrix}
	\right),
\end{equation}
where $\widetilde{\mathbb{D}}_2 = \frac{1}{\varepsilon^{4}} - \frac{1}{\varepsilon^{2}}\mathbb{D}_2$. Without any ambiguity, we still use the notations ${\exp}^i_{jk}$ and ${\varphi}^i_{jk}$, $i,j,k=1,2$ to represent the elements in $\exp(V_i)$ and $\varphi(V_i)$ for the KGZ equations, respectively.

\begin{rmk}
   As mentioned above, the differential matrix $\mathbb{D}_2$ under zero boundary conditions can also be diagonalized, so that the computation of those matrix exponentials is done similarly as \eqref{computed of exponential}. The only differences are the definitions of $\widetilde{\exp}^i_{jk}$ and $\widetilde{\varphi}^i_{jk}$, whose elements are obtained by replacing $\mathbb{D}_2$ and $\widetilde{\mathbb{D}}_2$ in \eqref{kgz-exp}-\eqref{kgz-phi} with $\mbox{diag}(\Lambda_{dir})$ and $\mbox{diag}(\widetilde{\Lambda}_{dir})$ respectively, where $\widetilde{\Lambda}_{dir}=\frac{1}{\varepsilon^{4}} - \frac{1}{\varepsilon^{2}}\Lambda_{dir}$.

\end{rmk}

Applying the EPAVF method for the semi-discretization \eqref{KGZ semi}, we obtain
\begin{equation}\label{epavf kgz}
	\begin{cases}
		\mathbf{E}^{n+1} = \exp_{11}^1\mathbf{E}^{n} + \exp_{12}^1\mathbf{F}^{n} - \frac{\tau}{\varepsilon^{2}} \varphi^{1}_{12} \mathbf{M}^{n} \odot\mathbf{E}^{n+1/2}, \\
		\mathbf{F}^{n+1} = \exp_{21}^1\mathbf{E}^{n} + \exp_{22}^1\mathbf{F}^{n} - \frac{\tau}{\varepsilon^{2}} \varphi^{1}_{22} \mathbf{M}^{n} \odot\mathbf{E}^{n+1/2}, \\ 
		\mathbf{M}^{n+1} = \exp_{11}^2\mathbf{M}^{n} + \exp_{12}^2\mathbf{N}^{n} + \tau \varphi^{2}_{12} \left(\mathbf{E}^{n+1}\right)^{2}, \\ 
		\mathbf{N}^{n+1} = \exp_{21}^2\mathbf{M}^{n} + \exp_{22}^2\mathbf{N}^{n} + \tau \varphi^{2}_{22} \left(\mathbf{E}^{n+1}\right)^{2}.
	\end{cases}
\end{equation} 
It is observed that once  $\mathbf{E}^{n+1}$  is solved from the first linear system, the remaining unknowns $\mathbf{F}^{n+1}$, $\mathbf{M}^{n+1}$, and $\mathbf{N}^{n+1}$ can then be explicitly calculated from the last three equations. Therefore, the EPAVF scheme \eqref{epavf kgz} for the KGZ equations is extremely efficient. Similar observations can be made from the adjoint scheme
\begin{equation}\label{epavf kgz adjoint}
	\begin{cases}
		\mathbf{E}^{n+1} = \exp_{11}^1\mathbf{E}^{n} + \exp_{12}^1\mathbf{F}^{n} - \frac{\tau}{\varepsilon^{2}} \varphi^{1}_{12} \mathbf{M}^{n+1} \odot\mathbf{E}^{n+1/2}, \\
		\mathbf{F}^{n+1} = \exp_{21}^1\mathbf{E}^{n} + \exp_{22}^1\mathbf{F}^{n} - \frac{\tau}{\varepsilon^{2}} \varphi^{1}_{22} \mathbf{M}^{n+1}\odot \mathbf{E}^{n+1/2}, \\ 
		\mathbf{M}^{n+1} = \exp_{11}^2\mathbf{M}^{n} + \exp_{12}^2\mathbf{N}^{n} + \tau \varphi^{2}_{12} \left(\mathbf{E}^{n}\right)^{2}, \\ 
		\mathbf{N}^{n+1} = \exp_{21}^2\mathbf{M}^{n} + \exp_{22}^2\mathbf{N}^{n} + \tau \varphi^{2}_{22} \left(\mathbf{E}^{n}\right)^{2}.
	\end{cases}
\end{equation}
We first compute $\mathbf{M}^{n+1}$ and $\mathbf{N}^{n+1}$ from the last two equations explicitly. Subsequently, $\mathbf{E}^{n+1}$ can be solved by the linear system, and then the computation of $\mathbf{F}^{n+1}$ becomes explicit as well. In summary, both EPAVF and its adjoint only have to solve one linear system with variable coefficients while the other three equations are solved explicitly. Therefore, their composition (second-order EPAVF-C) will also be very efficient comparing to the existing schemes. 

\begin{thm}
	The EPAVF schemes, i.e., \eqref{epavf kgz}, \eqref{epavf kgz adjoint} and their composition, all preserve a fully discrete energy conservation law
	\[
	H(\textbf{E}^{n+1},\textbf{F}^{n+1},\textbf{M}^{n+1},\textbf{N}^{n+1})=H(\textbf{E}^n,\textbf{F}^n,\textbf{M}^n,\textbf{N}^n),
	\]
	where the energy function $H$ is defined in \eqref{kgz-semi-ene}.
\end{thm}

For comparisons, we present another two energy-preserving schemes in \cite{KGZWang2007} whose spatial discretizations are replaced by the sine pseudospectral method. The first one is a fully implicit scheme which is denoted by CISP and has the form
\begin{equation*}
	\begin{cases}
		\varepsilon^{2}\delta_{t}^{2}\mathbf{E}^{n} - \frac{1}{2}\mathbb{D}_{2}\left(\mathbf{E}^{n+1} + \mathbf{E}^{n-1}\right) + \frac{1}{2\varepsilon^{2}}\left(\mathbf{E}^{n+1} + \mathbf{E}^{n-1}\right) + \frac{1}{4}\left(\mathbf{M}^{n+1} + \mathbf{M}^{n-1}\right) \odot\left(\mathbf{E}^{n+1} + \mathbf{E}^{n-1}\right)= 0, \\
		\delta_{t}^{2}\mathbf{M}^{n} - \frac{1}{2}\mathbb{D}_{2}\left(\mathbf{M}^{n+1}+\mathbf{M}^{n-1}\right) - \frac{1}{2}\mathbb{D}_{2}\left(\left(\mathbf{E}^{n+1}\right)^{2}+\left(\mathbf{E}^{n-1}\right)^{2}\right) = 0, \
	\end{cases}
\end{equation*}
where $\delta_t^2\textbf{E}^n=(\textbf{E}^{n+1}-2\textbf{E}^{n}+\textbf{E}^{n-1})/\tau^2$.  The other is a decoupled and explicit scheme written as
\begin{equation*}\label{disp}
	\begin{cases}
		\varepsilon^{2}\delta_{t}^{2}\mathbf{E}^{n} - \mathbb{D}_{2}\mathbf{E}^{n} + \frac{1}{2\varepsilon^{2}}\left(\mathbf{E}^{n+1}+\mathbf{E}^{n-1}\right) + \frac{1}{2}\mathbf{M}^{n}\odot\left(\mathbf{E}^{n+1} + \mathbf{E}^{n-1}\right) = 0, \\ 
		\delta_{t}^{2}\mathbf{M}^{n} - \mathbb{D}_{2}\mathbf{M}^{n} - \mathbb{D}_{2}\left(\mathbf{E}^{n}\right)^2 = 0, 
	\end{cases}
\end{equation*}
and will be denoted by DISP in the later discussion.

\subsubsection{Numerical experiments}
\begin{ex}
	Consider the one-dimensional KGZ equations with smooth initial conditions
	\begin{equation*}
		E_{0}(x) = \sin{\left(\dfrac{x}{2}\right)}e^{-x^{2}}, \ E_{1}(x) = \dfrac{1}{2} e^{-\sqrt{2}x^{2}}, \ M_{0}(x) = \mbox{\rm sech}{\left(x^{2}\right)}, \ M_{1}(x) = \cos{\left(\dfrac{x}{3}\right)}e^{-x^{2}},
	\end{equation*}
	and zero boundary conditions on the domain $\Omega= [-32, 32]$.
\end{ex}

To get the convergence rate in time, we first give the definitions of errors as  follows
\begin{equation*}
	e_{E, \varepsilon}^{\tau, h} =  \Vert E(\cdot, 1) - E^{n} \Vert_{\infty}, \quad 
	e_{M, \varepsilon}^{\tau, h} =  \Vert  M(\cdot, 1) - M^{n} \Vert_{\infty},
\end{equation*}
where $E(\cdot,1)$ and $M(\cdot,1)$ are reference solutions obtained by EPAVF-C with very fine mesh $h=1/32$ and time step $\tau=2.5\times 10^{-6}$. Let the spatial step be small enough, i.e., $h = 1/8$ so the spatial errors are negligible. We verify the convergence of EPAVF-C as well as CISP, DISP. In Tables~\ref{tab5}-\ref{tab7}, the corresponding errors and the convergence rates with different $\varepsilon$ are listed. For EPAVF-C, we present the convergence results of both $E$ and $M$, while for others we only show the results of $E$ as a similar result can be found for $M$.

 From Tables~\ref{tab6}-\ref{tab9}, we can draw the following observations: 
 \begin{itemize}
 	\item[(i)] The temporal discretization of all the schemes show a second-order convergence for fixed $\varepsilon$ provided the time step $\tau$ is small enough (cf. each row above the diagonal in Tables~\ref{tab6}-\ref{tab9}). In addition, for fixed $\varepsilon$ and $\tau$, the errors of EPAVF-C are much smaller than that of CISP and DISP, especially when $\varepsilon$ is very small.
 	 	
 	\item[(ii)] The non-exponential ones (CISP and DISP) have a very severe numerical stability constraint when $\varepsilon$ is small (cf. last row in Table~\ref{tab8} and last few rows in Table~\ref{tab9}), while EPAVF-C is stable in computation for all $0<\varepsilon\ll 1$.
 	
 	\item[(iii)] The mesh strategy (or $\varepsilon$-scalability) of EPAVF-C is again $\tau=\mathcal{O}(\varepsilon^2)$ as that for the KGS equations (cf. upper triangle above the diagonal with values in italics of Tables~\ref{tab6}-\ref{tab7}). While the $\varepsilon$-scalability of CISP and DISP is $\tau=\mathcal{O}(\varepsilon^3)$  (cf. upper triangle above the diagonal with values in italics of Tables~\ref{tab8}-\ref{tab9}). These clearly illustrate that EPAVF-C has much better resolution than CISP and DISP.
 	
 \end{itemize}

\begin{table}[H]
        \centering
        \caption{Temporal error analysis of $E$ solved by EPAVF-C with different $\varepsilon$}
        \begin{tabular*}{0.9\textwidth}[h]{@{\extracolsep{\fill}}c c c c c c c c c}
            \toprule[2pt]
            & & $\tau_0 = 0.2$ & $\tau_0/2^2$ & $\tau_0/2^4$ & $\tau_0/2^6$ & $\tau_0/2^8$ & $\tau_0/2^{10}$ \\  
            \hline
             \multirow{2}{*}{$\varepsilon_0=1$} & $e_{E,\varepsilon}^{\tau,h}$ & 4.5026e-04 & 2.8138e-05 & 1.7587e-06 & 1.0992e-07 & 6.8730e-09 & 4.3241e-10 \\ 
            & $\text{Rate}$ & - & 2.0001 & 2.0000 & 2.0000 & 1.9997 & 1.9952 \\ 
            \hline 
            \multirow{2}{*}{$\varepsilon_0/2$} & $e_{E,\varepsilon}^{\tau,h}$ & \textit{4.7201e-03} & 3.00077e-04 & 1.8773e-05 & 1.1734e-06 & 7.33444e-08 & 4.5906e-09 \\ 
            & $\text{Rate}$ & - & 1.9877 & 1.9993 & 2.0000 & 2.0000 & 1.9990 \\ 
            \hline 
            \multirow{2}{*}{$\varepsilon_0/2^2$} & $e_{E,\varepsilon}^{\tau,h}$ & 6.1662e-02 & \textit{4.4606e-03} & 2.8108e-04 & 1.7576e-05 & 1.0985e-06 & 6.8665e-08 \\ 
            & $\text{Rate}$ & - & 1.8945 & 1.9941 & 1.9996 & 2.0000 & 1.9999 \\ 
            \hline 
            \multirow{2}{*}{$\varepsilon_0/2^3$} & $e_{E,\varepsilon}^{\tau,h}$ & 1.0820e-01 & 4.7894e-02 & \textit{3.6695e-03} & 2.3245e-04 & 1.4540e-05 & 9.0864e-07 \\ 
            & $\text{Rate}$ & - & 0.5879 & 1.8531 & 1.9903 & 1.9994 & 2.0001 \\ 
            \hline 
            \multirow{2}{*}{$\varepsilon_0/2^4$} & $e_{E,\varepsilon}^{\tau,h}$ & 1.4215e-01 & 1.4202e-01 & 6.1338e-02 & \textit{4.4850e-03} & 2.8302e-04 & 1.7697e-05 \\ 
            & $\text{Rate}$ & - & 0.0001 & 0.6056 & 1.8868 & 1.9930 & 1.9998 \\ 
            \hline 
            \multirow{2}{*}{$\varepsilon_0/2^5$} & $e_{E,\varepsilon}^{\tau,h}$ & 1.7496e-01 & 1.7323e-01 & 1.7310e-01 & 6.4284e-02 & \textit{4.4111e-03} & 2.7710e-04 \\ 
            & $\text{Rate}$ & - & 0.0072 & 0.0005 & 0.7145 & 1.9326 & 1.9963 \\ 
            \bottomrule[2pt]
       \end{tabular*}
        \label{tab6}
    \end{table}

\begin{table}[H]
        \centering
        \caption{Temporal error analysis of $M$ solved by EPAVF-C with different $\varepsilon$}
        \begin{tabular*}{0.9\textwidth}[h]{@{\extracolsep{\fill}}c c c c c c c c c}
            \toprule[2pt]
            & & $\tau_0 = 0.2$ & $\tau_0/2^2$ & $\tau_0/2^4$ & $\tau_0/2^6$ & $\tau_0/2^8$ & $\tau_0/2^{10}$ \\  
            \hline
             \multirow{2}{*}{$\varepsilon_0=1$} & $e_{M,\varepsilon}^{\tau,h}$ & 1.3973e-03 & 8.5245e-05 & 5.3220e-06 & 3.3358e-07 & 2.1830e-08 & 2.8404e-09 \\ 
            & $\text{Rate}$ & - & 2.0174 & 2.0008 & 1.9979 & 1.9967 & 1.4711 \\ 
            \hline 
            \multirow{2}{*}{$\varepsilon_0/2$} & $e_{M,\varepsilon}^{\tau,h}$ & \textit{2.8792e-03} & 1.2223e-04 & 7.4972-06 & 4.6806e-07 & 2.9278e-08 & 2.7850e-09 \\ 
            & $\text{Rate}$ & - & 2.2790 & 2.0136 & 2.0008 & 1.9994 & 1.6970 \\ 
            \hline 
            \multirow{2}{*}{$\varepsilon_0/2^2$} & $e_{M,\varepsilon}^{\tau,h}$ & 8.0883e-02 & \textit{4.7937e-04} & 2.6215e-05 & 1.6257e-06 & 1.0184e-07 & 6.6343e-09 \\ 
            & $\text{Rate}$ & - & 3.6993 & 2.0963 & 2.0056 & 1.9984 & 1.9701 \\ 
            \hline 
            \multirow{2}{*}{$\varepsilon_0/2^3$} & $e_{M,\varepsilon}^{\tau,h}$ & 5.8848e-02 & 5.5010e-02 & \textit{1.4515e-04} & 9.0062e-06 & 5.6224e-07 & 3.5801e-08 \\ 
            & $\text{Rate}$ & - & 0.0487 & 4.2830 & 2.0052 & 2.0008 & 1.9866 \\ 
            \hline 
            \multirow{2}{*}{$\varepsilon_0/2^4$} & $e_{M,\varepsilon}^{\tau,h}$ & 1.5136e-02 & 1.4860e-02 & 1.4013e-02 & \textit{1.1820e-04} & 7.3985e-06 & 4.6162e-07 \\ 
            & $\text{Rate}$ & - & 0.0133 & 0.0423 & 3.4447 & 1.9989 & 2.0012 \\ 
            \hline 
            \multirow{2}{*}{$\varepsilon_0/2^5$} & $e_{M,\varepsilon}^{\tau,h}$ & 1.6533e-02 & 5.2087e-03 & 5.0550e-03 & 3.2985e-03 & \textit{1.1094e-04} & 6.9635e-06 \\ 
            & $\text{Rate}$ & - & 0.8332 & 0.0216 & 0.3080 & 2.4470 & 1.9969 \\ 
            \bottomrule[2pt]
       \end{tabular*}
        \label{tab7}
    \end{table}

    \begin{table}[H]
        \centering
        \caption{Temporal error analysis of $E$ solved by CISP with different $\varepsilon$. Here and in what follows, “*” means that the method is numerically unstable under the corresponding choice of $\varepsilon$ and $\tau$.}
        \begin{tabular*}{0.9\textwidth}[h]{@{\extracolsep{\fill}}c c c c c c c c c}
            \toprule[2pt]
            & & $\tau_0 = 0.2$ & $\tau_0/2^3$ & $\tau_0/2^6$ & $\tau_0/2^9$ & $\tau_0/2^{12}$ \\  
            \hline
             \multirow{2}{*}{$\varepsilon_0=1$} & $e_{E,\varepsilon}^{\tau,h}$ & 2.2250e-02 & 3.8839e-04 & 6.1241e-06 & 9.5739e-08 & 2.6319e-09  \\ 
            & $\text{Rate}$ & - & 1.9467 & 1.9956 & 1.9997 & 1.7283    \\ 
            \hline 
            \multirow{2}{*}{$\varepsilon_0/2$} & $e_{E,\varepsilon}^{\tau,h}$ & \textit{2.8618e-01} & 7.0958e-03 & 1.1231e-04 & 1.7566e-06 & 2.7495e-08   \\ 
            & $\text{Rate}$ & - & 1.7779 & 1.9938 & 1.9995 & 1.9992    \\ 
            \hline 
            \multirow{2}{*}{$\varepsilon_0/2^2$} & $e_{E,\varepsilon}^{\tau,h}$ & 1.5009e-00 & \textit{2.3124e-01} & 3.6975e-03 & 5.7757e-05 & 9.0209e-07 \\ 
            & $\text{Rate}$ & - & 0.8992 & 1.9889 & 2.0001 & 2.0002  \\ 
            \hline 
            \multirow{2}{*}{$\varepsilon_0/2^3$} & $e_{E,\varepsilon}^{\tau,h}$ & 6.2047e-00 & 5.4579e-01 & \textit{1.6995e-01} & 2.6003e-03 & 4.0609e-05  \\ 
            & $\text{Rate}$ & - & 1.1690 & 0.5611 & 2.0101 & 2.0003  \\ 
            \hline 
            \multirow{2}{*}{$\varepsilon_0/2^4$} & $e_{E,\varepsilon}^{\tau,h}$ & 2.5700e+01 & 2.9789e-00 & 7.3970e-01 & \textit{1.4971e-01} & 2.2510e-03   \\ 
            & $\text{Rate}$ & - & 1.0363 & 0.6699 & 0.7683 & 2.0185  \\ 
            \hline 
            \multirow{2}{*}{$\varepsilon_0/2^5$} & $e_{E,\varepsilon}^{\tau,h}$ & * & 1.8700e-00 & 1.6166e-00 & 3.8812e-01 & 1.1802e-01  \\ 
            & $\text{Rate}$ & - & * & 0.0700 & 0.6861 & 0.3668  \\ 
            \bottomrule[2pt]
        \end{tabular*}
        \label{tab8}
    \end{table}

   \begin{table}[H]
        \centering
        \caption{Temporal error analysis of $E$ solved by DISP with different $\varepsilon$}
        \begin{tabular*}{0.9\textwidth}[h]{@{\extracolsep{\fill}}c c c c c c c c c}
            \toprule[2pt]
            & & $\tau_0 = 0.2$ & $\tau_0/2^3$ & $\tau_0/2^6$ & $\tau_0/2^9$ & $\tau_0/2^{12}$ \\  
            \hline
             \multirow{2}{*}{$\varepsilon_0=1$} & $e_{E,\varepsilon}^{\tau,h}$ & 1.0823e-02 & 1.6960e-04 & 2.6500e-06 & 4.1345e-08 & 2.5757e-09  \\ 
            & $\text{Rate}$ & - & 1.9986 & 2.0000 & 2.0007 & 1.3349    \\ 
            \hline 
            \multirow{2}{*}{$\varepsilon_0/2$} & $e_{E,\varepsilon}^{\tau,h}$ & \textit{2.3028e-01} & 2.7611e-03 & 4.3160e-05 & 6.7439e-07 & 1.0606e-08   \\ 
            & $\text{Rate}$ & - & 2.1273 & 2.0000 & 2.0000 & 1.9969   \\ 
            \hline 
            \multirow{2}{*}{$\varepsilon_0/2^2$} & $e_{E,\varepsilon}^{\tau,h}$ & 1.7825e-00 & * & 3.0295e-03 & 4.7333e-05 & 7.3931e-07 \\ 
            & $\text{Rate}$ & - & * & * & 2.0000 & 2.0002  \\ 
            \hline 
            \multirow{2}{*}{$\varepsilon_0/2^3$} & $e_{E,\varepsilon}^{\tau,h}$ & 5.8139e-00 & * & \textit{1.6017e-01} & 2.4311e-03 & 3.7963e-05  \\ 
            & $\text{Rate}$ & - & * & * & 2.0139 & 2.0003  \\             
            \hline 
            \multirow{2}{*}{$\varepsilon_0/2^4$} & $e_{E,\varepsilon}^{\tau,h}$ & 2.9715e+01 & * & 7.3252e-01 & \textit{1.4722e-01} & 2.2097e-03   \\ 
            & $\text{Rate}$ & - & * & * & 0.7716 & 2.0193  \\ 
            \hline 
            \multirow{2}{*}{$\varepsilon_0/2^5$} & $e_{E,\varepsilon}^{\tau,h}$ & * & 1.4539e-00 & 1.4479e-00 & 3.6662e-01 & 1.8065e-01  \\ 
            & $\text{Rate}$ & - & * & 0.0020 & 0.6605 & 0.3404  \\ 
            \bottomrule[2pt]
                    \end{tabular*}
        \label{tab9}
    \end{table}

Next, we compare the efficiency of the three schemes with $\varepsilon=1$. From Figure~\ref{fig:eff-kgz}, we can observe that DISP is the most efficient one due to the explicit implementation, while CISP is the worst because of the fully implicit property. However, for very smaller $\varepsilon$, i.e., $0<\varepsilon\ll 1$, DISP is numerically unstable as depicted in Table~\ref{tab9}, and EPAVF-C is recommended then in terms of the stability, computational efficiency, and $\varepsilon$-scalability.

\begin{figure}[H]
	\centering
		\includegraphics[width=0.45\linewidth]{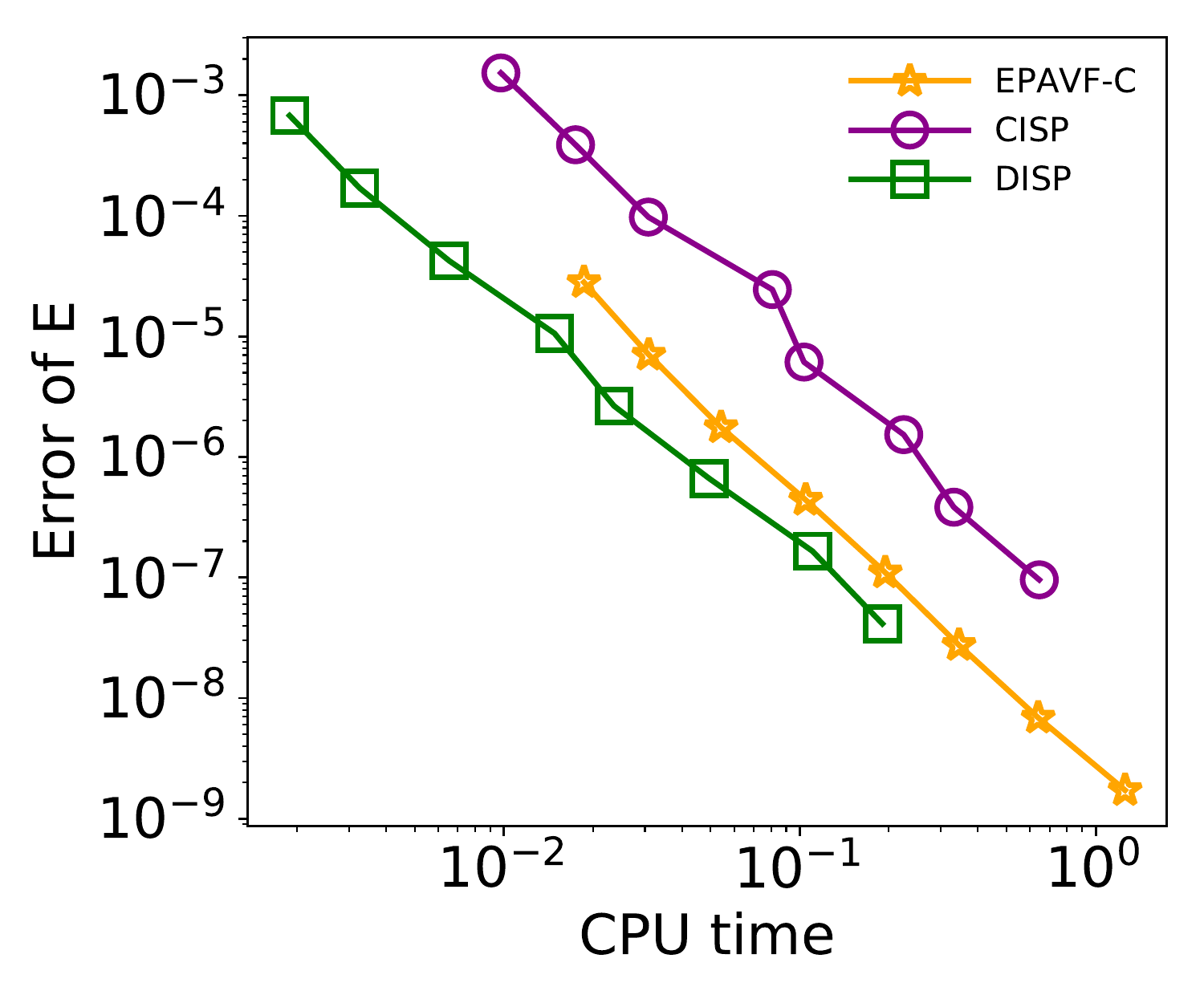}
		\includegraphics[width=0.45\linewidth]{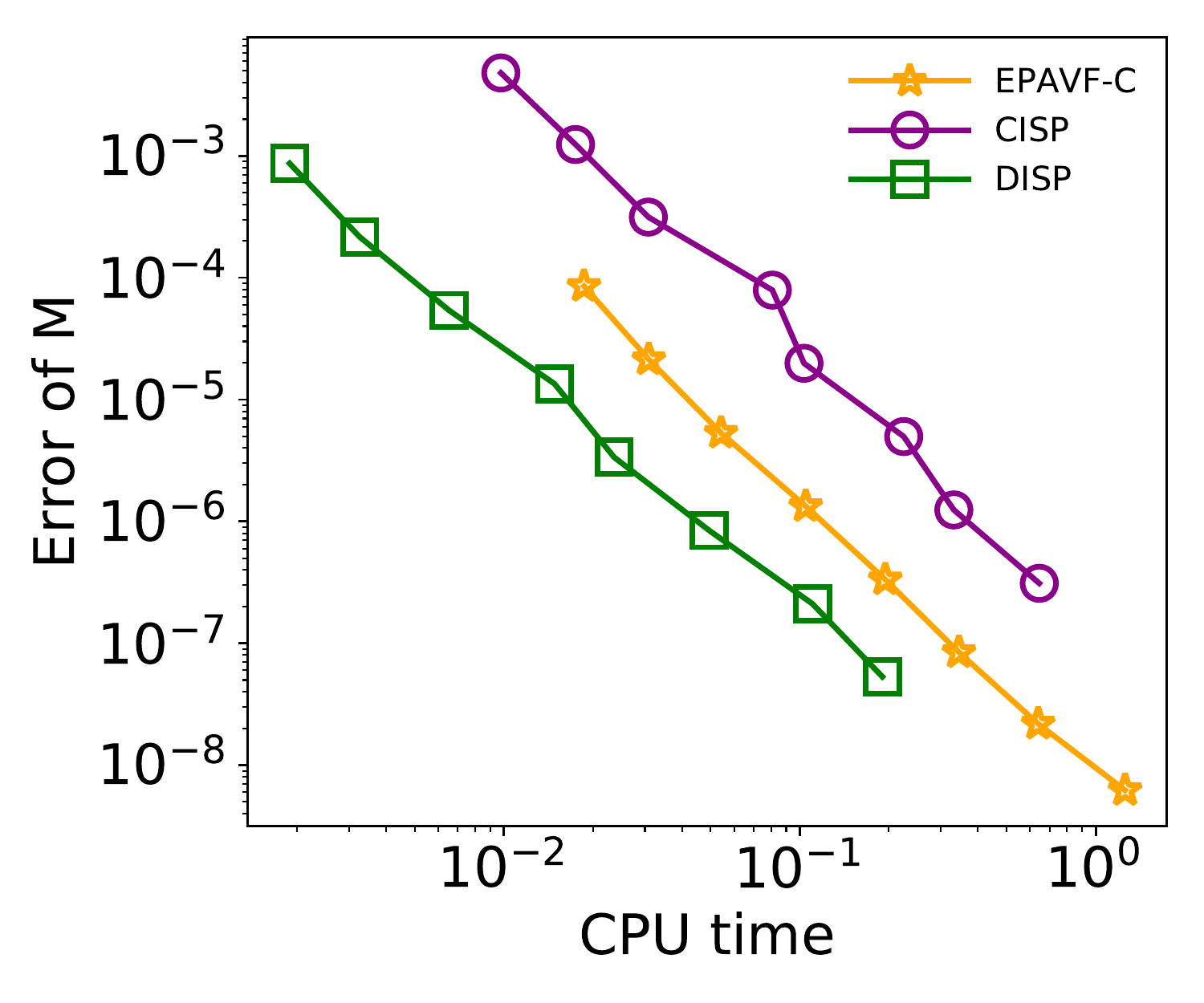}
	\caption{Computational efficiency of different schemes for the 1D KGZ equations with respect to $E$ (left) and $M$ (right) under the time step $\tau = 0.05 \times 2^{-k}, k =0 , \cdots, 7$ and $\varepsilon=1$.}
	\label{fig:eff-kgz}
\end{figure}

In Figure~\ref{accuracy epavf kgz}, the temporal errors of EPAVF and its adjoint for the KGZ equations are presented, which again verify the conclusion made in the experiment for the KGS equations. For long-term simulation, we set the termination time $T = 100$, and solve the KGZ equations by EPAVF-C with $\varepsilon = 1, 0.5, 0.25, 0.125$. The energy errors are presented in Figure~\ref{fig:ene-kgz1d}, which uniformly reach machine accuracy.

\begin{figure}[H]
	\centering
	\begin{minipage}[t!]{0.45\textwidth}
		\centering
		\includegraphics[width=1\linewidth]{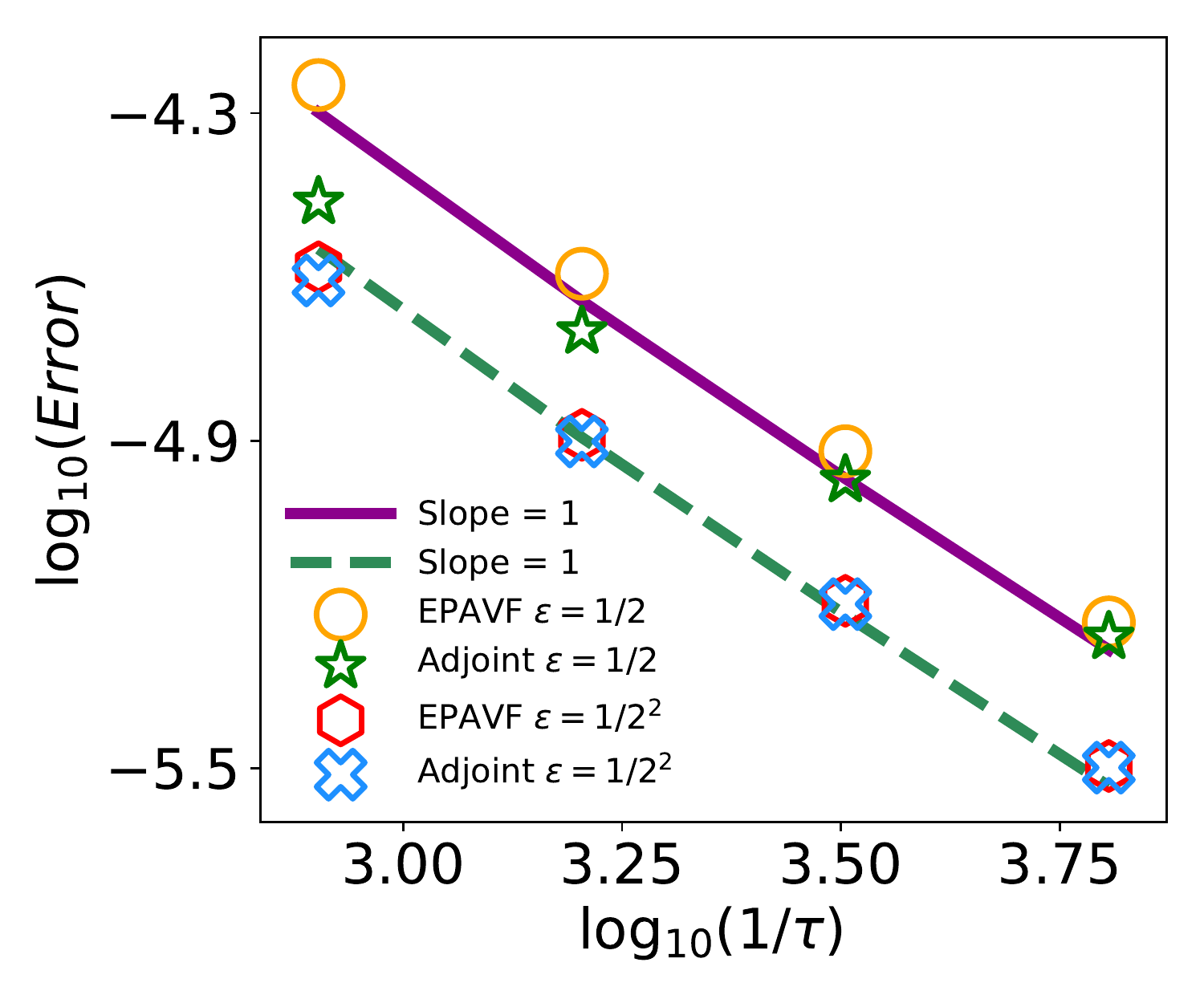}
	\end{minipage}
	\begin{minipage}[t!]{0.45\textwidth}
	\centering
	\includegraphics[width=1\linewidth]{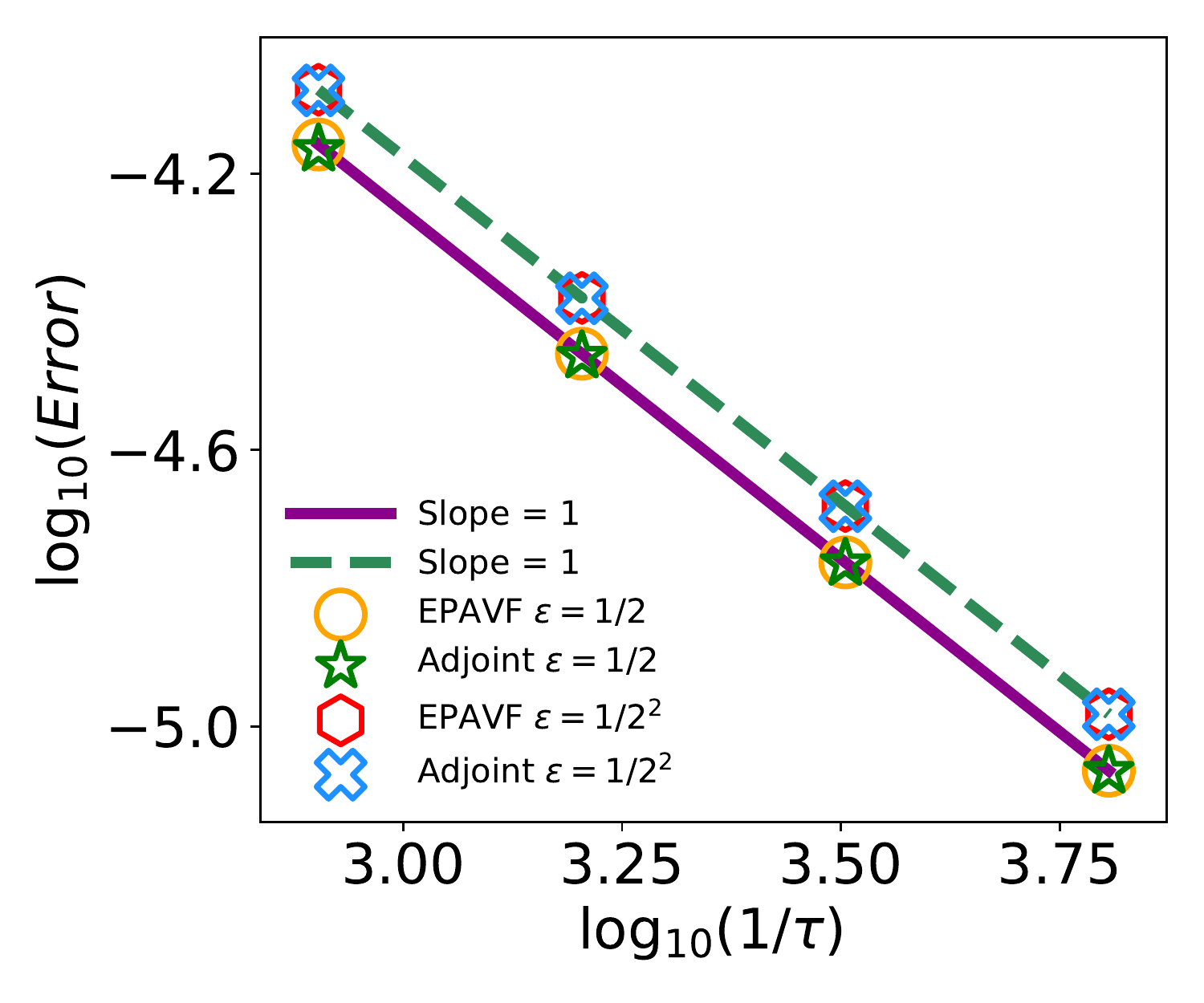}
\end{minipage}
	\caption{Temporal error of EPAVF and its adjoint for the 1D KGZ equations with respect to $E$ (left) and $M$ (right) under the time step $\tau = 0.00125\times 2^{-k}, k =0 , \cdots, 3$.}
\label{accuracy epavf kgz}
\end{figure}

\begin{figure}[H]
		\centering
		\includegraphics[width=0.9\textwidth]{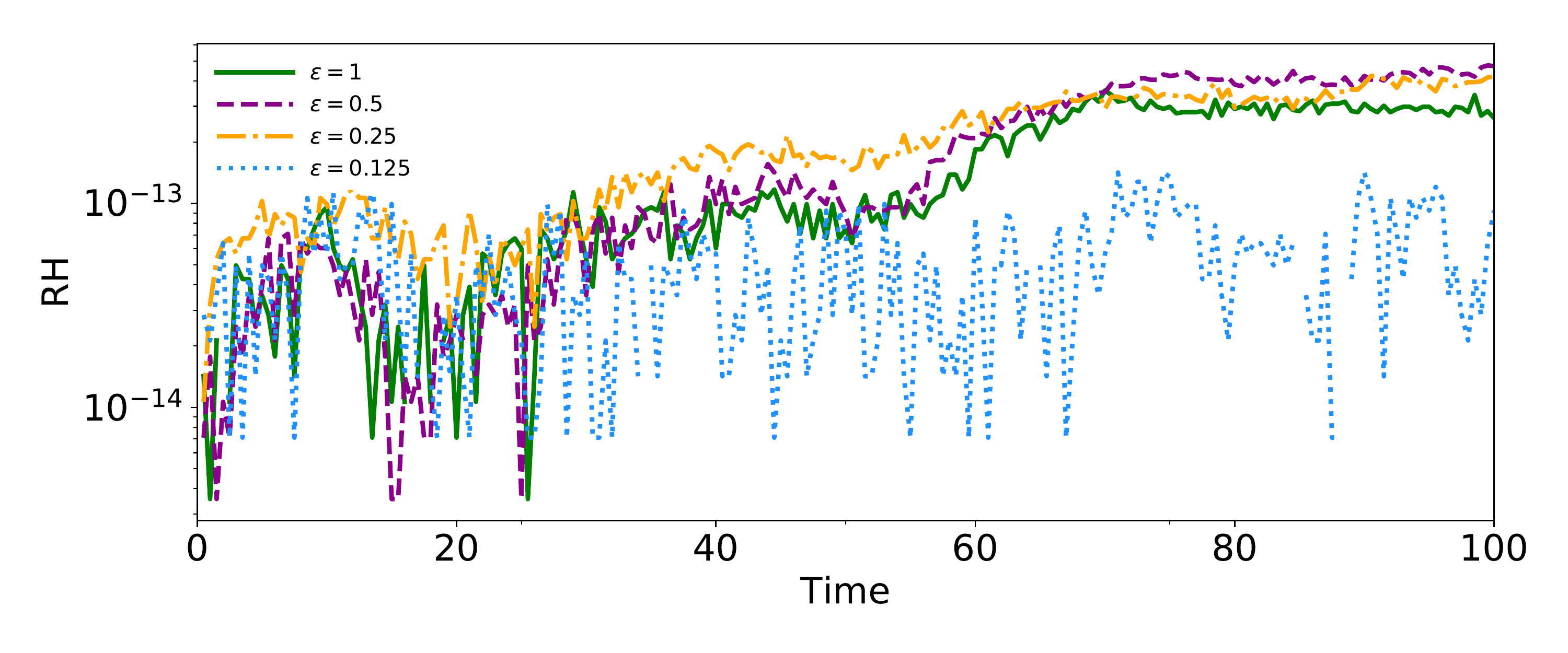}
	\caption{Energy error of the 1D KGZ equations solved by EPAVF--C with different $\varepsilon$.}
	\label{fig:ene-kgz1d}
\end{figure}

\begin{ex}
	In this example, we simulate the two-dimensional KGZ equations by EPAVF-C with initial conditions
	\begin{equation*}
		\begin{split}
			&E_{0}\left(x, y\right) = \exp{\left(-(x+2)^{2} - y^{2}\right)} + \exp{\left(-(x-2)^{2} - y^{2}\right)}, \ E_{1}\left(x, y\right) = \exp{\left(-x^{2} - y^{2}\right)}, \\ 
			&M_{0}\left(x, y\right) = \mbox{\rm sech}{\left(x^{2} + \left(y+2\right)^{2}\right)} + \mbox{\rm sech}{\left(x^{2} + \left(y-2\right)^{2}\right)}, \ M_{1}\left(x, y\right) = \mbox{\rm sech}{\left(x^{2} + y^{2}\right)}.
		\end{split}
	\end{equation*}
	and zeros boundary conditions on a rectangular domain $\Omega = \left[-32, 32\right]\times\left[-32, 32\right]$. We set $h_x = h_y = 1/4$ and $\tau = 0.1$. 
\end{ex}

Figure~\ref{Kgz 2D solution eh}-\ref{Kgz 2D solution mh} show snapshots of $E$ and $M$ by EPAVF-C with different $\varepsilon$ at $t = 0, 1, 2, 4$, respectively.  We can observe that the dynamic behavior of the 2D KGZ equations depends greatly on the parameter $\varepsilon$. Highly oscillatory waves appear in $E$ when $\varepsilon$ becomes smaller, and the EPAVF-C scheme can efficiently simulate these waves. Also, the related energies are preserved to round-off errors as expected.

\begin{figure}[H]
	\centering
	\begin{minipage}{0.23\textwidth}
		\includegraphics[width=1.0\linewidth]{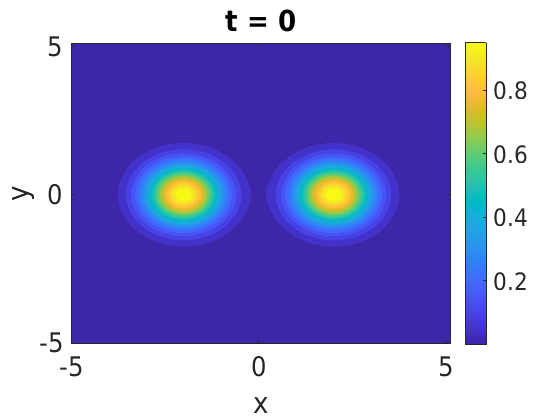}
	\end{minipage}
	\begin{minipage}{0.23\textwidth}
		\includegraphics[width=1.0\linewidth]{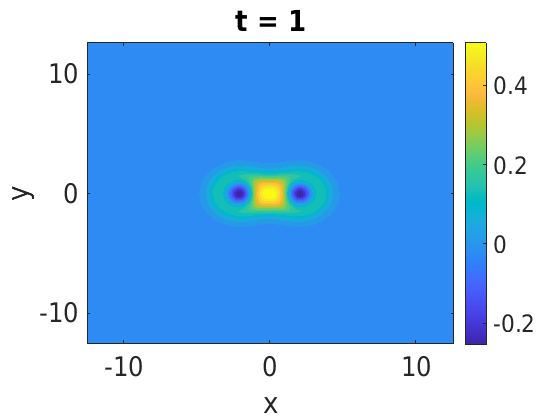}
	\end{minipage}
	\begin{minipage}{0.23\textwidth}
		\includegraphics[width=1.0\linewidth]{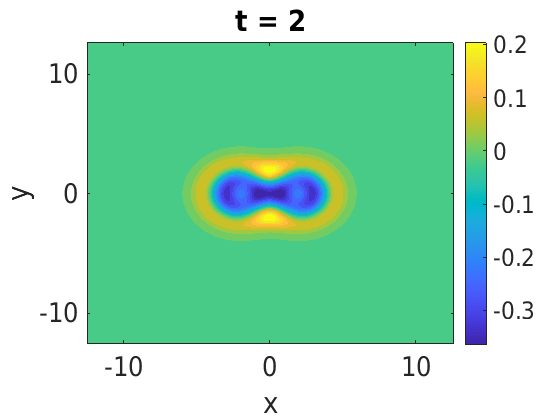}
	\end{minipage}
	\begin{minipage}{0.23\textwidth}
		\includegraphics[width=1.0\linewidth]{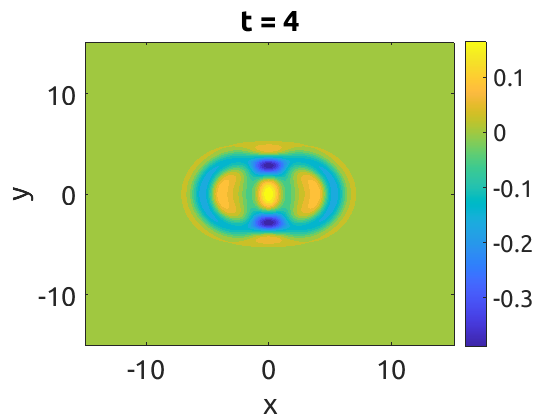}
	\end{minipage}
	\begin{minipage}{0.23\textwidth}
		\includegraphics[width=1.0\linewidth]{kgz2d_eh_t0.png}
	\end{minipage}
	\begin{minipage}{0.23\textwidth}
		\includegraphics[width=1.0\linewidth]{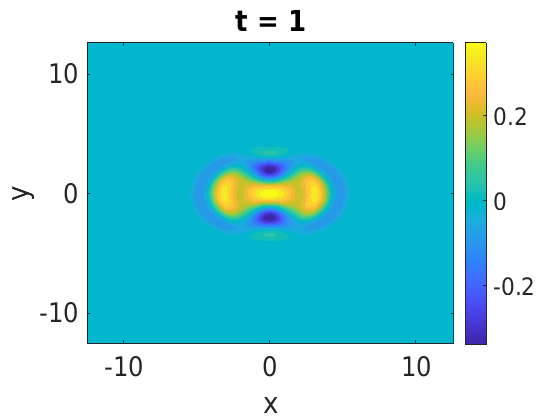}
	\end{minipage}
	\begin{minipage}{0.23\textwidth}
		\includegraphics[width=1.0\linewidth]{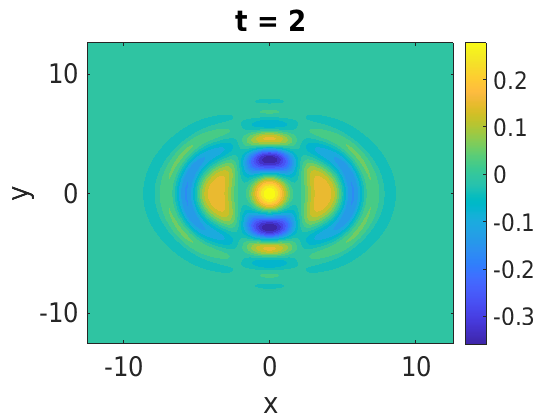}
	\end{minipage}
	\begin{minipage}{0.23\textwidth}
		\includegraphics[width=1.0\linewidth]{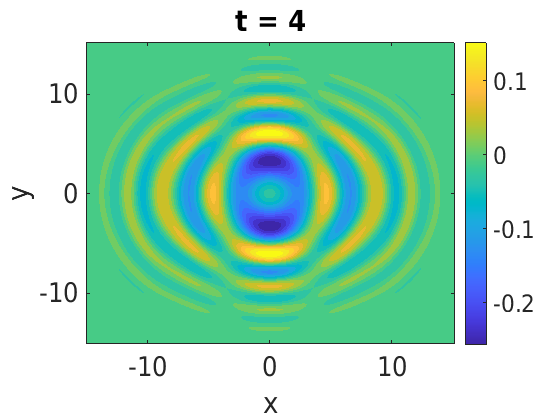}
	\end{minipage}
	\begin{minipage}{0.23\textwidth}
		\includegraphics[width=1.0\linewidth]{kgz2d_eh_t0.png}
	\end{minipage}
	\begin{minipage}{0.23\textwidth}
		\includegraphics[width=1.0\linewidth]{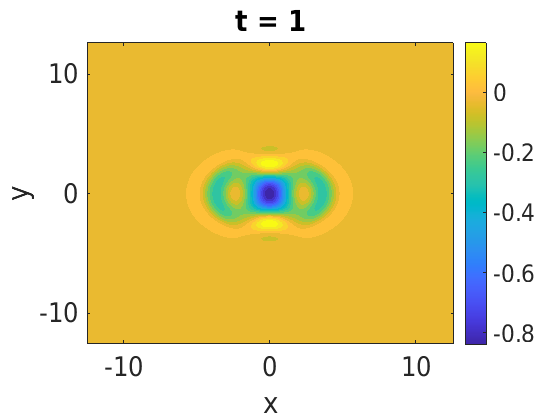}
	\end{minipage}
	\begin{minipage}{0.23\textwidth}
		\includegraphics[width=1.0\linewidth]{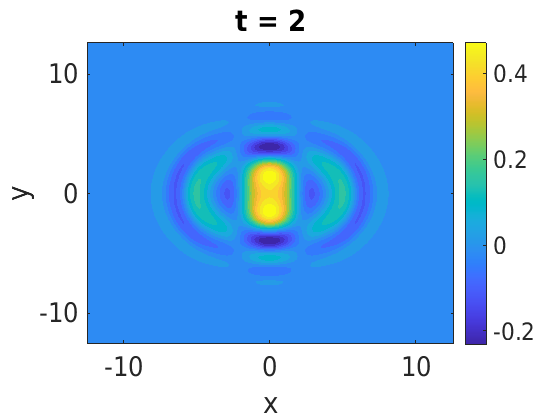}
	\end{minipage}
	\begin{minipage}{0.23\textwidth}
		\includegraphics[width=1.0\linewidth]{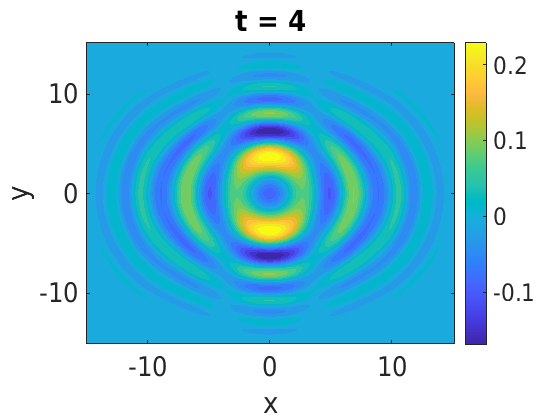}
	\end{minipage}
	\caption{Snapshots of $E$ for the 2D KGZ equations solved by EPAVF-C with $\varepsilon = 1$ (first row); $\varepsilon = 0.1$ (middle row);  $\varepsilon = 0.01$ (last row).}
	\label{Kgz 2D solution eh}
\end{figure}
\begin{figure}[H]
	\centering
	\begin{minipage}{0.23\textwidth}
		\includegraphics[width=1.0\linewidth]{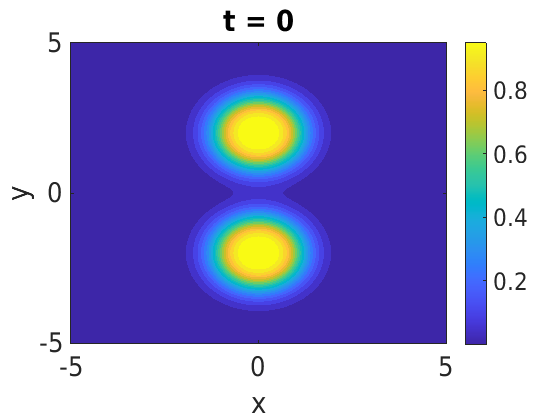}
	\end{minipage}
	\begin{minipage}{0.23\textwidth}
		\includegraphics[width=1.0\linewidth]{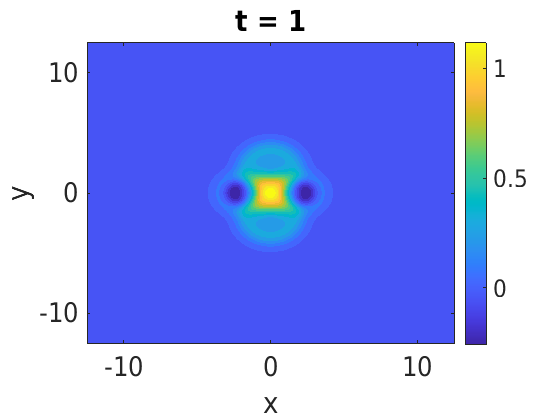}
	\end{minipage}
	\begin{minipage}{0.23\textwidth}
		\includegraphics[width=1.0\linewidth]{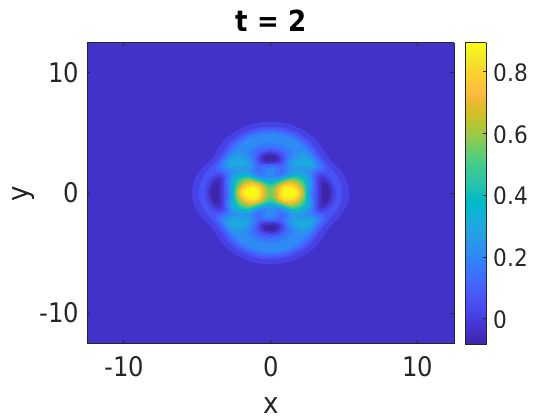}
	\end{minipage}
	\begin{minipage}{0.23\textwidth}
		\includegraphics[width=1.0\linewidth]{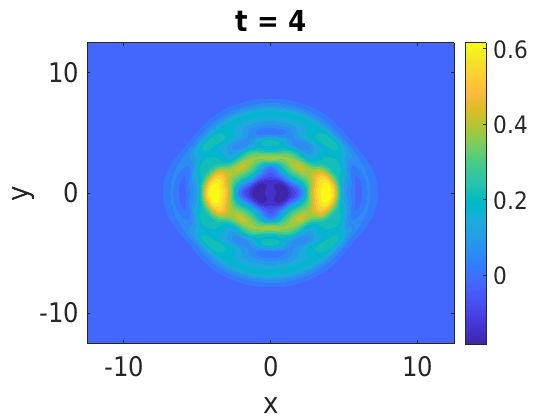}
	\end{minipage}
	\begin{minipage}{0.23\textwidth}
		\includegraphics[width=1.0\linewidth]{kgz2d_mh_t0.png}
	\end{minipage}
	\begin{minipage}{0.23\textwidth}
		\includegraphics[width=1.0\linewidth]{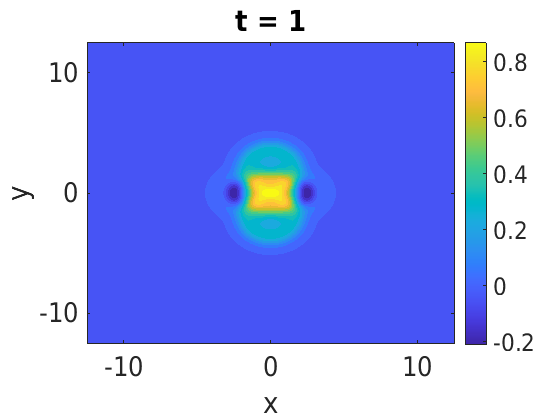}
	\end{minipage}
	\begin{minipage}{0.23\textwidth}
		\includegraphics[width=1.0\linewidth]{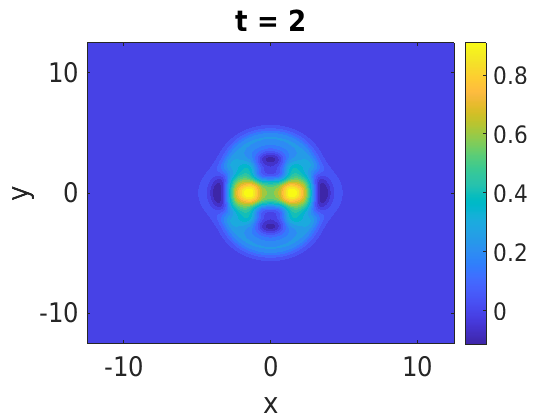}
	\end{minipage}
	\begin{minipage}{0.23\textwidth}
		\includegraphics[width=1.0\linewidth]{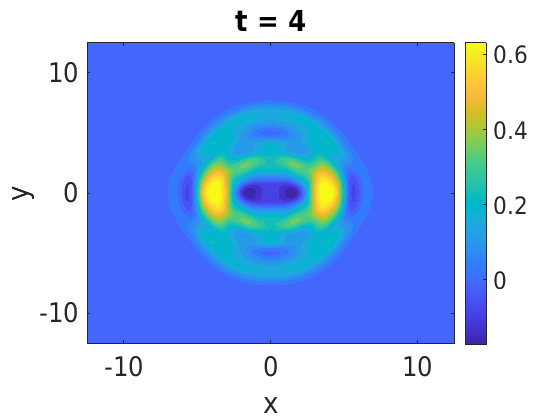}
	\end{minipage}
	\begin{minipage}{0.23\textwidth}
		\includegraphics[width=1.0\linewidth]{kgz2d_mh_t0.png}
	\end{minipage}
	\begin{minipage}{0.23\textwidth}
		\includegraphics[width=1.0\linewidth]{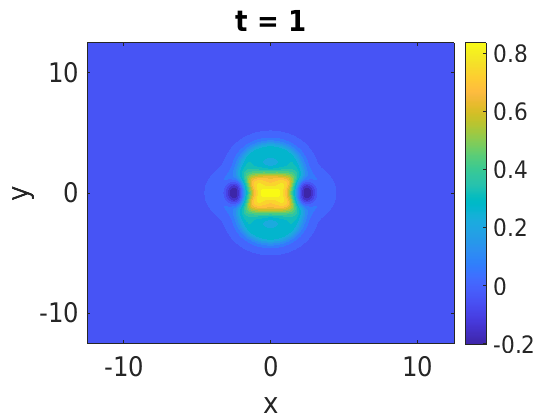}
	\end{minipage}
	\begin{minipage}{0.23\textwidth}
		\includegraphics[width=1.0\linewidth]{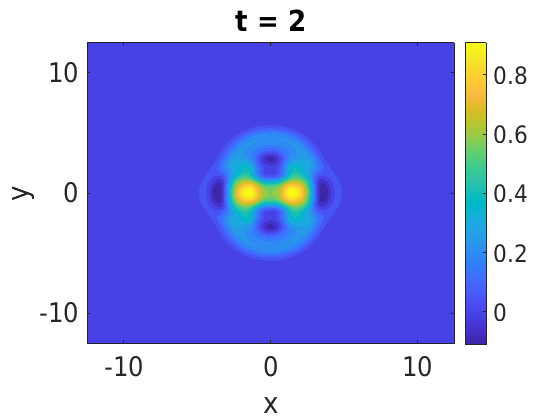}
	\end{minipage}
	\begin{minipage}{0.23\textwidth}
		\includegraphics[width=1.0\linewidth]{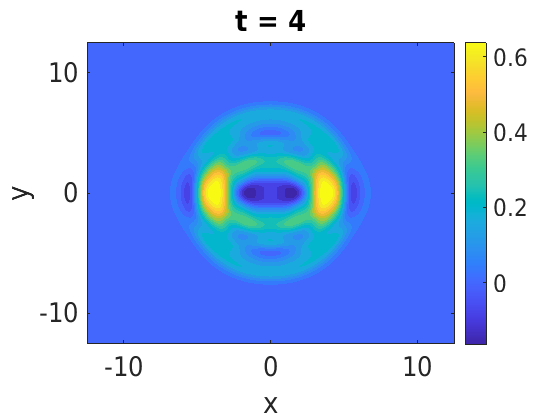}
	\end{minipage}
	\caption{Snapshots of $M$ for the 2D KGZ equations solved by EPAVF-C with $\varepsilon = 1$ (first row); $\varepsilon = 0.1$ (middle row);  $\varepsilon = 0.01$ (last row).}
	\label{Kgz 2D solution mh}
\end{figure}
\begin{figure}[H]
	\centering
	\includegraphics[width=0.9\textwidth]{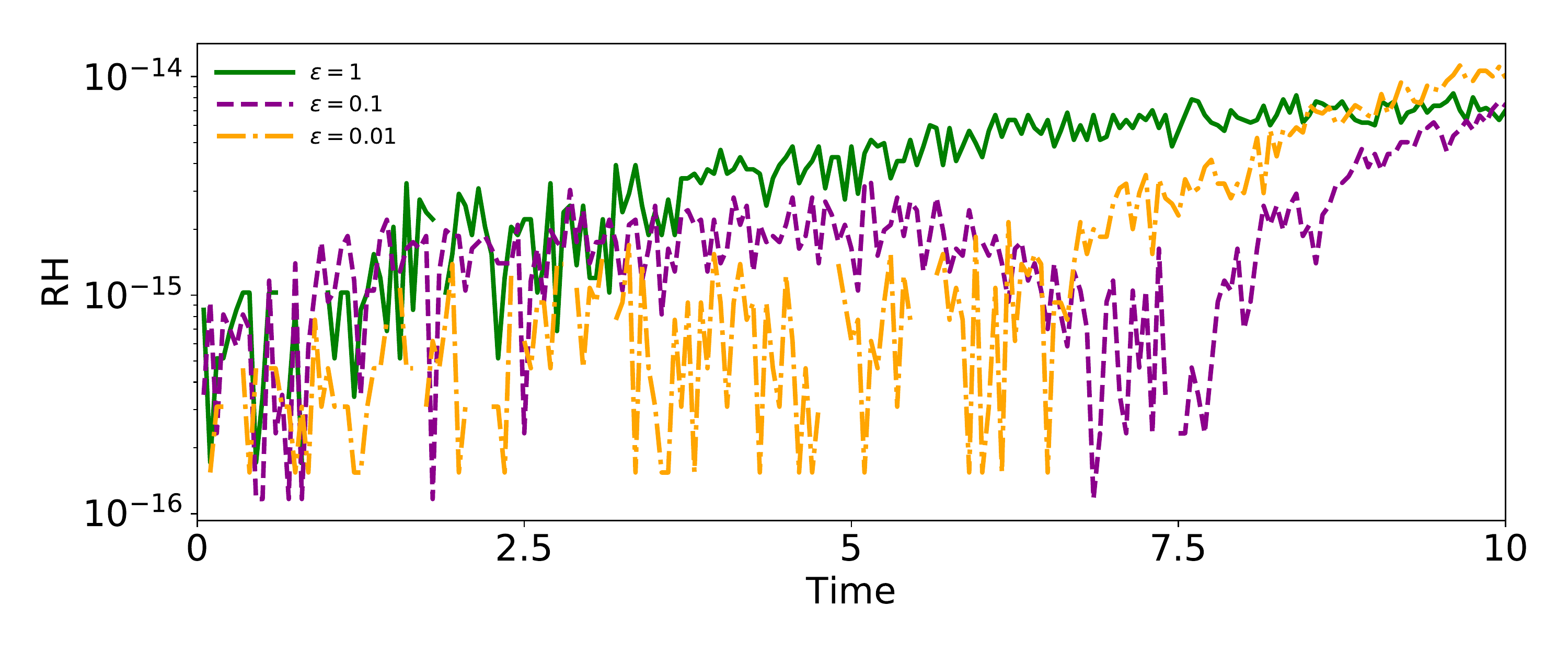}
	\caption{Energy error of the 2D KGZ equations solved by EPAVF-C with different $\varepsilon$.}
	\label{Kgz 2D energy}
\end{figure}

\section{Conclusion}
Based on the exponential integrator and the partitioned averaged vector field method, we develop a new class of energy-preserving exponential integrators for multi-component Hamiltonian systems, which is highly efficient that every subsystem can be solved one by one. For specific problems like the KGS and KGZ equations considered in this paper, such methods may exhibit more superior behavior. That is, one part of the resulting schemes is explicitly solved while the other is linearly implicit. Compared with most of existing linearly implicit schemes such as IEQ, SAV, approaches and other multistep methods, the proposed scheme can preserve the original energy. A rigorous proof of energy conservation is given for general multi-component Hamiltonian systems without any additional assumptions, which can be viewed as an improvement and generalization of the existing ones. Ample numerical experiments are carried out to demonstrate the significant advantages of the proposed methods in accuracy, computational efficiency, and the ability in capturing highly oscillatory solutions.

	Notice that in the simulation of the limit regime, i.e., $0<\varepsilon\ll 1$ in the KGS and KGZ equations, the $\varepsilon$-scalability or meshing strategy for the EPAVF method is $\tau=\mathcal{O}(\varepsilon^2)$, which can be further improved by utilizing the multiscale time integrators proposed in \cite{bao2016,bao2017}. Nevertheless, how to maintain energy conservation in the meantime needs to be studied.

\section{Acknowledgements}
This work is supported by the National Key Research and Development Project of China (2018YFC1504205),  the National Natural Science Foundation of China (12171245, 11971242), the Yunnan Fundamental Research Projects (Nos. 202101AT070208, 202101AS070044), and the Science and Technology Innovation Team on Applied Mathematics in Universities of Yunnan.

\appendix

\section{Derivation of the EPAVF scheme for 2D KGS equations}\label{appendix}

Let the computation domain $\Omega=[a,b]\times[c,d]$ and $N_x$, $N_y$ be given even integers. The spatial steps are then defined as $h_x=(b-a)/N_x$, $h_y=(d-c)/N_y$, and the mesh grid is denoted by $\Omega_h=\{(x_i,y_j)|x_i=a+ih_x, y_j=c+jh_y,i=0,\cdots,N_x, j=0,\cdots,N_y\}$. Let $V_h=\{V|V=(V_{ij}),(x_i,y_j)\in\Omega_h\}$ be the space of grid function on $\Omega_h$. Applying the Fourier pseudospectral method to the 2D KGS equations and after some arrangements, we have
\begin{equation}\label{kgs-semi2d}
	\begin{cases}
		{Q}_t + \beta (\mathbb{D}_2^x {P}+{P}\mathbb{D}_2^y) +  {P} \odot{U} = 0,  \\
		{P}_t - \beta (\mathbb{D}_2^x  {Q}+{Q}\mathbb{D}_2^y)-  {Q}  \odot{U} = 0,  \\
		{U}_t =  {V},\\
		\varepsilon^{2}{V}_t - (\mathbb{D}_2^x {U} +{U}\mathbb{D}_2^y ) + \dfrac{1}{\varepsilon^{2}}  {U} -  {Q}^{2} -  {P}^{2} = 0,
	\end{cases}
\end{equation}
where $\mathbb{D}_2^\alpha, \alpha=x,y$ are the second-order spectral differential matrices in $x,y$ directions, respectively. For convenience of the derivation, we reshape $U$ by columns into a vector \textbf{u} of dimension $N_x\times N_y$, \textit{etc.} Then the above system \eqref{kgs-semi2d} can be reformed as
\begin{equation}\label{kgs-semi-vec}
	\begin{cases}
		\textbf{q}_t + \beta \mathbb{D} \textbf{p} +  \textbf{p} \odot\textbf{u} = 0,  \\
		\textbf{p}_t - \beta \mathbb{D}  \textbf{q}-  \textbf{q}  \odot\textbf{u} = 0,  \\
		\textbf{u}_t =  \textbf{V},\\
		\varepsilon^{2}{v}_t - \mathbb{D} \textbf{U} + \dfrac{1}{\varepsilon^{2}}  \textbf{u} -  \textbf{q}^{2} -  \textbf{p}^{2} = 0,
	\end{cases}
\end{equation}
where $\mathbb{D} = I_{N_y}\otimes \mathbb{D}_2^x + \mathbb{D}_2^y \otimes I_{N_x}$. The only differences of the semi-discretization \eqref{KGS-semi} and \eqref{kgs-semi-vec} appear in the dimension of variables and the differential matrix $\mathbb{D}$. Therefore, the corresponding EPAVF scheme can be similarly obtained as \eqref{scheme-kgs}, where the matrix exponentials are calculated by
\[
\begin{aligned}
		\exp(V_1) = 
	\left(\begin{array}{cc}
		\cos(\tau \beta \mathbb{D}) & -\sin(\tau \beta \mathbb{D})\\ 
		\sin(\tau \beta \mathbb{D}) & \cos(\tau \beta\mathbb{D})
	\end{array}\right), \quad
	\exp(V_2)= 
	\left(\begin{array}{cc}
		\cos(\tau \widetilde{\mathbb{D}})& \frac{\sin( \tau\widetilde{\mathbb{D}}) }{ \widetilde{\mathbb{D}}}	\\ 
		-\widetilde{\mathbb{D}}\sin( \tau\widetilde{\mathbb{D}}) & \cos(\tau \widetilde{\mathbb{D}})
	\end{array} \right), \\ 
	\end{aligned}
\]
and
\[
\begin{aligned}
		\varphi(V_1) = 
	\left(\begin{array}{cc}
		\frac{\sin(\tau \beta \mathbb{D})}{\tau\beta  \mathbb{D}} & \frac{I_{N_x\times N_y}  - \cos(\tau \beta  \mathbb{D})}{\tau\beta  \mathbb{D}} \\ 
		-\frac{I_{N_x\times N_y} - \cos(\tau \beta  \mathbb{D})}{\tau\beta  \mathbb{D}} & \frac{\sin(\tau \beta  \mathbb{D})}{\tau\beta  \mathbb{D}}
	\end{array}\right), \ 
	\varphi(V_2) = 
	\left(\begin{array}{cc}
		\frac{\sin( \tau \widetilde{\mathbb{D}} ) } {\tau\widetilde{\mathbb{D}}} & \frac{   I_{N_x\times N_y}  - \cos(\tau \widetilde{\mathbb{D}}) }{\tau\widetilde{\mathbb{D}}^{2}}  \\
		\frac{\cos( \tau \widetilde{\mathbb{D}} ) - I_{N_x\times N_y} }{\tau} & \frac{\sin( \tau \widetilde{\mathbb{D}} )  }{\tau\widetilde{\mathbb{D}}}
	\end{array}\right).
\end{aligned}
\]
Notice that the differential matrix $\mathbb{D}$ can be decomposed as
\[
	\mathbb{D}= (F_{N_y}^{-1} \otimes F_{N_x}^{-1})(I_{N_y}\otimes \Lambda_2^x+\Lambda_2^y \otimes I_{N_x} )(F_{N_y} \otimes F_{N_x}),
\]
so that the calculation of trigonometric functions of matrix $\mathbb{D}$ can also be accelerated by FFT. For example, we have
\[
\cos{(\tau \beta \mathbb{D})} = (F_{N_y}^{-1} \otimes F_{N_x}^{-1})\cos( \tau\beta ( I_{N_y}\otimes \Lambda_2^x + \Lambda_2^y \otimes I_{N_x}) ) (F_{N_y} \otimes F_{N_x}).
\]
Moreover, due to the fact that $(B^\top\otimes A){\rm vec}(X)={\rm vec}(AXB)$ for any suitable matrices $A, B, X$,  where the function `${\rm vec}(\cdot)$' represents the vectorization of a matrix, we have
\[
\cos{(\tau \beta \mathbb{D})} \mathbf{u}
={\rm vec}\Big( F_{N_x}^{-1}( \cos{(\tau\beta\Lambda)}\odot\widetilde{{U}} ) F_{N_y}^{-\top} \Big),
 \]
where $\Lambda_{ij} = (\Lambda_2^x)_{ii} + (\Lambda_2^y)_{jj}$, $\widetilde{{U}} = F_{N_x} {U} F_{N_y}^\top$ and $(\cos{(\tau\beta \Lambda)})_{ij} = \cos{(\tau\beta \Lambda_{ij})}$. As a consequence, in practical computation we implement the EPAVF scheme \eqref{scheme-kgs} for 2D KGS equations just in the following matrix form
\begin{equation}\label{scheme-kgsfft}
	\begin{cases}
		\widetilde{Q}^{n+1} = \exp_{11}^1\odot \widetilde{Q}^{n} + \exp_{12}^1\odot \widetilde{{P}}^{n} +  \frac{\tau}{2} \varphi_{12}^1 \odot\widetilde{G}_1 - \frac{\tau}{2} \varphi_{11}^1 \odot\widetilde{G}_2,  \\ 
		\widetilde{P}^{n+1} =  \exp_{21}^1\odot \widetilde{Q}^{n} + \exp_{22}^1\odot \widetilde{P}^{n} +  \frac{\tau}{2} \varphi^1_{22} \odot\widetilde{G}_1 - \frac{\tau}{2} \varphi^1_{21} \odot\widetilde{G}_2, \\
		\widetilde{U}^{n+1} =\exp_{11}^2\odot \widetilde{U}^{n} + \exp_{12}^2\odot \widetilde{V}^{n} + \frac{\tau}{\varepsilon^{2}}\varphi_{12}^2 \odot\widetilde{G}_3, \\ 
		\widetilde{V}^{n+1} = \exp_{21}^2\odot \widetilde{U}^{n} + \exp_{22}^2\odot \widetilde{V}^{n} + \frac{\tau}{\varepsilon^{2}}\varphi_{22}^2 \odot\widetilde{G}_3,
	\end{cases}
\end{equation}
 where ${G}_1=U^{n} \odot Q^{n+1/2}, {G}_2=U^{n} \odot P^{n+1/2}, {G}_3=\left(Q^{n+1}\right)^2 + \left(P^{n+1}\right)^2$. The components of matrix exponentials, i.e., $\exp_{ij}^k$, $\varphi_{ij}^k$ are obtained by replacing $\mathbb{D}$ with $\Lambda$, and the corresponding trigonometric functions are computed element-by-element.
 
 Although we have only derived the practical EPAVF scheme \eqref{scheme-kgsfft} for the 2D KGS equations here, its adjoint scheme as well the EPAVF schemes for 2D KGZ equations can be obtained similarly. Furthermore, the procedures of derivation can be directly generalized to 3D cases with few changes.

\bibliographystyle{abbrv}
\bibliography{reference.bib}

\end{document}